\documentclass[a4paper,10pt]{article}
\usepackage[english]{babel}
\usepackage[latin1]{inputenc}
\usepackage[T1]{fontenc}
\usepackage{amsmath}
\usepackage{amsthm}
\usepackage{latexsym}
\usepackage{amssymb}
\usepackage{mathrsfs}
\usepackage{epic}
\usepackage{epsfig}
\usepackage{color}
\usepackage{dsfont}
\usepackage[all]{xy}
\usepackage[nottoc, notlof, notlot]{tocbibind}
\usepackage{eucal}
\usepackage{calc}
\usepackage{multicol}

\newcommand{\R}{\mathbb{R}}

\newcommand{\N}{\mathbb{N}}

\newcommand{\ud}{\text{d}}

\newcommand{\scri}{{\mathscr I}}
\newcommand{\scal}{\textrm{Scal}}
\newcommand{\sop}{{\Sigma_0^{u_0>}}}
\newcommand{\som}{{\Sigma_0^{u_0<}}}
\numberwithin{equation}{section}
\newtheorem{theorem}{Theorem}[section]  
\newtheorem{remark}[theorem]{Remark}  
\newtheorem{lemma}[theorem]{Lemma}

\newtheorem{definition}[theorem]{Definition}  
\newtheorem{proposition}[theorem]{Proposition}

\begin{document}
\title{Conformal scattering for a nonlinear wave equation\\
 on a curved background}
\author{J\'er\'emie Joudioux\footnote{Currently: Max-Planck-Institut f\"ur Gravitationphysik - Albert Einstein Institute; joudioux@aei.mpg.de}\\
Laboratoire de Math\'ematiques de Brest, CNRS  U.M.R. 6205\\ 
Universit\'e de Bretagne Occidentale, Universit\'e europ\'eenne de Bretagne,\\
6, avenue Victor Le Gorgeu, CS 93837, F-29238 BREST Cedex 3}
\maketitle

\begin{quote}
The purpose of this paper is to establish a geometric scattering result for a conformally invariant nonlinear wave equation on an asymptotically simple spacetime. The scattering operator is obtained via trace operators at null infinities. The proof is achieved in three steps. A priori linear estimates are obtained via an adaptation of the Morawetz vector field in the Schwarzschild spacetime and a method used by H\"ormander for the Goursat problem. A well-posedness result for the characteristic Cauchy problem on a light cone at infinity is then obtained. This requires a control of the nonlinearity uniform in time which comes from an estimates of the Sobolev constant and a decay assumption on the nonlinearity of the equation. Finally, the trace operators on conformal infinities are built and used to define the conformal scattering operator. 
  \end{quote}

\section*{Introduction}

Scattering theory is one of the most precise tools to analyze and describe the asymptotic behavior of solutions of evolution equations. As a consequence, it has a great importance in relativity to understand the influence of the geometry on the propagation of waves. Scattering in relativity was developed by many authors: Dimock (\cite{d85}), Dimock and Kay (\cite{MR857397,MR887979}) and more recently, Bachelot (\cite{MR1311537,MR1994882,MR1468518,MR1671210}), Bachelot and Bachelot-Motet (\cite{MR1244181}), H\"afner (\cite{MR2031494, ha06}), H\"afner-Nicolas (\cite{hn04}), Melnyk (\cite{MR2016993}) and Daud\'e \cite{MR2134232,MR2134850}.

The scattering method used by these authors relies on spectral theory: this requires the metric to be independent of a certain time function. It has consequently been necessary to develop a scattering theory which is not time dependent. As remarked by Friedlander in \cite{MR583989}, it is possible to use a conformal rescaling to study an asymptotic behavior. This method was used for the first time by Baez-Segal-Zhou in \cite{MR1073286} for the wave equation on the Minkowski spacetime. Their method consisted in embedding conformally the Minkowski spacetime in a bigger compact manifold.

The conformal compactification of a spacetime was first introduced by Penrose in the sixties in \cite{MR0175590} to describe the asymptotic behavior of solutions of the Dirac equations. A boundary, which represents in some way the infinity for causal curves, is added to the manifold. This boundary is divided into two connex components $\scri^+$ and $\scri^-$. When the spacetime satisfies Einstein equations (with no cosmological constant), $\scri^+$ and $\scri^-$ are light cones from two singularities $i^+$ and $i^-$. The asymptotic behavior can then be obtained by considering the traces on these hypersurfaces of a solution of the conformal wave equation (more precisely the scattering operator is obtained from the trace operators on $\scri^+$ and $\scri^-$). Asymptotically simple curved spacetimes, that is to say spacetimes admitting a conformal compactification, with specifiable regularity at $i^+$ and $i^-$,  were constructed by Chrusciel-Delay (\cite{MR1920322,MR1902228}), Corvino (\cite{MR1794269}) and Corvino-Schoen (\cite{MR2225517}). Mason and Nicolas successfully adapted the method of Baez-Segal-Zhou in the linear case for the Dirac and Maxwell equation  in \cite{mn04} on this curved background. They also obtained a complete peeling result for the wave equation on a Schwarzschild background in \cite{mn07}.

This paper presents the construction of a conformal scattering operator for the conformally invariant defocusing cubic wave equation:
$$
\nabla_a \nabla^a \Phi  + b  \Phi^3=0 
$$ 
on the asymptotically simple spacetime obtained by Chrusciel-Delay and Corvino-Schoen. Our construction relies on vector fields methods which were previously used to obtain the well-posedness of the Cauchy problem equation (see the result of Cagnac--Choquet-Bruhat in \cite{MR789558}) to obtain energy estimates. The choice which is made here is for the vector field is the same as the one made in \cite{mn07} and the techniques are essentially the same as in \cite{mn04}.

A specific method is used to handle the singularity in $i^+$: the characteristic hypersurface is described as the graph of a function. This method was introduced by H\"ormander in  \cite{MR1073287} and generalized in \cite{MR2334072,MR2244222} by Nicolas to establish the well-posedness of the characteristic Cauchy problem.

The main obstacle and difference with the linear case are the necessity to obtain uniform estimates of the non linearity. This requires to obtain Sobolev embeddings from $H^1$ into $L^6$ over a spacelike foliation transverse to a light cone up to the vertex of this cone. This is achieved by considering results concerning the norm of this embeddings given by M\"uller zum Hagen and Dossa (see \cite{MR633558, MR0606056} and \cite{MR2015759}): studying the characteristic Cauchy problem for nonlinear equations, they noticed that the norm of this embedding is not bounded when the foliation reaches the vertex of the light cone and gave a precise estimate of the speed of the blow up. Their result is adapted here to the conformal infinity. The blow up of the norm of the Sobolev embedding is compensated by a decay assumption on the function $b$.

The paper is organized as follows:
\begin{itemize}
\item the first section introduces the geometrical and analytical background: the spacetimes obtained by Corvino-Schoen and Chrusciel-Delay are precisely defined and the function space on the characteristic hypersurface at infinity is given. 
\item The a priori estimates are derived in section 2: these estimates are established in three specific subsets of $M$: a neighborhood of $i^0$ where the estimates come from the asymptotic behavior of the chosen vector field (the Morawetz vector field), a neighborhood of $i^+$ where the estimates are established by following the method developed by H\"ormander, and finally in a neighborhood of a Cauchy hypersurface. The techniques consist essentially in the use of Gronwall lemma and Stokes theorem.
\item Section 3 is devoted to the well-posedness of the Cauchy problem. The proof is made as follows: estimates on the propagator of the cubic wave equation are established from a control of the norm of Sobolev embeddings, in the spirit of the work of M\"uller zum Hagen and Dossa. The second step consists in establishing a local existence theorem for the characteristic Cauchy problem up to a spacelike hypersurface close to the conformal infinity. This is done by slowing down the propagation of waves, reducing the characteristic Cauchy problem to a Cauchy problem. A complete solution, up to a reference Cauchy surface, is then obtained by gluing the local solution and a solution of the Cauchy problem on the well chosen hypersurface close to conformal infinity.
\item The appendix contains, on the one hand, a certain number of calculations useful for the a priori estimates and, on the other hand, a proof of the characteristic Cauchy problem for small data; this proof is considered as interesting because of its simplicity. 
\end{itemize}

\tableofcontents

\section*{Conventions and notations}

Let $(M,g)$ be a 4 dimensional manifold of Chrusciel-Delay and Corvino-Schoen type. Its compactification is denoted by $(\hat M, \hat g)$. The associated connections are denoted $\nabla$ and $\hat \nabla$.

Let us consider on $M$ the following nonlinear wave equation:
\begin{equation*}
\nabla_a \nabla^a \Phi  + b  \Phi^3=0 
\end{equation*}
We assume that:
\begin{enumerate}
\item\label{A1} $b$ is positive;
\item\label{A2} $b$ admits a continuous extension to  $\hat M$ such that $b$ vanishes at $\scri$;
\item\label{A3} $b$ satisfies: there exists a constant $c$ such that, uniformly on $\hat M$:
$$
\exists c, |\hat T^a\nabla_a b|\leq c b.
$$
\item\label{A4} $b$ satisfies the following decay assumption: there exist a uniformly spacelike foliation $(\Sigma_r)_{r\in [0,R]}$ of a neighborhood $i^+$ and a constant $c$, which does not depend on $r$
such that
$$
\forall r \in (0,R], \frac{||b||_{L^\infty(\Sigma_r)}}{r^3}\leq c.
$$
The same decay assumption holds in $i^-$. 
\end{enumerate}
These assumptions will be referred to as assumptions \ref{A1}, \ref{A2}, \ref{A3} and \ref{A4}.
\begin{remark}\nonumber
\begin{enumerate}
\item The positivity of $b$ corresponds to the defocusing case.
\item Since $\hat M$ is compact, the differential inequality satisfied by $b$ does in fact not impose another specific asymptotic behaviour than the fact that $\hat T_a\nabla^a \phi$ decrease and vanishes at the same rate of $b$. This assumption, essentially technical, is required to insure the global well-posedness of the Cauchy problem accordingly to the Cagnac--Choquet-Bruhat theorem (see theorem \ref{choquetbruhat}).
\item The purpose of the decay assumption on $b$ is to compensate the blow-up of the Sobolev constant at the vertex of the light cones $\scri^\pm$. 
\end{enumerate}
\end{remark}

The following notations will be used:
\begin{itemize}
\item we will note:
$$
\phi \lesssim \psi
$$
where $\phi$ and $\psi$ are two functions over $U$, a subset of $M$, whenever there exists a constant $C$, depending only on the geometry, the vector $\hat T^a$, the Killing form $\hat \nabla^{(a}\hat T^{b)}$ and the function $b$,  such as:
$$
\psi \leq  C \phi \text{ on } U.
$$
If $\psi$ and $\phi$ both satisfy:
$$
\phi \lesssim \psi \text{ and } \psi \lesssim \phi, 
$$ 
we say that $\phi$ and $\psi$ are equivalent and note:
$$
\phi\approx \psi.
$$
\item  The geometric notations are the following:
\begin{itemize}
\item The quantities with $\hat{\phantom{a}}$ are geometric quantities related to the unphysical metric. 
\item $\mu[\hat g]$ is the volume form associated with the metric $\hat g$.
\item If $\nu$ is a form over $\hat M$, then $\star \nu$ is its Hodge dual. If $\nu$ is a 1-form and $V$ the vector field associated with $\nu$ via the metric $\hat g$, then:
$$
\star \nu=V\lrcorner \mu[\hat g] \text{ or } \star V_a=V^a\lrcorner \mu[\hat g]  
$$
where $\lrcorner \mu[\hat g]$ is the contraction with the volume form $\mu[\hat g]$ 
\item $i_\Sigma$ is the restriction to the submanifold  $\Sigma$.  The pull-forward of a form $\nu$ on $\hat M$ over the tangent space to $\Sigma$ is denoted by $i^\star_{\Sigma}(\nu)$.
\end{itemize}
\end{itemize}

\section{Functional and geometric preliminaries}

We present in this section the geometric and analytic background. A specific care is brought to the structure at null infinity and the definition of function spaces on this structure.

\subsection{Geometric framework}
The geometric framework is based on the results of Corvino-Schoen (\cite{MR2225517,MR1794269}) and Chrusciel-Delay(\cite{MR1920322,MR1902228}). They gave a construction of asymptotically simple spacetimes satisfying Einstein equations with specifiable regularity at null and timelike infinities.

\subsubsection{Asymptotic simplicity}
The notion of asymptotically simple spacetimes was introduced by Penrose as a general model for asymptotically flat Einstein spacetimes and their conformal compactification (see \cite{Penrose:1986fk} definition 9.6.11):
\begin{definition}
A smooth Lorentzian manifold $M$ satisfying Einstein equations is said to be ($C^k$) asymptotically simple if there exists a smooth Lorentzian manifold $\hat M$ with boundary, a metric $\hat g$ and a conformal factor $\Omega$
such that:
\begin{enumerate}
\item $M$ is the interior of $\hat M$;
\item $\hat g= \Omega^2 g$ in $M$;
\item $\Omega$ and $\hat g$ and are $C^k$ on $M$;
\item $\Omega$ is positive in $M$; $\Omega$ vanishes at the boundary $\scri$ of $\hat M$ and $\ud \Omega$  does not vanish at $\scri$;
 \item every null geodesic in $M$ acquires a past and future end-point in $\scri$
\end{enumerate}
\end{definition}

We assume that the boundary $\scri$ is $C^2$ (which is sufficient for this work). It is known that this boundary is a null hypersurface (that it is to say that the restriction of the metric to $\scri$ is degenerate) provided that the cosmological constant is zero. Furthermore, $\scri$ has two connected components $\scri^+$ and $\scri^-$ consisting of, respectively, the future and past endpoints of null geodesics. $\scri^+$ and $\scri^-$ are both diffeomorphic to $\R\times \mathbb{S}^2$.

The manifold $(M,g)$ is usually referred to as the physical spacetime and its compactification is referred to as the unphysical spacetime. In order to remain consistent with this notation all along this paper, the quantities associated with the unphysical metric are denoted with a "$\hat{\phantom{a}} $".

\subsubsection{Global hyperbolicity}
An important assumption in the context of the Cauchy problem for a wave equation is the possibility to write the equation as an evolution partial differential equation. This is usually achieved by requiring that the manifold $M$ is globally hyperbolic:
\begin{definition}
A Lorentzian manifold $(M,g)$ is said to be globally hyperbolic if, and only if, there exists in $M$ a global Cauchy hypersurface, i.e. a spacelike hypersurface such that any inextendible timelike curve intersects this surface exactly once.
\end{definition}
A useful consequence of this is the existence of a time function on $M$ and the parallelizability of $M$, that is to say the existence of a global section of the principal bundle of orthonormal frames (see the work of Geroch in \cite{g68,g70}  and Bernal-Sanchez in \cite{MR2029362}).

In the case of an asymptotically simple manifold $(M,g)$,  this property extends of course to the manifold $(\hat M, \hat g)$.

\subsubsection{Corvino-Schoen/Chrusciel-Delay spacetimes}
We can then introduce the spacetimes obtained by Corvino-Schoen and Chrusciel-Delay:
\begin{definition}
A spacetime $M$ is of Chrusciel-Delay/Corvino-Schoen type if:
\begin{enumerate}
\item $M$ is asymptotically simple;
\item $M$ is globally hyperbolic; let $\Sigma_0$ be a spacelike Cauchy hypersurface; $M$ is then diffeomorphic to $\R\times\Sigma_0 $;
\item $\hat M$ can be completed into a compact manifold by adding three points $i^0$, $i^+$ and $i^-$ such that $\scri^+$ and $\scri^-$ are respectively the past null and future null cones from $i^+$ and $i^-$ and $i^0$ is the conformal infinity  of the spacelike hypersurface $\Sigma_0$ for the metric $\hat g$;
\item there exists a compact set $K$ in $\Sigma_0$ such that $(\R\times \Sigma_0\backslash K, g)$  is isometric to $(\R\times]r_0, +\infty[\times \mathbb{S}^2, g_S)$, where $g_S$ is the Schwarzschild metric with mass $m$ and $r_0>2m$;
\item there exists a neighborhood of $i^+$ such that the metric $\hat g$ is obtained in this neighborhood as the restriction of a smooth ($C^2$) Lorentzian metric of an extension of $\hat M$ in the given neighborhood of $i^+$. The same property holds in $i^-$.
\end{enumerate}
\end{definition}   
\begin{remark}
\begin{enumerate}
\item The result of Corvino-Schoen/Chrusciel-Delay states that the metric is isometric to the Kerr metric outside a compact set; we restrict ourself to a Schwarzschild metric.
\item The extension of the manifold $\hat M$ in the neighborhood of $i^+$ was used by Mason-Nicolas (see \cite{mn04, mn07}). The point $i^+$ remains singular in $\hat M$ but it is nonetheless possible, because of the existence of this extension, to consider geometric data in $i^+$ (metric, exponential map, connection, curvature) as being the one obtained from the Lorentzian manifold extending $(\hat M, \hat g)$ in a neighborhood of this point.
\item The point $i^0$ is a "real" singularity of $\hat M$; nonetheless, the geometry is completely known in its neighborhood. This singularity is the main problem we have to deal with in this paper, being an obstacle to global estimates and Sobolev embeddings for instance.
\end{enumerate}
\end{remark}

The last part of this section is devoted to the geometric description around the point $i^0$, that is to say in the Schwarzschildean part of the manifold.

We consider here a neighborhood $O$ in $\hat M$ of $i^0$ where the metric $g$ is the Schwarzschild metric. This metric, in the Schwarzschild coordinates $(t,r, \omega_{\mathbb{S}^2})$, can be written:
\begin{equation*}
g= F(r)\ud t^2-\frac{1}{F(r)}\ud r^2 -r ^2 \ud^2 \omega_{\mathbb{S}^2}
\end{equation*}
where:
\begin{equation*}
F(r)=1-\frac{2m}{r}
\end{equation*}
with $m$ a positive constant. Introducing the new coordinates: 
\begin{equation*}
r^\ast=r+2m\log(r-2m), u= t-r^\ast \text{ and } R=\frac{1}{r}
\end{equation*}
the metric can then be expressed as:
$$
g=(1-2mR)\ud ^2u -\frac{2}{R^2}\ud u\ud R-\frac{1}{R^2} \ud^2\omega_{\mathbb{S}^2}.
$$ 
where $\ud^2\omega_{\mathbb{S}^2}$ stands for the standard volume form on the 2-sphere, which can be written in polar coordinates $(\theta, \phi)$:
$$
\ud^2\omega_{\mathbb{S}^2}=\sin{\theta}\ud \theta \wedge \ud \phi.
$$
Its inverse is:
$$
g^{-1}=\frac{-2}{R^2}\partial_u\partial_R-(1-2mR)\partial^2_R-\frac{1}{R^2}\eth\eth'.
$$

The part of the Cauchy hypersurface $\Sigma_0$ is given by the equation $\{t=0\}$ in these coordinates.

In the neighborhood $O$, the conformal factor is chosen to be:
$$
\Omega=R
$$ 
and is extended smoothly in $M\backslash O$. The conformal metric is then:
$$
\hat g=R^2(1-2mR)\ud ^2u - 2\ud u\ud R-\ud^2\omega_{\mathbb{S}^2}.
$$ 
and its inverse:
$$
\hat g^{-1}=-2\partial_u\partial_R-R^2(1-2mR)\partial^2_R-\eth\eth'.
$$
The point $i^0$ is given in these coordinates by $u=-\infty,\, R=0$.

The description of the geometry around $i^0$ is completed by the following lemma (lemma A.1 in \cite{mn07}):
\begin{lemma}\label{schwarzestimates} Let $\epsilon>0$. There exists $u_0<0$,  $|u_0|$ large enough, such that the following decay estimates in the coordinate $(u,r,\theta, \psi)$ hold:
\begin{gather*}
r<r^\ast<(1+\epsilon)r,
1<Rr^\ast<1+\epsilon,
0<R|u|<1+\epsilon,\\
1-\epsilon<1-2mR<1, 
0<s=\frac{|u|}{r^\ast}<1
\end{gather*}
\end{lemma}
\begin{proof} The proof of this lemma is straightforward: it only consists in writing simultaneously the asymptotic behavior or the continuity over $\hat M$ of each of the coordinates involved in the lemma. \end{proof}

\begin{remark}\label{thuniformlytimelike}\begin{enumerate}
\item  As mentioned in the introduction, we choose to work in the neighborhood of $i^0$ with the Morawetz vector $\hat T^a$ defined by:
$$
\hat T^a = u^2\partial_u-2(1+uR)\partial_R.
$$
The squared norm of this vector is:
\begin{equation*}
\hat g_{ab} \hat T^a \hat T^b=u^2(4(1+uR)+u^2R^2(1-2mR)).
\end{equation*}
This polynomial in $uR$ vanishes at:
$$
2\frac{1\pm\sqrt{2mR}}{1-2mR}=\frac{-2}{1\mp \sqrt{2mR}},
$$
so that, if $R$ is chosen small enough, these roots are arbitrarily close to $-2$. Let $\epsilon$ be a positive number chosen such that the inequalities in lemma \ref{schwarzestimates} hold for a given $u_0$.  The largest zero of this polynomial satisfies:
$$
\frac{-2}{1+ \sqrt{2mR}}\leq \frac{-2}{1+\sqrt{\epsilon}}.
$$
Choosing $\epsilon$ such that
$$
\frac{-2}{1+\sqrt{\epsilon}}\leq -1-\epsilon, 
$$
the norm of $\hat T^a$ is then uniformly controlled by:
$$
\hat g_{ab} \hat T^a \hat T^b=u^2(4(1+uR)+u^2R^2(1-2mR))\geq 4u_0^2\epsilon
$$
on the domain $\Omega^+_{u_0}=\{(u,R,\omega_{\mathbb{S}^2})|u<u_0\}\cap J^+(\Sigma_0)$, $J^+(\Sigma_0)$ being the future of $\Sigma_0$.
\item Another criterion, given in proposition \ref{energyequivalenceschwarzschild}, will be required to define $\epsilon$.
\end{enumerate}
\end{remark}

\subsection{Analytical requirements}
We introduce in this section the technical and analytical tools which are required to the description of the solution for the wave equation.

\subsubsection{Conformal wave equation}
We recall here how the problem on the physical space time is transformed into a problem on the unphysical spacetime. This is based on the classic transformation of the wave d'Alembertian operator:
\begin{lemma}
Let $M$ be a Lorentzian manifold endowed with the conformal metrics $g$ and $\hat g=\Omega^2 g$ where $\Omega$ is a conformal factor in $C^2(M)$. The connections associated with $g$ and $\hat g$ are denoted by $\nabla$ and $\hat \nabla$ respectively.\\
Then, for any smooth function $\phi$ on $M$, the following equality holds:
$$
\nabla_a\nabla^a \phi +\frac{1}{6}\scal_{g} \phi=\Omega^{-3}\left(\hat \nabla_a\hat \nabla^a \left(\Omega^{-1}\phi
\right) +\frac{1}{6}\scal_{\hat g} \left(\Omega^{-1}\phi\right)\right)
$$
where $\scal_{g}$ and $\scal_{\hat g}$ are the scalar curvatures associated with $g$ and $\hat g$ respectively.
\end{lemma}
Assuming that we are working on a vacuum spacetime, for which the scalar curvature vanishes, the equation becomes:
\begin{equation}\label{conformalchange}
\nabla_a\nabla^a \phi =\Omega^{-3}\left(\hat \nabla_a\hat \nabla^a \left(\Omega^{-1}\phi
\right) +\frac{1}{6}\scal_{\hat g} \left(\Omega^{-1}\phi\right)\right)
\end{equation}
We obtain in particular the useful formula:
\begin{equation}\label{conformalchangecurvature}
\Omega^3\nabla_a\nabla^a\Omega =\frac16 \scal_{\hat g}.
\end{equation}

Let us now consider the Cauchy problem on the physical spacetime $M$:
\begin{equation}\label{cauchy001}
\left\{
\begin{array}{l}
\square \phi+b\phi^3=0\\
\phi\big|_{\Sigma_0}=\theta\in C^\infty_0(\Sigma_0)\\
\hat T^a \nabla_a \phi\big|_{\Sigma_0}=\xi \in C^\infty_0(\Sigma_0) 
\end{array}\right..
\end{equation}
Using this transformation, this Cauchy problem is transformed into a Cauchy problem on the unphysical spacetime $\hat M$ as follows:
\begin{lemma} The function $\phi$ is a solution of problem \eqref{cauchy001} if, and only if, the function
$$
\psi = \Omega^{-1} \phi 
$$
is solution of the problem on $\Sigma_0$:
\begin{equation*}
\left\{
\begin{array}{l}
\hat \square \psi+ \frac16 \scal_{\hat g}\psi+b\psi^3=0\\
\psi\big|_{\Sigma_0}=\Omega^{-1}\theta\in C^\infty_0(\Sigma_0)\\
\hat T^a \hat \nabla_a \psi\big|_{\Sigma_0}=\frac{1}{\Omega}\left(\xi -(\hat T^a\hat \nabla_a \Omega)\frac{\theta}{\Omega}\right)\in C^\infty_0(\Sigma_0)
\end{array}\right..
\end{equation*}
\end{lemma}
\begin{remark}
\begin{enumerate}
\item Because of the finite propagation speed, since the data on the physical spacetime are smooth with compact support on $\Sigma_0$, the data on the unphysical spacetime are smooth with compact support in $\Sigma_0$. 
\item Another consequence of the finite propagation speed is that, since the data remain with compact support in $\Sigma_0$, we do not have to deal with the singularity in $i^0$.
\item Conversely, it is possible to start with a Cauchy problem on $\Sigma_0$ in the unphysical spacetime and obtain a Cauchy problem on the physical spacetime: starting with the Cauchy problem on $\hat M$:
\begin{equation*}
\left\{
\begin{array}{lcl}
\hat \square \psi+ \frac16 \scal_{\hat g}\psi+b\psi^3=0&&\\
\psi\big|_{\Sigma_0}=\hat \theta\in C^\infty_0(\Sigma_0)&&\\
\hat T^a \hat \nabla_a \psi\big|_{\Sigma_0}=\hat \xi \in C^\infty_0(\Sigma_0)  &&
\end{array}\right.,
\end{equation*}
then $\phi = \Omega \psi$ satisfies the Cauchy problem:
\begin{equation*}
\left\{
\begin{array}{lcl}
\hat \square \phi+b\psi^3=0&&\\
\psi\big|_{\Sigma_0}=\Omega\hat \theta\in C^\infty_0(\Sigma_0)&&\\
\hat T^a \nabla_a \psi\big|_{\Sigma_0}= \Omega\hat \xi+(\hat T^a\nabla_a \Omega) \hat \theta \in C^\infty_0(\Sigma_0)  &&
\end{array}\right..
\end{equation*}
\end{enumerate}
\end{remark}

\subsubsection{Function spaces}\label{functionalspaces}

The purpose of this section is to give a precise description of the Sobolev spaces which are used in the present paper. Two problems are encountered in this section: the first one consists in adapting the definition of Sobolev spaces on a null hypersurface and the second is the difficulty coming from the singularity at $i^0$. This difficulty has two aspects: the necessity to adapt the definition of the Sobolev space to the singularity: this is solved using weighted Sobolev spaces on $\scri$. The other aspect of this singularity is encountered in section \ref{continuitypart} when trying to obtain uniform Sobolev estimates of the non linearity.

We recall the definition of a Sobolev space on a uniformly spacelike hypersurface $\Sigma$: for a smooth function $u$ on $\Sigma$, consider the norm, when the integral exists:
$$
||u||^2_{p}=\int_{\Sigma} \sum_{k=0}^p ||D^k u||^2_{h}\ud \mu[ h ],
$$
where $h$ is the restriction of the metric $\hat g$ and $D$ is the restriction of the connection $\hat \nabla$ to $\Sigma$.
\begin{definition}
The completion of the space:
$$
\Big\{ u \in C^\infty(\Sigma)\big|\, ||u||_p<+\infty\Big\}
$$ 
 in the norm $||\star||_p $is denoted by $H^p(\Sigma)$.\\
 When $\Sigma$ is a compact spacelike hypersurface with boundary, the completion of the space of smooth functions with compact support in the interior of $\Sigma$
 in the norm $||\star||_p $is denoted by $H_0^p(\Sigma)$.\\
\end{definition}
\begin{remark}
It is known that, when working on a Riemannian closed manifold, the Sobolev spaces are independent of the choice of the metric. Nonetheless, this fact is not true any more when working with a weakly spacelike hypersurface, as we are about to see. Arbitrary choices are made for their definitions.
\end{remark}

Because of the degeneracy of the metric, it is not possible to define on $\scri^+$ geometric quantities that only depend on the metric $\hat g$. Two solutions can be provided:
\begin{itemize}
\item lifting the metric from $\Sigma_0$ to $\scri^+$;
\item adding geometric information on $\scri^+$ by using the uniformly timelike vector field $\hat T^a$.
\end{itemize}
Following \cite{arXiv:0903.0515v1,joudioux-2009} and using Geroch-Held-Penrose formalism, $\scri^+$ is endowed with a basis $(l^a, n^a, e^a_3, e^a_4)$ such that:
\begin{itemize}
\item $l^a$ and $n^a$ are two future directed null vectors; $n^a$ is tangent to $\scri^+$; they satisfy:
$$
l^a+n^a=\hat T^a;
$$ 
in the neighborhood of $i^0$, they are chosen to be:
$$
l^a= -2\partial_R\text{ and } n^a= u^2 \partial_u
$$
\item the set $\{l^a, n^a\}$ is completed by two vectors $(e^a_3, e^a_4)$ orthogonal to $\{l^a, n^a\}$, orthogonal to each other and normalized.
\end{itemize}
$\scri^+$ is then endowed with the volume form $i^\star_{\scri^+}\left(l^a\lrcorner \mu[\hat g]\right)$. The following norm is defined on $\scri^+$, for $u$ a smooth function with compact support which does not contain $i^0$ or $i^+$:
$$
||u||^2_{H^1(\scri^+)}=\int_{\scri^+}{\left(\frac{\left(n^a\hat \nabla u \right)^2}{\hat g_{cd}\hat T^c \hat T^d} +\left|\hat \nabla_{\mathbb{S}^2 }u\right|^2+u^2\right)}i^\star_{\scri^+}\left(l^a\lrcorner \mu[\hat g]\right)
$$
where $\left|\hat \nabla_{\mathbb{S}^2}u\right|$ stands for the derivatives with respect to $(e^a_3, e^a_4)$.

Following chapter 5.4.3 of Friedlander's book (\cite{Friedlander:1975vn}), the Sobolev space $H^1$ on $\scri^+$ is finally defined:
\begin{definition}\label{defh1}
Let $\tilde M$ be an extension of $\hat M$ behind $i^+$ and consider the function space $\mathcal{D}(\scri^+)$ on $\scri^+$ obtained as the trace of smooth functions with compact support in $\tilde M$ which does not contain $i^0$.
The weighted Sobolev space $H^1(\scri^+)$ is defined as the completion of the space $\mathcal{D}(\scri^+)$ in the norm $||\star ||_{H^1(\scri^+)}$.
\end{definition}

Since the volume form is written on the Schwarzschildean part of $\hat M$ as, using polar coordinates:
$$
\mu[\hat g]=\sin(\theta)\ud u \wedge \ud r \wedge \ud \theta\wedge\ud \psi,
$$
this norm is written in this region of the manifold as:
$$
||\phi||^2_{H^1(\scri^+)}=2\int_{\scri^+} \left(\frac14 u^2 (\partial_u \phi)^2+|\hat \nabla_{\mathbb{S}^2}\phi |^2+\phi^2\right) \ud u\ud \omega_{\mathbb{S}^2}
$$
since, on $\scri^+$,
$$
n^a=u^2\partial_u\text{ and }\hat g_{cd}\hat T^c\hat T^d=4u^2 
$$

The metric at $i^+$ is obtained as the restriction of a smooth metric of an extension of $\hat M$ beyond $\scri^+$. As a consequence, a trace theorem could give another way, more intrinsic, to define $H^1(\scri^+\cap O)$ where $O$ is a bounded open set around $i^+$. Another way to obtain the fact that the point $i^+$ does not matter in the definition of the Sobolev space over $\scri^+$ is to notice the following property:
\begin{proposition}\label{localsobolev}
The set of smooth functions with compact support on $\scri^+$, whose support does not contain $i^+$, is dense in $H^1(\scri^+)$.
\end{proposition}
\begin{proof} The method of the proof relies on the construction of 
an identification between $H^1_0(\Sigma_0)$ and $H^1(\scri^+)$.  This identification is brought by H\"ormander in \cite{MR1073287} and is obtained as follows.

Let $t$ be a smooth time function in the future of $\Sigma_0$ in $\hat M$. This time function gives rise to local coordinates, where $t$ is the time coordinate. We denote by $\partial_t$ the vector field associated with this coordinate. The flow associated with this vector field is denoted $\Psi_t$.

For $x$ in $\Sigma_0$, let $\phi(x)$ be the time at which the curve $\Psi_t(x)$ hits $\scri^+$ and consider the application defined by:
$$
\xi \in C_0^\infty (\Sigma_0)\longmapsto \left(y\in \scri^+\mapsto \xi\left(\Psi_{-\phi(y)}(y)\right)\right)
$$ 
This application has value in $C_0^\infty(\scri^+)$ since the future of a compact set in $\Sigma_0$ has compact support in $\scri^+$. Furthermore, this application can easily be inverted:
$$
\xi \in C_0^\infty (\scri^+)\longmapsto \left(x\in \scri^+\mapsto \xi\left(\Psi_{-\phi(x)}(x)\right)\right)\in C^\infty_0(\Sigma_0) 
$$
and can consequently be used to define on $\scri^+$ a Sobolev space by pushing forward the $H^1$-norm on $\Sigma_0$. The Sobolev spaces which are obtained are then equivalent on $H^1(\scri^+)$ since the norms are equivalent on any compact set of $\scri^+$.

To prove the proposition, it is then sufficient to prove that the smooth functions with compact support in $\Sigma_0$ which does not contain the preimage of $i^+$ by the flow associated with the time function $t$. Since $\Sigma_0$ has no topology, it is sufficient to establish the following lemma:
\begin{lemma}\label{techlemma1}
Let us consider the set of smooth functions defined in $\overline{B(0,1)}\subset \R^3$ with support which does not contain $0$. Then this set is dense in $H^1(B(0,1))$.
\end{lemma}
\begin{proof} The proof is given in appendix \ref{techlemma11}. 
\end{proof}

Using this lemma in the neighborhood of the preimage of $i^+$ immediatly gives the result. \end{proof}

\subsection{Cauchy problem}
A well-posedness result for the Cauchy problem in our framework is now stated. It is based on a result of Cagnac and Choquet-Bruhats in \cite{MR789558}
(and see also \cite{MR2473363}, appendix III for a survey on the wave equation, and appendix III chapter 5 for our problem).

The geometric framework for this well-posedness theorem is the following (definition 11.8 in \cite{MR2473363}): 
\begin{definition}\label{regularlysliced}
A spacetime $(M,g)$ is said to be regularly sliced if there exists a smooth 3-manifold $\Sigma$ endowed with coordinates $(x^i)$ and an interval $I$ of $\R$ such that $M$ is diffeomorphic to $I\times \Sigma$ and the metric $g$ can be written:
$$
g=N^2 \ud t^2 -g_{ij}(\ud x^i+\beta^i\ud t).
$$
and its coefficients satisfy:
\begin{enumerate}
\item the lapse function $N$ is bounded above and below by two positives constants:
$$
\exists (c, C), 0<c\leq N\leq C;
$$
\item for $t$ in $I$, the 3-dimensional Riemannian manifolds $\left(\{t\}\times M, g_{t,ij}=g_{ij}\Big|_{\{t\}\times M}\right)$ are complete and the metrics $g_t$ are bounded below by a metric $h$ i.e.: 
$$
 \forall V \in T\Sigma,  h_{ij}V^i V^j\leq g_{t,ij}V^iV^j;
$$ 
\item and, finally, the norm for the metric $g_t$ of the vector $\beta$ is uniformly bounded on $M$. 
\end{enumerate}
\end{definition}
\begin{remark}
\begin{enumerate}\label{globalhyp}
\item This hypothesis implies that the spacetime is globally hyperbolic (theorem 11.10 in \cite{MR2473363}).
\item The asymptotically simple spacetime and its compacification which we are working with do not satisfy this property. 
\end{enumerate}
\end{remark}
The following theorem, obtained by Cagnac and Choquet-Bruhat  in \cite{MR789558}, gives existence and uniqueness of solutions to the Cauchy problem for a cubic wave equation:
\begin{theorem}[Cauchy problem for a nonlinear wave equation]\label{choquetbruhat}
Let us consider the Cauchy problem on the regularly sliced manifold  $(M=\R\times\Sigma, g)$:
\begin{equation*}\left\{
\begin{array}{l}
\hat \square \phi+ \frac16 \scal_{\hat g}\phi+b\phi^3=0\\
\phi\big|_{\{0\}\times\Sigma}=\theta\in H^1(\Sigma)\\
\partial_t  \phi\big|_{\{0\}\times\Sigma}= \tilde \theta\in L^2(\Sigma)
\end{array}\right.,
\end{equation*}
where the function $b$ satisfies the assumptions \ref{A1} and \ref{A3}.\\
Then this problem admits a unique global solution on $M$ in $C^0(\R, H^1(\Sigma))\cap C^1(\R, L^2(\Sigma))$.
\end{theorem}

\begin{proof}
As already noted (see remark \ref{globalhyp}), this theorem cannot of course be applied directly to the compactification of $M$ because of the geometry in the Schwarzschildean part. This problem can be solved using an extension of $\hat M$ constructed as follows:
\begin{enumerate}
\item Let $\theta$ and $\tilde \theta$ be two functions respectively in $H^1(\Sigma_0)$ and $L^2(\Sigma_0)$ with compact support in the interior of $\Sigma_0$. Let $K$ be a compact subset of $\Sigma_0$ containing the supports of $\theta$ and $\tilde \theta$.
\item Let $\hat t$ be a time function on $\hat M$ such that $\Sigma_0$ is given by $\{\hat t=0\}$; the associated foliation is denoted by $(\Sigma_{\hat t})$ for $\hat t\in [0,\hat T]$; we assume that the gradient of this time function is uniformly timelike for the metric $\hat g$.
\item The manifold $(J^+(K),g)$ is a 4-dimensional Lorentzian manifold with boundary. This boundary is constituted of $K$, the light cone from $K$, $C^+(K)$, and the part of $\scri\cup \{i^+\}$ in the future of K. Since $(\hat M, g)$ is extendible smoothly in the neighborhood of $i^+$, there exists a smooth extension of $(J^+(K), g)$ into a 4-dimensional Lorentzian manifold $(\tilde M, \tilde g)$, depending on the support of $K$ such that the manifold $\tilde M$ is diffeomorphic to $[0,\hat T]\times \tilde \Sigma$, where $\tilde \Sigma$ is topologically equivalent to $\mathbb{S}^3$ and such that the foliation $(\Sigma_{\hat t})$ is extended into the uniformly spacelike foliation $(\{t\}\times \tilde \Sigma)$.
\end{enumerate}

Since $\tilde M$ is compact, conditions 1, 2 and 3 of definition \ref{regularlysliced} are  satisfied. As a consequence, using theorem \ref{choquetbruhat}, we obtain the well-posedness in $C^0(\R, H^1(\Sigma_0))$ of the Cauchy problem for
\begin{equation*}\left\{
\begin{array}{l}
\hat \square \phi+ \frac16 \scal_{\hat g}\phi+b\phi^3=0\\
\phi\big|_{\{0\}\times\Sigma}=\theta\in H^1(\Sigma) \text{ with compact support in }K \\
\partial_t  \phi\big|_{\{0\}\times\Sigma}= \tilde \theta\in L^2(\Sigma)  \text{ with compact support in }K
\end{array}\right.
\end{equation*}
As a consequence, we obtain by restriction to $J^+(K)$ and on $\hat M$ well posedness of the Cauchy problem for data in $H^1_0(\Sigma_0)\times L^2(\Sigma_0)$.
\end{proof}
\begin{remark} The extension of the solution of the wave equation to a cylinder was also used by Mason-Nicolas in \cite{mn07} (proposition 6.1) to obtain energy estimates.
\end{remark}

\section{A priori estimates}\label{aprioriestimates}

The purpose of this section is to establish a priori estimates for solutions of the wave equation, in the sense that it is possible to control the energy on $\scri^+$ by the energy on $\Sigma_0$ and reciprocally. These a priori estimates will be used in the next section to establish the continuity of the conformal wave operator and obtain its domain of definition and the existence of trace operators.

Let us consider in this section a smooth solution with compactly supported data of the problem:
\begin{equation}\label{smoothproblem}
\hat \square \phi+ \frac16 \scal_{\hat g}\phi+b\phi^3=0\\
\end{equation}
and the associated stress-energy tensor:
$$
T_{ab}=\hat \nabla_a \phi\hat \nabla_b \phi+\hat g_{ab}\left(-\frac12 \hat \nabla_c\phi \hat \nabla^c\phi+\frac{\phi^2}{2}+b\frac{\phi^4}{4}\right).
$$
The contraction of this tensor with $\hat T^a$, $\hat T^aT_{ab}$, is called "energy 3-form"; it satisfies an "approximate conservation law":
\begin{lemma}
The derivative of the energy 3-form satisfies:
\begin{eqnarray*}
\hat \nabla^a \left(\hat T^bT_{ab}\right)&=&\left(\hat \nabla^{(a} \hat T^{b)}\right)T_{ab}+\left(1-\frac{1}{6}\scal_{\hat g}\right)\phi \hat T^a\hat \nabla_a \phi+\hat T^a\hat \nabla_a b\frac{\phi^4}{4}.
\end{eqnarray*}
\end{lemma}
This derivative will be designated as "the error term" since it arises in the volume term when applying Stokes theorem.

A quantity which is equivalent to the integral:
$$
 \int_S i^\star (\star T^aT_{ab})
$$
on the hypersurface $S$ is called an energy and is denoted by $E(S)$. Global energies on $\Sigma_0$ and $\scri^+$ are piecewise defined in propositions \ref{energyequivalenceschwarzschild}, \ref{energyequivalence} and \ref{eqenergie3}. The purpose of this multiple definitions is to simplify the comparison with other quantities.

\subsection{Estimates in the neighborhood of $i^0$}\label{schwarzschildpart}

The purpose of this section is to obtain a priori estimates for the energy associated with the energy 3-form on $\scri^+$ and $\Sigma_0$, using the fact that the geometry is known almost completely.

The estimates which are obtained in this section are close to the ones obtained for the linear wave equation by Mason-Nicolas in \cite{mn07}. These estimates are based on two main tools:
\begin{itemize}
\item an explicit control of the decay of the physical metric in the neighborhood of $i^0$
\item and the use of Gronwall lemma. 
\end{itemize}

We define, in $\Omega^+_{u_0}=\{t>0, u<u_0\}$, the following hypersurfaces, for $u_0$ given in $\R$:
\begin{itemize}
\item $S_{u_0}=\{u=u_0\}$, a null hypersurface transverse to $\scri^+$;
\item $\Sigma_{0}^{u_0>}=\Sigma_0\cap\{u_0>u\}$, the part of the initial data surface $\Sigma_0$ beyond $S_{u_0}$;
\item $\scri_{u_0}^+=\Omega^+_{u_0}\cap \scri^+$, the part of $\scri^+$  beyond $S_{u_0}$;
\item $\mathcal{H}_s=\Omega^+_{u_0}\cap\{u=-sr^\ast\}$, for $s$ in $[0,1]$, a foliation of $\Omega^+_{u_0}$ by spacelike hypersurfaces accumulating on $\scri$. 
\end{itemize}

The volume form associated with $\hat g$ in the coordinates $(R,u,\omega_{\mathbb{S}^2})$ is then:
\begin{equation}\label{volumeformcoor}
\mu[\hat g]= \ud u\wedge\ud  R \wedge \ud^2\omega_{\mathbb{S}^2}. 
\end{equation}

\begin{figure}
\begin{center}
\resizebox{9cm}{8cm}{\input{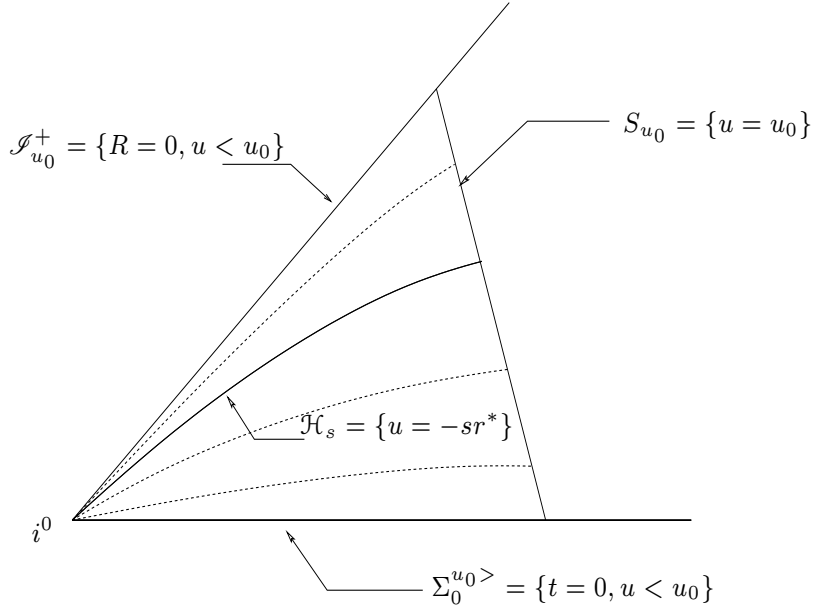}}
\end{center}
\caption{Neighborhood of $i^0$}
\end{figure}

Finally, we consider the approximate conformal Killing vector field $\hat T^a$:
$$
\hat T^a= u^2\partial_u-2(1+uR)\partial_R.
$$
\begin{remark}
\begin{enumerate}
\item This vector field is timelike for the unphysical metric and, as such, is transverse to $\scri$. More precisely, it is uniformly timelike in a neighborhood of $i^0$ (see remark \ref{thuniformlytimelike} for the choice of the neighborhood).
\item Its expression is derived from the so-called Morawetz vector field in the Minkowski space and was previously used to obtain pointwise estimates in the flat case.
\end{enumerate}
\end{remark}

The strategy of the proof in this section is the following: 
\begin{enumerate}
\item writing an explicit description of the hypersurfaces $S_{u_0}$ and $\mathcal{H}_s^{u_0}$;
\item\label{stokes1} proving energy equalities in both ways for  $\int \star \hat T^aT_{ab}$ using Stokes theorem between the hypersurfaces $\Sigma_{u_0}^+$ and $\mathcal{H}_s^{u_0}$;
\item determining an energy $E(\mathcal{H}_s^{u_0})$ equivalent to $\int \star \hat T^aT_{ab}$ from the decay of the metric $g$;
\item obtaining an integral inequality for $E(\mathcal{H}_s^{u_0})$ to apply the Gronwall lemma;
\item starting from point \ref{stokes1}, doing the same work from Stokes theorem applied between $\scri_{u_0}^+ $ and $\mathcal{H}_s^{u_0}$.
\end{enumerate}

\subsubsection{Geometric description}

This section is devoted to the description of the energy associated with the nonlinear wave equation in the neighborhood of $i^0$.

\begin{proposition}\label{expressionschwenergie}The energy 3-form, written in the coordinates $(R,u,\theta, \psi)$, is given by:
\begin{gather*}
\star \hat T ^aT_{ab}=\left[u^2(\partial_u\phi)^2+R^2(1-2mR)\left(u^2\partial_R\phi\partial_u\phi-(1+uR)(\partial_R\phi\right)^2)\right.\\
+\left.\left(\frac12|\nabla_{\mathbb{S}^2}\phi|^2+\frac{\phi^2}{4}+b\frac{\phi^4}{4}\right)\right]\sin(\theta)\ud u \wedge \ud \theta \wedge \ud \psi \\
+\left[\frac12\left((2+uR)^2-2mR^3u^2\right)\left(\partial_R \phi^2\right) +u^2\left(\frac12|\nabla_{\mathbb{S}^2}\phi|^2+\frac{\phi^2}{2}+b\frac{\phi^4}{4}\right)\right]\sin(\theta)\ud R \wedge \ud \theta \wedge \ud \psi\\
+\sin(\theta)\left[u^2\partial_u\phi-2(1+uR)\partial_R\phi\right]\left(-\partial_\theta\phi \ud u \wedge \ud R \wedge \ud \psi+\partial_\psi\phi \ud u \wedge \ud R \wedge \ud \theta \right)
\end{gather*}
The restriction of the energy 3-form can be written:
\begin{itemize}
\item to $\mathcal{H}_s$:
\begin{gather*}
i^\star_{\mathcal{H}_s}(\star \hat T^aT_{ab})=\left(u^2(\partial_u\phi)^2+R^2(1-2mR)u^2\partial_R\phi\partial_u\phi\right.\\\left.+R^2(1-2mR)\left( \frac{(2+uR)^2}{2s}-\frac{mu^2R^3}{s}-(1+uR)\right)(\partial_R \phi^2)\right.\\
+\left.\left(\frac{u^2R^2(1-2mR)}{s}+2(1+uR)\right)\left(\frac12|\nabla_{\mathbb{S}^2}\phi|+\frac{\phi^2}{2}+b\frac{\phi^4}{4}\right)\right) \sin(\theta)\ud u \wedge\ud \theta \wedge \ud \psi;
\end{gather*}
\item to $S_u$:
\begin{gather*}
i^\star_{S_u}(\star \hat T^aT_{ab})=\Bigg(\frac12\left((2+uR)^2-2mR^3u^2\right)\left(\partial_R \phi^2\right)\\ +u^2\left(\frac12|\nabla_{\mathbb{S}^2}\phi|^2+\frac{\phi^2}{2}+b\frac{\phi^4}{4}\right)\Bigg)\sin(\theta)\ud R \wedge \ud \theta \wedge \ud \psi;
\end{gather*}
\item to $\scri^+_{u_0}$
$$
i^\star_{\scri^+}(\star \hat T^aT_{ab})=\left(u^2(\partial_u\phi)^2+|\nabla_{\mathbb{S}^2}\phi|+\phi^2+b\frac{\phi^4}{2}\right)\sin(\theta)\ud u  \wedge \ud \theta \wedge \ud \psi.
$$
\end{itemize}
\end{proposition}
\begin{proof} The proof of this proposition is a straightforward calculation and is given in the appendix (see proposition \ref{expressionschwenergie2}).\end{proof}

\begin{proposition}\label{energyequivalenceschwarzschild}
There exists $u_0$, such that the following energy estimates holds on $\mathcal{H}_s$ in $\Omega^+_{u_0}$:
$$
\int_{\mathcal{H}_s}i^\star_{\mathcal{H}_s}\left(\star \hat T^a T_{ab}\right)\approx \int_{\mathcal{H}_s}\left(u^2(\partial_u\phi)^2+\frac{R}{|u|}(\partial_R\phi)^2+|\nabla_{\mathbb{S}^2}\phi|^2+\frac{\phi^2}{2}+b\frac{\phi^4}{4}\right)\ud u \wedge \ud \omega_{\mathbb{S}^2}
$$
\end{proposition}
\begin{proof}The proof of this proposition can be found in the appendix (see proposition \ref{energyequivalenceschwarzschild1}).  \end{proof}

$u_0$ is now fixed, being equal to the $u_0$ associated with the $\epsilon$ which ensures that the energy equivalence established in proposition \ref{energyequivalenceschwarzschild} holds. The neighborhood of $i^0$ where the energy estimates are relevant is then $\Omega^+_{u_0}$.

\subsubsection{Energy estimates near $i^0$}

The energy estimates are established between $\sop$, $S_{u_0}$ and $\scri^+_{u_0}$, by writing a Stokes theorem between $\mathcal{H}_s$ , $S_{u_0}^s=\{(u,R,\omega_{\mathbb{S}^2})|u=u={u_0}, u\leq  -sr^\ast\}$ and $\scri^+_{u_0}$.

The first step consists in evaluating the error term:
\begin{lemma}\label{errorschwarzschild}
The error is given by:
$$
\hat\nabla^{a}\left(\hat T^bT_{ab}\right)=4mR^2(3+uR)\left(\partial_R\phi\right)^2+\left(1-12mR\right)\phi\left(u^2\partial_u \phi-2(1+uR)\partial_R \phi\right)+\hat T^a \hat \nabla_a b \frac{\phi^4}{4}.
$$
\end{lemma}
\begin{proof} The proof of this lemma is essentially based on a computation (See the complete calculation in the appendix in proposition \ref{errorschwarzschild2})\end{proof}

%
\begin{remark}
As noticed in \cite{mn07}, one obstacle to the use of the parameter $s$ for the foliation is the fact that this parametrization in $s$ is not smooth in the sense that $(r^\ast)^{-1}$ is not a smooth function of $R$ at $R=0$.
\end{remark}
In order to avoid this singularity, the speed of the identifying vector field is decreased by a change of variable: let $\tau$ be the function defined by:
\begin{equation}\label{parametrization}
\tau:
\begin{array}{ccc}
[0,1]&\longrightarrow& [0,2]\\
s&\longmapsto& -2(\sqrt{s}-1).
\end{array}
\end{equation}
$\scri^+_{u_0}$ is then given by $s=0$ and $\tau=2$ and $\sop$ is given by $s=1$ and $\tau=0$. The new identifying vector field $V^a$ is then chosen such that:
\begin{equation}\label{idvecf}
\ud \tau(V^a)= 1 \text{ so that } V^a=(r^\ast R)^\frac{3}{2}(1-2mR)\sqrt{\frac{R}{|u|}}\partial_R^a.
\end{equation}
The foliation $\mathcal{H}_s$, when parametrized by $\tau$, is denoted by $\Sigma_\tau$.

Finally, we can prove the following estimates:
\begin{proposition}\label{aprioriestimates1}
The following equivalence holds:
$$
E(\scri^+_{u_0})+E(S_{u_0})\approx E(\som)
$$
where
$$
E(\scri^+_{u_0})=\int_{\scri_{u_0}^+}i^\star_{\scri^+}(\star \hat T^aT_{ab})
\text{    and   } 
E(S_{u_0})=\int_{S_{u_0}}i^\star_{S_{u_0}}(\star \hat T^aT_{ab}).
$$
\end{proposition}
\begin{proof} The proof of these estimates relies on Stokes theorem applied between the hypersufaces $S_{u_0}^s=\{(u,R,\omega_{\mathbb{S}^2})|u=u_0, |u|\geq sr^\ast\}$, $\mathcal{H}_{s(\tau)}=\Sigma_\tau$ and $\sop$. This theorem can be used here since the data are compactly supported in $\Sigma_0$ and, as a consequence, the future of the support of the initial data does not contain the singularity $i^0$. Let be $M_{u_0}^{s}$ be the subset of $\Omega_{u_0}^+$ whose boundary consists in these hypersufaces. We have: 
\begin{gather*}
\int_{S_{u_0}^{s(\tau)}}i^\star_{S_u}(\star \hat T^aT_{ab})+\int_{\Sigma_\tau}i^\star_{\Sigma_\tau}(\star \hat T^aT_{ab})-\int_{\sop}i^\star_{\Sigma_0}(\star \hat T^aT_{ab})-=\int_{M_{u_0}^{s(\tau)}}\hat\nabla^{a}\left(\hat T^bT_{ab}\right) \mu[\hat g]
\end{gather*}
 and, using the notations in the proposition, the foliation given by $\tau$ defined by equation \eqref{parametrization} and lemma \ref{errorschwarzschild}, this becomes:
 \begin{gather*}
E(S_{u_0}^{s(\tau)})+\int_{\Sigma_\tau}i^\star_{\Sigma_\tau}(\star \hat T^aT_{ab})-\int_{\sop}i^\star_{\Sigma_0}(\star \hat T^aT_{ab})\\
 =\int_ {0}^{\tau}\Bigg(\int_{\Sigma_\tau}\bigg\{4mR^2(3+uR)\left(\partial_R\phi\right)^2
 +\left(1-12mR\right)\phi\left(u^2\partial_u \phi-2(1+uR)\partial_R \phi\right)\\+\hat T^a \hat \nabla_a b \frac{\phi^4}{4}\bigg\}(r^\ast R)^\frac{3}{2}(1-2mR)\sqrt{\frac{R}{|u|}}\ud u\wedge \ud \omega_{\mathbb{S}^2} \Bigg)\ud \tau
 \end{gather*}
 
 The error term is bounded above in absolute value; each term is evaluated separetly:
 \begin{gather*}
| (r^\ast R)^\frac{3}{2}(1-2mR)\sqrt{\frac{R}{|u|}}4mR^2(3+uR)\left(\partial_R\phi\right)^2|\\
= (r^\ast R)^\frac{3}{2}(1-2mR)(3+|u|R)(R|u|)^{1/2}R\frac{R}{|u|}\left(\partial_R\phi\right)^2\\
\leq(1+\epsilon)^{\frac{3}{2}}\cdot 1\cdot (4+\epsilon)\cdot (1+\epsilon)\frac{\epsilon}{2m}\frac{R}{|u|}\left(\partial_R\phi\right)^2\\
\lesssim \frac{R}{|u|}\left(\partial_R\phi\right)^2.
 \end{gather*}
 and 
 \begin{gather*}
|\left(1-12mR\right)(r^\ast R)^\frac{3}{2}(1-2mR)\sqrt{\frac{R}{|u|}}u^2\phi\partial_u \phi |
=|\left(1-12mR\right)(r^\ast R)^\frac{3}{2}(1-2mR)\sqrt{R|u|}\phi(u\partial_u \phi |)\\
\lesssim \frac{1}{2}(1+6\epsilon)(1+\epsilon)^\frac{3}{2}\cdot 1\cdot (1+\epsilon)\left(\phi^2 +(u\partial_u \phi)^2\right)\\
 \lesssim \phi^2+(u\partial_u \phi)^2.
\end{gather*}
The remaining term is controlled by:
\begin{eqnarray*}
|(1-12mR)(r^\ast R)^\frac{3}{2}(1-2mR)\sqrt{\frac{R}{|u|}}\phi\partial_R\phi&\leq&(1+6\epsilon)(1+\epsilon)^\frac32\phi \left(\sqrt{\frac{R}{|u|}}\partial_R\phi\right)\\
&\leq& (1+6\epsilon)(1+\epsilon)^\frac32\left(\phi^2+\frac{R}{|u|}(\partial_R\phi)^2\right).
\end{eqnarray*}
\begin{remark} This term is the main obstacle in the use of the parameter $s$: if the foliation was parametrized by $s$, this term would be replaced by:
$$\left|(1-12mR)\frac{(r^\ast R)^2(1-2mR)}{|u|}\phi \partial_R\phi \right|\leq (1+6\epsilon)(1+\epsilon)^2 \left|\frac{\phi \partial_R\phi}{u}\right|
$$
which cannot be compared to $\phi^2+\frac{R}{|u|}(\partial_R \phi)^2$.
\end{remark}
Gathering these inequalities, it remains:
\begin{gather*}
\left| \int_ {0}^{\tau}\Bigg(\int_{\Sigma_\tau}\bigg\{4mR^2(3+uR)\left(\partial_R\phi\right)^2
 +\left(1-12mR\right)\phi\left(u^2\partial_u \phi-2(1+uR)\partial_R \phi\right)\right.\\\left.+\hat T^a \hat \nabla_a b \frac{\phi^4}{4}\bigg\}(r^\ast R)^\frac{3}{2}(1-2mR)\sqrt{\frac{R}{|u|}}\ud u\wedge \ud \omega_{\mathbb{S}^2} \Bigg)\ud \tau\right|
 \lesssim \int_0^{\tau } E(\Sigma_{r})\ud r.
 \end{gather*}
Finally, the following inequalities hold:
\begin{equation}\label{ineg1111}
\begin{array}{lcl}
E(\Sigma_\tau)+E(S_{u_0}^{s(\tau)})&\lesssim& \int_0^{\tau} E(\Sigma_{r})\ud r+E(\sop)\\
E(\sop)&\lesssim& \int_0^{\tau} E(\Sigma_{r})\ud r+E(\Sigma_\tau)+E(S_{u_0}^{s(\tau)})\\
\end{array}
\end{equation} 
Since $E(S_{u_0}^{s(\tau)})$ is positive, the integral inequality holds:
$$
E(\Sigma_\tau) \lesssim \int_0^{\tau} E(\Sigma_{r})\ud r+E(\sop).
$$
Using Gronwall's lemma, we obtain:
\begin{equation}\label{energyequivalenceeq}
E(\Sigma_\tau)\lesssim E(\sop).
\end{equation}
Putting this inequality back into \eqref{ineg1111}, we obtain the first part of the inequality, for $\tau=2$:
$$
E(\scri^+_{u_0})+E(S_{u_0})\lesssim E(\sop).
$$

The other inequality is obtained by doing the same calculation between the hypersurfaces $S_{u_0,s}=\{(u,R,\omega_{\mathbb{S}^2})|u=u_0, |u|\leq sr^\ast\}$, $\mathcal{H}_{s}$ and $\sop$. Let be $M^{u_0}_{s}$ be the subset of $\Omega_{u_0}^+$ whose boundary consists of these hypersurfaces. Applying Stokes theorem, using the parametrization by $\tau$ and the previous estimates of the error term, we obtain:
$$
 E(\Sigma_\tau) \lesssim \int_{0}^{\tau }E(\Sigma_r) \ud r+E(\scri^+_{u_0}) + E(S_{u_0}^{s(\tau)})
$$
As a consequence, since the integrand in $E(S^{u_0,s(\tau)})$ is positive, the following integral inequality holds:
$$
E(\Sigma_\tau)\lesssim \int_{0}^{\tau }E(\Sigma_r) \ud r+E(\scri^+_{u_0}) + E(S_{u_0}).
$$
The use of Gronwall lemma gives the second inequality:
$$
E(\som)\lesssim E(\scri^+_{u_0}) + E(S_{u_0}).
$$ \end{proof}

\subsection{Energy estimates far from the spacelike infinity $i^0$}

The estimates which are obtained in this section are widely inspired by the work of H\"ormander \cite{MR1073287} and generalized in \cite{MR2334072} to establish the existence of solutions for the characteristic Cauchy problem for the wave equation on a curved background. The main tool consists in writing the characteristic, or weakly characteristic (that is to say is locally either spacelike or degenerate; this is also referred to as achronal) surface as the graph of a function and expressing all the relevant quantities in term of this graph. This method was also used by Mason-Nicolas in \cite{mn04} to control how spacelike surfaces converge to null infinity.

$\hat M\backslash \Omega^+_{u_0}\cap J^+(\Sigma_0)$ is divided in two parts as follows: 
\begin{itemize}
\item let $\Sigma$ be a spacelike hypersurface in $\hat M$ for the metric $\hat g$ such that $\Sigma \cap \scri^+=S_{u_0}\cap \scri ^+$ and $\Sigma$ is orthogonal to $\hat T^a$.
\item the part of $\hat M\backslash \Omega^+_{u_0}$ contained between $\Sigma_0$ and $\Sigma$, denoted by $V$;
\item and finally the future of $\Sigma$, containing $i^+$, U. The subset of its boundary in $\scri^+$ is denoted by $\scri^+_T$. 
\end{itemize}
This decomposition of the future of $\Sigma_0$ is represented in figure \ref{futuresigma}. 

\begin{figure}[ht]\label{futuresigma}
\begin{center}
\scalebox{2.1}{\input{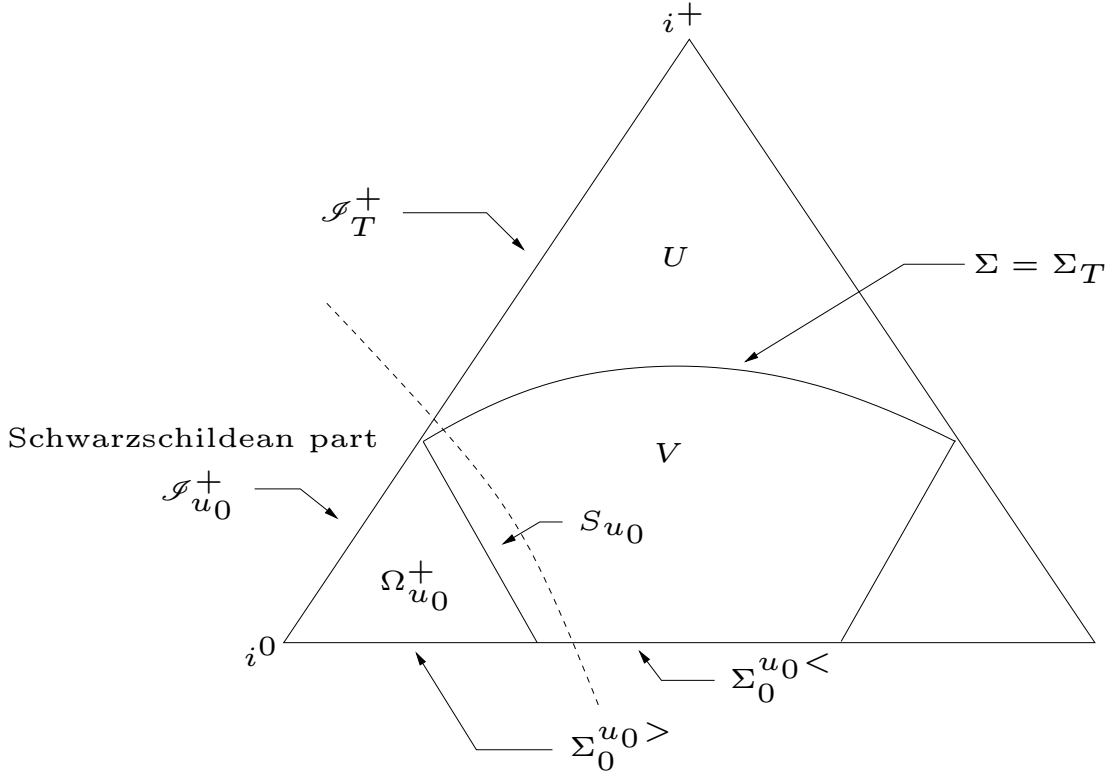}}
\end{center}
\caption{Future of $\Sigma_0$}
\end{figure}

\begin{remark}
The fact that the timelike vector field $\hat T^a$ is orthogonal to the spacelike hypersurface is an assumption which is made about this timelike vector field. It is built outside the Schwarzschildean part of the manifold as follows:
 \begin{itemize}
  \item outside $\Omega^+_{u_0}$, $\hat T^a$ is smoothly extended such that it is orthogonal to a uniformly spacelike hypersurface $\Sigma$ whose boundary is $\scri^+\cap S_{u_0}$.
 \item The intersection of $S_{u_0}$ and $\scri^+$ is the Schwarzschildean part and is given $\{u=u_0,R=0\}$. The vector field $\hat T^a$ is then orthogonal to this 2-dimensional surface.
 \end{itemize}
\end{remark}
This subsection deals with estimates on $U$. As previously said, we use here H\"ormander's technique which consists in writing $\scri$ as the graph of a function.

We consider on $U$ the flow associated with $e^a_0=\frac{\hat T^a}{\hat g_{cd}\hat T^c\hat T^d}$. Let $t$ be the time function induced by this flow. The vector field $e_0^a$ is completed in an orthonormal basis of $T\hat M$ by choosing an orthonormal basis $\{e_i^a; i=1,2,3\} $ of the spacelike foliation $\{\Sigma_t\}$ induced by $t$. For the sake of clarity, $\Sigma$ is denoted $\Sigma_T$ as corresponding to the slice $\{t=T\}$ ($T$ is chosen to be non zero, in order not to introduce confusion with $\Sigma_0$).

\subsubsection{Geometric description of $\scri^+_T$}\label{geometricdescriptioniplus}
Using the flow $\Psi_t$ associated with $e^a_0$, $\scri^+_T$ can be identified with $\Sigma_T$:
\begin{equation}
\begin{array}{ccc}
\Sigma_T&\longrightarrow &\scri^+_T\\
x & \longmapsto& \Psi_{\varphi(x)}(x)
\end{array}
\end{equation}
where $\varphi(x)$ is the time at which the curve $t\mapsto \Psi_t(x)$ hits $\scri^+_T$. $\scri^+_T$ can then be considered as defined by the graph of the function $\varphi:
x\in \Sigma_T \longmapsto \varphi(x)$.

We denote by $\nabla_i \varphi$ the derivatives of $\varphi$ with respect to the vector tangent to $\Sigma_T$ $e^a_{\textbf{i}}$ at time $T$. 
\begin{remark} 
\begin{enumerate} 
\item As noticed in the introduction, the spacetimes constructed by Chrusciel-Delay and Corvino-Schoen have the specificity that the regularity at $\scri^+\backslash\{i^0,i^+\}$ can be specified arbitrarily.  In order to insure that some geometric quantities are defined, we assumed that the manifold is $C^2$ differentiable at $\scri^+\backslash\{i^0,i^+\}$. The use of implicit function theorem then insures that the function $\varphi$ has the same regularity. 
\item The function $\varphi$ is defined on a compact set and as such admits a maximum. This maximum is denoted by $T_{max}$.
\end{enumerate}
\end{remark}
The lapse function associated with this choice of time $t$ is denoted by $N$. The metric can be decomposed as:
$$
\hat g = N^2(\ud t)^2- h_{\Sigma_t}
$$
where $h$ is a Riemannian metric on $\Sigma_t$ depending on the spacelike leaves of the foliation induced by the time function, and $N$ is the (positive) lapse function. Since the time function is built from the vector $e^a_0$, the vector field $\partial_t$ satisfies:
$$
\partial_t= N e^a_0.
$$

The following lemma describes the geometry of $\scri^+_T$ in term of the parametrization:
\begin{lemma} The vector $N^a$ defined by
$$
N^a= e_0^a-\sum_{j\in \{1,2,3\}}N\nabla_{j}\varphi e^a_{\textbf{j}}
$$  
is normal and tangent to the hypersurface $\scri^+_T$.\\ 
The set of vectors defined by, for $i\in \{1,2,3\}$,
$$
t^a_{\textbf{i}}=N\nabla_i \varphi e^a_0-e^a_{\textbf{i}}
$$
are normal to $N^a$ and, as such, forms a basis of $T\scri^+_T$.
\end{lemma} 
\begin{proof}  The fact that $N^a$ is null directly comes from the fact $N^a$ is normal to $\scri^+_T$, which is a null surface. The derivatives of $\varphi$ then satisfy:
$$
N^2\sum_{i=1,2,3}(\nabla_i \varphi )^2=1
$$
 The tangent plane to $\scri^+_T$ is given by the kernel of the differential of the application
$$
x\longmapsto (\varphi(x), x)
$$ 
which is given by 
$$
h^a\in \Sigma_T \longmapsto h^a+g_{ij}(\varphi(x),x)\nabla^i\varphi h^j \underbrace{Ne^a_0}_{\partial_t}. 
$$
It is then clear that the set of vectors
$$
t^a_{\textbf{i}}=N \nabla^i \varphi e^a_0-e^a_{\textbf{i}}
$$
forms a basis of $T\scri^+_T$. $N^a$ is then a linear combination of them:
$$
N^a=\sum_{i=1,2,3}N\nabla^i  \varphi t^a_{\textbf{i}}.
$$\end{proof}

As a direct consequence of this lemma, the vector defined by
\begin{equation*}
\tau^a=  e_0^a+\sum_{j\in \{1,2,3\}}N\nabla_{j}\varphi e^a_{\textbf{j}}
\end{equation*}
is null and transverse to $\scri^+_T$.

In order to complete the geometric description of $\scri^+_T$ and facilitate the calculation afterwards, we introduce the following objects:
\begin{itemize}
\item using the Geroch-Held-Penrose formalism, the set of two null vectors $(\tau^a, N^a)$ is completed by two normalized spacelike vectors $(v^a_1, v^a_2)$ tangent to $\scri^+_T$ to form a basis of $T\hat M$;
\item $\scri^+_T$ is endowed with the 3-form:
\begin{equation*}
\mu_{\scri}=t^1_a\wedge t^2_a\wedge t^3_a
\end{equation*}
which satisfies:
\begin{eqnarray*}
\tau_a\wedge t^1_a\wedge t^2_a\wedge t^3_a&=&(1+N^2\sum_{i=1,2,3}(\nabla^i \varphi )^2)\mu[\hat g]\\ 
&=&2\mu[\hat g];
\end{eqnarray*}
this 3-form will be used as the form of reference to calculate the energy on $\scri^+_T$.
\end{itemize}

\begin{remark} The tangent vector $n^a$ used in definition \ref{defh1} for the $H^1$-norm on $\scri^+$ is colinear to the vector $N^a$:
$$
N^a=\frac{n^a}{\hat g_{cd}\hat T^c \hat T^d}.
$$
As a consequence, the norms associated with these vector fields are equivalent.
\end{remark}
Finally, in order to prepare the estimates, the expression of the 3-form $\star \hat T^aT_{ab}$ on the surfaces $\Sigma_t$ and $\scri^+_T$ is given:
\begin{lemma}\label{energyexpression2}
The restrictions of the energy 3-form $\star \hat T^aT_{ab}$ to $\Sigma_t$, for $t$
 given in $[T, T_{max}]$, and $\scri^+_T$ are given by, respectively:
 \begin{equation*}
  i^\star_{\Sigma_t} \left( \star \hat T^aT_{ab}\right)=||\hat T^a||\left(\frac{1}{2}\sum_{i=0}^4 (e^a_{\textbf{i}}\nabla_a\phi)^2+\frac{\phi^2}{2}+b\frac{\phi^4}{4} \right) e_a^1\wedge e_a^2\wedge e^3_a
 \end{equation*}
 and 
 \begin{equation*}
 i^\star_{\scri^+_T} \left( \star \hat T^aT_{ab}\right)=\frac{||\hat T^a||}{4}\left((N^a\hat\nabla_a\phi)^2+(v_1^a\hat \nabla_a\phi)^2+(v_2^a\hat \nabla_a\phi)^2)+\frac{\phi^2}{2}\right) t_a^1\wedge t_a^2\wedge t^3_a
 \end{equation*}
\end{lemma}
\begin{remark}
\begin{enumerate}
\item The expression which is given for the energy form on $\scri$ is consistent with the one obtain by H\"ormander, since it only depends on tangential derivatives to null infinity. It is nonetheless not identical: the result of H\"ormander is similar to a calculation made with respect to the Riemannian metric obtained from $\hat g$ and the timelike vector field $\hat T^a$.
\item The part of the energy form on $\scri$ given by
$
(v_1^a\hat \nabla_a\phi)^2+(v_2^a\hat \nabla_a\phi)^2
$
is usually interpreted as the norm of the gradient on a 2 sphere, even though the distribution of 2-planes $\text{Span}(v_1, v_2)$ is not integrable.
\end{enumerate}
\end{remark}
\begin{proof} 
Using the basis $(e^a_{\textbf{i}})_{i=0,\dots,3}$ which is adapted to the foliation, the energy 3-form over $\Sigma_t$ can easily be calculated:
\begin{eqnarray*}
i^\star_{\Sigma_t} \left( \star \hat T^aT_{ab}\right)&=&\left(\hat T^a\hat \nabla_a \phi\right)  i^\star_{\star \Sigma_t}(\hat \nabla_b \phi)+\left(-\frac{1}{2}\hat g_{cd}\hat\nabla^c\phi\hat\nabla^d\phi+\frac{\phi^2}{2}+b\frac{\phi^4}{4} \right)i^\star_{\Sigma_t}(\star T_b)\\
\end{eqnarray*}
Since the vectors $\{e^a_{\textbf{i}}\}_{i=1,2,3}$ are tangent to the hypersufaces $\Sigma_t$, we obtain:
\begin{itemize}
\item $i^\star_{\Sigma_t}(\star T_a)=||\hat T^a|| e^a_0 \lrcorner \mu[\hat g]=||\hat T^a||e_a^1\wedge e_a^2\wedge e^3_a$
\item $ i^\star_{\Sigma_t}(\hat \nabla_b \phi)=e^b_0\nabla_b\phi \left(e^a_0 \lrcorner \mu[\hat g]\right)=(e^b_0\nabla_b\phi) e_a^1\wedge e_a^2\wedge e^3_a $
\end{itemize}  
and 
\begin{equation*}
 i^\star_{\Sigma_t} \left( \star \hat T^aT_{ab}\right)=  ||\hat T^a||\left(\frac{1}{2}\sum_{i=0}^3 (e^a_{\textbf{i}}\nabla_a\phi)^2+\frac{\phi^2}{2}+b\frac{\phi^4}{4} \right) e_a^1\wedge e_a^2\wedge e^3_a.
\end{equation*}

To calculate the restriction of the energy form to $\scri^+_T$, the vector fields $\hat T^a$ and $\hat \nabla \phi$ are split over the basis $(N^a, \tau^a, v^a_1, v^a_2)$:
\begin{eqnarray}
\hat T^b&=&\frac{||\hat T^b||}{2}\left(N^b+\tau^b\right)\label{expressiont}\\
\hat \nabla^b\phi &=& \frac{N^a\hat \nabla_a \phi}{2}\tau^b+\frac{\tau^a\hat \nabla_a \phi}{2}N^b-(v_1^a\hat \nabla_a\phi)v^b_1+(v_2^a\hat \nabla_a\phi) v^b_2\nonumber\\
\hat \nabla_c\phi\hat \nabla^c\phi&=&N^a\hat \nabla_a \phi\tau^a\hat \nabla_a \phi-(v_1^a\hat \nabla_a\phi)^2-(v_2^a\hat \nabla_a\phi)^2.\label{gradientphi}
\end{eqnarray}
The only relevant terms in the expressions of  $i^\star_{\scri^+}(\star \hat \nabla_b \phi)$ and $i^\star_{\scri^+}(\hat T_b)$ are those which are transverse to $\scri$, so that:
\begin{equation*}
\begin{array}{cc}
\begin{array}{lcl}
i^\star_{\scri^+}(\hat T_b)&=&\frac{||\hat T^a||}{2}i^\star_{\scri}(\tau_a)\\
&=&\frac{||\hat T^a||}{2}\tau^a\lrcorner \mu[\hat g]\\
&=&\frac{||\hat T^a||}{4} t_a^1\wedge t_a^2\wedge t^3_a
\end{array}&
\begin{array}{lcl}
i^\star_{\scri^+}(\star \hat \nabla_b \phi)&=& \frac{N^a\hat \nabla_a \phi}{2}i^\star_{\scri}(\tau_a)\\
&=&\frac{N^a\hat \nabla_a \phi}{2} \tau^a\lrcorner \mu[\hat g]\\
&=& \frac{N^a\hat \nabla_a \phi}{4} t_a^1\wedge t_a^2\wedge t^3_a
\end{array}
\end{array}
\end{equation*}
and, finally, using equations \eqref{expressiont} and \eqref{gradientphi}, since the function $b$ vanishes at $\scri$:
\begin{eqnarray*}
 i^\star_{\scri^+_T} \left( \star \hat T^aT_{ab}\right)&=&\left(e^a_0\hat\nabla_a\phi N^a\hat \nabla_a \phi-\frac12\hat\nabla_c\phi\hat \nabla^c\phi+\frac{\phi^2}{2}\right)\frac{||\hat T^a||}{4} t_a^1\wedge t_a^2\wedge t^3_a\\
 &=&\frac{||\hat T^a||}{8}\left((N^a\hat\nabla_a\phi)^2+(v_1^a\hat \nabla_a\phi)^2+(v_2^a\hat \nabla_a\phi)^2)+\phi^2\right) t_a^1\wedge t_a^2\wedge t^3_a.
\end{eqnarray*}\end{proof}

\subsubsection{Energy estimates on $U$}\label{estimatesonu}
The techniques used in this section are exactly the same as in the other section: they rely on Gronwall lemma and Stokes theorem carefully applied to the 3-form $\star \hat T^aT_{ab}$.

The first step of the proof consists in establishing a decay result of the energy on slices $\{t=constant\}$.
\begin{proposition}\label{energyequivalence}
Let $E(\Sigma_t)$ be the energy on the slice $\Sigma_t$:
$$
E(\Sigma_t)=\int_{\Sigma_t} \left(\frac12\sum_{i=0}^3 (e^a_{\textbf{i}}\nabla_a\phi)^2+ \frac{\phi^2}{2}+b\frac{\phi^4}{4}\right)\mu_{\Sigma_t}
$$
where $\mu_{\Sigma_t}=e_a^1\wedge e_a^2\wedge e^3_a$.\\
Then this energy satisfies:
$$
E(\Sigma_t)\approx \int_{\Sigma_t} i^\star_{\Sigma_t} \left( \star \hat T^aT_{ab}\right) 
$$
and for $s$ and $t$ two real numbers in $[T, T_{max}]$, such that $t\geq s$:
$$
E(\Sigma_t)\lesssim E(\Sigma_s)
$$
\end{proposition}
\begin{remark} \label{energyreference} In this section, the calculations are made with respect to $E(\Sigma_T)$ rather than $\int_{\Sigma_t} i^\star_{\Sigma_t} \left( \star \hat T^aT_{ab}\right)$.
\end{remark}
\begin{proof} 
Since $\hat T^a$ is a non-vanishing timelike vector field on the compact $\hat M$, there exists a positive constant $C$ such that:
\begin{equation*} 
\frac{1}{C}\leq  ||\hat T^a||\leq C
\end{equation*}
and, as a consequence of lemma \ref{energyexpression2}, the energy $E(\Sigma_t)$ is equivalent to $\int_{\Sigma_t} i^\star_{\Sigma_t} \left( \star \hat T^aT_{ab}\right)$ since:
\begin{equation*}
\frac{1}{2C}E(\Sigma_t)\leq \int_{\Sigma_t} i^\star_{\Sigma_t} \left( \star \hat T^aT_{ab}\right)\leq CE(\Sigma_t).
\end{equation*}
\begin{remark}\label{energyscri+}
The same result holds for the energy on $\scri^+_T$, that it to say that:
$$
\int_{\scri^+_T} i^\star_{\scri^+_T} \left( \star \hat T^aT_{ab}\right)\approx\int_{\scri^+_T}\left((N^a\hat\nabla_a\phi)^2+(v_1^a\hat \nabla_a\phi)^2+(v_2^a\hat \nabla_a\phi)^2)+\frac{\phi^2}{2}\right) t_a^1\wedge t_a^2\wedge t^3_a 
$$
We denote by $E(\scri^+_T)$ the right-hand side of this equation. The expression of this energy is not intrisic, since it depends, on the one hand, on the parametrization of $\scri^+$ by the function $\varphi$ and, on the other hand, on the choice of a basis.  The energy used by H\"ormander has the same property that it depends on the graph and on the chosen coordinates.
\end{remark}

We assume here that $t>s$ and apply Stokes theorem between the surfaces $\Sigma_t$ and $\Sigma_s$. The part of $\scri^+$ between the time $t$ and $s$ is denoted $\scri^t_s$ and the part of $U$ between $\Sigma_t$ and $\Sigma_s$, $U_s^t$. The following equality holds:
\begin{gather*}
\int_{\Sigma_t} i^\star_{\Sigma_t} \left( \star \hat T^aT_{ab}\right)+\int_{\scri^t_s} i^\star_{\Sigma_t} \left( \star \hat T^aT_{ab}\right)-\int_{\Sigma_s} i^\star_{\Sigma_s} \left( \star \hat T^aT_{ab}\right)\\=\int_{U_s^t}\left(\hat\nabla^{(a}\hat T^{b)}T_{ab}+(1-\frac16\scal_{\hat g})\phi\hat T^a\hat \nabla_a\phi+\hat T^a\hat \nabla_a \frac{\phi^4}{4}\right)\mu[\hat g]
\end{gather*}
As it was noticed in lemma \ref{energyexpression2}, the integral over $\scri^+$ of the energy 3-form restricted to $\scri$ is positive. So the following inequality holds:
\begin{gather*}
\left|\int_{\Sigma_t} i^\star_{\Sigma_t} \left( \star \hat T^aT_{ab}\right)+\int_{\scri^t_s} i^\star_{\Sigma_t} \left( \star \hat T^aT_{ab}\right)-\int_{\Sigma_s} i^\star_{\Sigma_s} \left( \star \hat T^aT_{ab}\right)\right|\\\leq\int_{U_s^t}\left|\left(\hat\nabla^{(a}\hat T^{b)}T_{ab}+(1-\frac16\scal_{\hat g})\phi\hat T^a\hat \nabla_a\phi+\hat T^a\hat \nabla_a b \frac{\phi^4}{4}\right)\right|\mu[\hat g]
\end{gather*}
and, as a consequence, 
\begin{gather*}
\int_{\Sigma_t} i^\star_{\Sigma_t} \left( \star \hat T^aT_{ab}\right)+\int_{\scri^t_s} i^\star_{\Sigma_t} \left( \star \hat T^aT_{ab}\right)\\\leq \int_{U_s^t} \left|\hat\nabla^{(a}\hat T^{b)}T_{ab}\right|+\left|(1-\frac16\scal_{\hat g})\phi\hat T^a\hat \nabla_a\phi\right|+\left|\hat T^a\hat \nabla_a b\right| \frac{\phi^4}{4}\mu[\hat g]+\int_{\Sigma_s} i^\star_{\Sigma_s} \left( \star \hat T^aT_{ab}\right).
\end{gather*}
Since $\int_{\scri^t_s} i^\star_{\Sigma_t}$ is non-negative (see remark \ref{energyscri+}) and
$$
\int_{\Sigma_t} i^\star_{\Sigma_t} \left( \star \hat T^aT_{ab}\right)\approx E(\Sigma_r),
$$
  it remains:
$$
E(\Sigma_t)\lesssim \int_{U_s^t} \left|\hat\nabla^{(a}\hat T^{b)}T_{ab}\right|+\left|(1-\frac16\scal_{\hat g})\phi\hat T^a\hat \nabla_a\phi\right|+\left|\hat T^a\hat \nabla_a b\right| \frac{\phi^4}{4}\mu[\hat g] +E(\Sigma_s).
$$ 
Since $\overline{U}$ is compact, there exists a contant $c$ depending on $\hat \nabla^a \hat T^b$, $\scal_{\hat g}$ and $\hat T^a\hat \nabla_a b $ which controls each term in the error 
\begin{equation*}
\int_{U_s^t}\left(\hat\nabla^{(a}\hat T^{b)}T_{ab}+(1-\frac16\scal_{\hat g})\phi\hat T^a\hat \nabla_a\phi+\hat T^a\hat \nabla_a b \frac{\phi^4}{4}\right)\mu[\hat g]
\end{equation*}
in function of the energy on a slice at time $r$. To perform such an estimate, the 2-form $\hat\nabla^{(a}\hat T^{b)}$ is split over the orthonormal basis $(e^a_{\textbf{i}})_{i=0,\dots,4}$; each of the components is bounded by $c$. The remaining terms, when contracting with $T_{ab}$, are either products of derivatives or functions which can be estimated by $\phi^2$ or $b\phi^4$ and their derivatives.

Finally, the volume form $\mu[\hat g]$ is decomposed over the basis $(e^a_{\textbf{i}})_{i=0,\dots,4}$ to obtain:
\begin{equation*}
E(\Sigma_t) \lesssim \int_s^t E(\Sigma_r)\ud r+ E(\Sigma_s)
\end{equation*}
where the form $\ud r$ is $e^0_a$. So, applying Gronwall lemma, we obtain, since we are working in finite time:
$$
E(\Sigma_t)\lesssim E(\Sigma_s).
$$\end{proof}

A straightforward consequence of this proposition is that all the energies on slices are controlled by the energy on $\Sigma_T$. This is a necessary step to establish the following proposition:
\begin{proposition}\label{aprioriestimate2}
The energies on $\scri^+_T$ and on $\Sigma_T$ are equivalent:
$$
E(\scri^+_T)\approx E(\Sigma_T).
$$
\end{proposition}
\begin{proof} The proof is based on the use of Stokes theorem. We denote by, for $t$ between $T$ and $T_{max}$:
\begin{itemize}
\item $U_t$ the part of $U$ for time greater than $t$;
\item $\scri^+_t$ the part of $\scri^+$ for time greater than $t$.
\end{itemize}
Stokes theorem is used between the hypersurfaces $\scri^+_t$ and $\Sigma_t$:
\begin{gather*}
\int_{\scri_t^+} i^\star_{\scri^+} \left( \star \hat T^aT_{ab}\right)-\int_{\Sigma_t}i^\star_{\Sigma_t} \left( \star \hat T^aT_{ab}\right)\\=\int_{U_t}\left(\hat\nabla^{(a}\hat T^{b)}T_{ab}+(1-\frac16\scal_{\hat g})\phi\hat T^a\hat \nabla_a\phi+\hat T^a\hat \nabla_a \frac{\phi^4}{4}\right)\mu[\hat g]
\end{gather*}
Using exactly the same estimate as in proposition \ref{energyequivalence} of the error term, we obtain:
$$
\left|\int_{\scri_t^+} i^\star_{\scri^+} \left( \star \hat T^aT_{ab}\right)-\int_{\Sigma_t}i^\star_{\Sigma_t} \left( \star \hat T^aT_{ab}\right)\right| \lesssim \int_t^{T_{max}}E(\Sigma_r)\ud r.
$$
As a consequence, the two following inequalities hold:
\begin{equation}\label{inegaliteu1}
\int_{\scri_t^+} i^\star_{\scri^+} \left( \star \hat T^aT_{ab}\right)-\int_{\Sigma_t}i^\star_{\Sigma_t} \left( \star \hat T^aT_{ab}\right)\lesssim \int_t^{T_{max}}E(\Sigma_r)\ud r.
\end{equation}
and 
\begin{equation}\label{inegaliteu2}
\int_{\Sigma_t}i^\star_{\Sigma_t} \left( \star \hat T^aT_{ab}\right)-\int_{\scri_T^+} i^\star_{\scri^+} \left( \star \hat T^aT_{ab}\right)\lesssim \int_t^{T_{max}}E(\Sigma_r)\ud r,
\end{equation}
since 
$$
\int_{\scri_T^+} i^\star_{\scri^+} \left( \star \hat T^aT_{ab}\right)\geq \int_{\scri_t^+} i^\star_{\scri^+} \left( \star \hat T^aT_{ab}\right).
$$

We first deal with inequality \eqref{inegaliteu1}. Since, according to proposition \ref{energyequivalence}, all the energies on a slice $\{t=constant\}$ are controlled by $E(\Sigma_T)$ for $t\geq T$, the integral $ \int_t^{T_{max}}E(\Sigma_s)\ud r$ satisfies:
$$
 \int_t^{T_{max}}E(\Sigma_r)\ud r \lesssim E(\Sigma_t).
$$
Using inequality \eqref{inegaliteu1}, a straightforward consequence is:
$$
\int_{\scri_T^+} i^\star_{\scri^+} \left( \star \hat T^aT_{ab}\right) \lesssim E(\Sigma_T)
$$
and with remark \ref{energyscri+}:
$$
E(\scri^+_T)\lesssim E(\Sigma_T).
$$

On the other hand, to obtain the second inequality, we use Gronwall lemma in inequality \eqref{inegaliteu2}; this gives:
$$
E(\Sigma_t)\lesssim \int_{\scri^+} i^\star_{\scri^+} \left( \star \hat T^aT_{ab}\right)
$$
and, consequently, for $t=T$:
$$
 E(\Sigma_T)\lesssim E(\scri^+_T).
$$\end{proof}

\subsection{Estimates on $V$}\label{estimatesonv}

The geometric situation in this section is almost the same as in the previous one since the hypersurface $S_{u_0}$ is known to be null, the only difference being that an additional term comes from the boundary of the future of $\som$. The energy estimates will then be obtained in exactly the same way. Nonetheless, we wish to keep the term with an energy on the hypersurface $S_{u_0}$, in order to compare these terms with the inequalities obtained in section \ref{schwarzschildpart}.

It is clear that the time function which was defined in the previous section cannot be used here since $\hat T^a$ is not necessarily orthogonal to $\Sigma_0$. We now consider another time function $\tau$, defined in $\overline{V}$ (or in a neighborhood of $V$, as done in remark \ref{constructiontime} below) such that the hypersurface $\{\tau=0\}$ corresponds to $\Sigma_0$ and the hypersurface $\{\tau=1\}$ to $\Sigma=\Sigma_T$. We consider on $V$ the orthonormal basis $(e^a_0,e^a_{\textbf{i}})_{i=1,2,3}$ such that:
$$
e^a_0=\frac{\nabla^a \tau}{\hat g_{cd}\nabla^c \tau\nabla^d\tau}.
$$
By construction, the vector fields $(e^a_{\textbf{i}})_{i=1,2,3}$ are tangent to the time slices.

We introduce the following function:
$$
\alpha=1+\sum_{i=1,2,3}(\hat g_{cd}f^ce^d_\textbf{i})^2\geq 1
$$
where $f^c$ is the normalization with respect to the metric $\hat g$ of the vector field $\hat T^a$.

\begin{remark}\label{constructiontime}
\begin{itemize}
\item Such a time function $\tau$ can be constructed as follows: since $\hat M$ is globally hyperbolic, there exists a time function on $\hat M$. Let $\tilde \tau$ be a time function. Let $\Psi_{\tilde \tau}$ be the flow associated with $\tau$ and let $V_0$ be the preimage of $\Sigma$ on $\Sigma_0$ by the flow. We then obtain a diffeomorphism defined by:
$$
\begin{array}{ccc}
V_0&\longrightarrow &\Sigma\\ 
x&\longmapsto& \Psi_{\phi(x)}(x)
\end{array}
$$
where $\phi(x)$ is the time at which the curve $\tilde \tau\mapsto \Psi_{\tilde \tau}(x)$ hits $\Sigma$. The new time function $\tau$ is then defined as: let $p$ be a point lying between $V_0$ and $\Sigma$, $p$ is written $\Psi_{\tilde \tau}(x)$; then 
$$
\tau (p)=\frac{\tilde \tau (p)}{\phi(x)}
$$
satisfies the required assumption.
\item $\hat T^a$ and $e^a_0$ are both uniformly timelike. Therefore, the scalar product 
$$
\beta=\hat g_{cd}\hat f^c e^a_0
$$
defines a positive function over $\overline{V}$ since $T^a$ and $e^a_0$ are both future directed and timelike over a compact set.
\end{itemize}
\end{remark}

The following notations will be used, in order to be coherent with section \ref{schwarzschildpart}:
\begin{itemize}
\item the section of the initial data surface below $u_0$ is denoted $\som$;
\item as previously introduced, the slices $\{t=contant\}$ in $V$ are denoted $\Sigma_t$;
\item the part $V$ between time $t$ and $s$ (with $t<s$) is denoted by $V^s_t$;
\item the part of $S_{u_0}$ between time $t$ and $s$ (with $t<s$) is denoted by $S^s_t$;
\end{itemize}
The expression of the energy on $S_{u_0}$ are the same as the one defined in \ref{schwarzschildpart} (see proposition \ref{expressionschwenergie}).

Following the method already used in this paper, a geometric description of the energy 3-form is given and an equivalence result of the integral of the 3-form with an well chosen energy is established:
\begin{proposition}\label{eqenergie3}
The restriction of the energy 3-form to $\Sigma_t$ is given by:
\begin{gather*}
i^\star_{\Sigma_t}\left(\star\hat T^aT_{ab}\right)=
\Bigg\{\frac{(f^c\hat \nabla_c\phi)^2}{2(1+\sum_{i=1,2,3}(\hat g_{cd}f^ce^dc_i)^2)}\\+\frac{1}{2}\left(\sum_{i=1,2,3}\left(1-\frac{(\hat g_{cd}f^ce^dc_i)^2}{1+\sum_{i=1,2,3}(\hat g_{cd}f^ce^dc_i)^2}\right)(e^a_{\textbf{i}}\hat \nabla_a\phi)^2\right)+\frac{\phi^2}{2}+b\frac{\phi^4}{4}\Bigg\}\frac{\hat g_{cd}\hat T^c\hat T^d}{\hat g_{cd}\hat T^ce^d_0}e^a_1\wedge e^a_2\wedge e^a_3
\end{gather*}
and, as a consequence, the following equivalence holds, for all $t$ in $[0,1]$:
$$
\int_{\Sigma_t}i^\star_{\Sigma_t}\left(\star\hat T^aT_{ab}\right)\approx\int_{\Sigma_t}\left((f^c\hat \nabla_c\phi)^2+\sum_{i=1,2,3}(e^a_{\textbf{i}}\hat \nabla_a\phi)^2+\frac{\phi^2}{2}+b\frac{\phi^4}{4}\right)e^a_1\wedge e^a_2\wedge e^a_3.
$$
We denote by $E(\Sigma_t)$ this energy:
$$
E(\Sigma_t)=\int_{\Sigma_t}\left((f^c\hat \nabla_c\phi)^2+\sum_{i=1,2,3}(e^a_{\textbf{i}}\hat \nabla_a\phi)^2+\frac{\phi^2}{2}+b\frac{\phi^4}{4}\right)e^a_1\wedge e^a_2\wedge e^a_3.
$$
\end{proposition} 
\begin{remark}
This proposition together with proposition \ref{energyequivalence} states that the energy on a spacelike slice for two uniformly timelike (for the metric $\hat g$) vector fields are equivalent. This justifies that we write in the same way the energy in proposition \ref{energyequivalence} and in this proposition.
\end{remark}
\begin{proof}  The proof is essentially the same as the case where the vector field $\hat T^a$ is orthogonal to the foliation. It is given in the appendix (see proposition \ref{eqenergie333})\end{proof}

From now on,  the strategy is exactly the same as in subsection \ref{estimatesonu}. The fact that the energy over $\Sigma$ dominates the energies on the slices $\Sigma_t$ is established:
\begin{proposition} \label{refenergy4}

The following estimate holds, for all $t$ in $[0,1]$:
$$
E(\Sigma_t)\lesssim E(\Sigma)
$$

\end{proposition}
\begin{proof} Let $t$ be in $[0,1]$. The energy 3-form is integrated over the surfaces $\Sigma$, $\Sigma_t$ and $S_t^1$; Stokes theorem then gives:
\begin{gather*}
\int_{\Sigma}i^\star_\Sigma\left(\star\hat T^aT_{ab}\right)- \int_{S_t^1}i^\star_{S_{u_0}}\left(\star\hat T^aT_{ab}\right)
-\int_{\Sigma_t}i^\star_{\Sigma_t}\left(\star\hat T^aT_{ab}\right)\\
=\int_{V_t^1}\left(\hat\nabla^{(a}\hat T^{b)}T_{ab}+(1-\frac16\scal_{\hat g})\phi\hat T^a\hat \nabla_a\phi+\hat T^a\hat \nabla_a \frac{\phi^4}{4}\right)\mu[\hat g].
\end{gather*}
As in the proof of proposition \ref{aprioriestimate2}, the error term is estimated considering that $\nabla^{(a}\hat T^{b)}$ has bounded coefficients in the given basis, that $\scal_{\hat g}$ is bounded and using the behavior of $b$. The volume form is decomposed on the basis $(\ud t, e^\textbf{i}_a)_{1,2,3}$ as:
$$
\mu[\hat g]= \frac{\ud  t\wedge e^1_a\wedge e^2_a \wedge e^3_a}{e^a_0 \hat \nabla_a t},
$$
where $e_0^a\hat \nabla_a t$ is a positive function and, as a consequence, is bounded above and below by two positive constants.

 We obtain, when contracting the volume form with $e^a_0$ in order to integrate in time:
 $$
 \left|\int_{V_t^1}\left(\hat\nabla^{(a}\hat T^{b)}T_{ab}+(1-\frac16\scal_{\hat g})\phi\hat T^a\hat \nabla_a\phi+\hat T^a\hat \nabla_a \frac{\phi^4}{4}\right)\mu[\hat g]\right| \lesssim \int_{t}^1E(\Sigma_t)\ud t.
 $$
Using the energy equivalence proved in proposition \ref{eqenergie3}, the following inequality then holds:
 \begin{eqnarray}
E(S^1_t)+E(\Sigma_t)&\lesssim&\int_{t}^1E(\Sigma_t)\ud t+E(\Sigma).\label{ineg212}
 \end{eqnarray}

Since $E(S^1_t)$ is non negative (see lemma \ref{schwarzestimates} and proposition \ref{expressionschwenergie}), \eqref{ineg212} turns into the integral inequality:
$$
E(\Sigma_t)\lesssim\int_{t}^1E(\Sigma_t)\ud t+E(\Sigma).
$$
Using Gronwall's lemma, we get:
\begin{equation*}
E(\Sigma_t)\lesssim E(\Sigma).
\end{equation*}\end{proof}

Finally, the following proposition holds: 
\begin{proposition}\label{aprioriestimate3}
The following estimates are satisfied on $V$:
$$
E(\Sigma)\approx E(S_{u_0})+E(\som)
$$
\end{proposition}
\begin{proof} The energy 3-form is integrated over the surfaces $\Sigma_0$, $\Sigma_t$ and $S_0^t$; the application of Stokes theorem gives:
\begin{gather*}
\int_{\Sigma_t}i^\star_{\Sigma_t}\left(\star\hat T^aT_{ab}\right)- \int_{S_0^t}i^\star_{S_{u_0}}\left(\star\hat T^aT_{ab}\right)
-\int_{\Sigma_0}i^\star_{\Sigma_0}\left(\star\hat T^aT_{ab}\right)\\
=\int_{V_0^t}\left(\hat\nabla^{(a}\hat T^{b)}T_{ab}+(1-\frac16\scal_{\hat g})\phi\hat T^a\hat \nabla_a\phi+\hat T^a\hat \nabla_a \frac{\phi^4}{4}\right)\mu[\hat g].
\end{gather*}
The error term satisfies the same inequality as in proposition \ref{refenergy4}:
\begin{equation*}
 \left|\int_{V_0^t}\left(\hat\nabla^{(a}\hat T^{b)}T_{ab}+(1-\frac16\scal_{\hat g})\phi\hat T^a\hat \nabla_a\phi+\hat T^a\hat \nabla_a \frac{\phi^4}{4}\right)\mu[\hat g]\right| \lesssim \int_{0}^t E(\Sigma_s)\ud s.
\end{equation*}
As a consequence, the two following inequalities hold:
\begin{eqnarray}
E(S^t_0)+E(\Sigma_0)&\lesssim&\int_{0}^tE(\Sigma_t)\ud t+E(\Sigma_t)\label{ineg214}\\
E(\Sigma_t)&\lesssim&\int_{0}^tE(\Sigma_t)\ud t+E(S^t_0)+E(\Sigma_0).\label{ineg215}
\end{eqnarray}

The right-hand side of inequality \eqref{ineg214} is estimated via proposition \ref{refenergy4} as:
$$
\int_{0}^tE(\Sigma_t)\ud t\lesssim E(\Sigma)
$$
and, as a consequence, for $t=1$, the first part of the equivalence can be stated:
$$
E(S_{u_0})+E(\Sigma_0)\lesssim E(\Sigma_1).
$$

Using the positivity of $E(S^t_0)$,  inequality \eqref{ineg215} becomes:
$$
E(\Sigma_t)\lesssim\int_{0}^tE(\Sigma_t)\ud t+E(S_{u_0})+E(\Sigma_0).
$$ 
Using Gronwall's lemma and setting $t=1$, the second part of the equivalence is obtained:
$$
E(\Sigma)\lesssim E(S_{u_0})+E(\Sigma_0).
$$\end{proof}

\subsection{Final estimates}

Finally, using the three propositions \ref{aprioriestimates1}, \ref{aprioriestimate2} and \ref{aprioriestimate3}, the following a priori global estimates hold:
\begin{theorem}\label{aprioriestimatesfinal}
Let $u$ be a smooth solution with compactly supported data of the nonlinear wave equation:
$$
\square u+ \frac16 \scal_{\hat g} u +b u^3=0.
$$
Then, the a priori estimate holds:
$$
E(\Sigma_0)\approx E(\scri^+).
$$
\end{theorem}
\begin{remark} 
\begin{enumerate} 
\item The solution of the wave equation is assumed to be smooth in order to avoid the problem of defining trace operators for weak solutions of the equation. Nonetheless, using a usual trace theorem, as soon as $u$ is in $H^\frac32(\hat M)$, its trace over $\scri^+$ and $\Sigma_0$ is well defined.
\item Furthermore, in the framework of a (characteristic) Cauchy problem, it is know that the solution is in $H^1(\hat M)$. Using the same theorem of existence of trace operators, $u$ is only in $H^\frac12(\Sigma_0)$ or $H^\frac12(\scri^+)$, which is not sufficient to write such estimates. It will be shown in section \ref{traceopsection} that these operators are well defined and with values in $H^1(\Sigma_0)$ or $H^1(\scri^+)$. 
\end{enumerate}
\end{remark}
\begin{proof} Let us consider the hypersurface $\Sigma_0$. This hypersurface is split in $\som$ and $\sop$:
$$
E(\Sigma_0)=E(\som)+E(\sop).
$$ 
Using proposition \ref{aprioriestimates1}:
$$
E(\sop)\lesssim E(\scri^+_{u_0})+E(S_{u_0}),
$$
proposition \ref{aprioriestimate3}:
$$
E(S_{u_0})+E(\som)\lesssim E(\Sigma_T),
$$
and proposition \ref{aprioriestimate2}:
$$
E(\Sigma_T)\lesssim E(\scri^+_T),
$$
we obtain the first part of the apriori estimate:
$$
E(\Sigma_0)\lesssim E(\scri^+_T)+E(\scri^+_{u_0})=E(\scri^+).
$$
Conversely, let us consider 
$$
E(\scri^+)=E(\scri^+_{u_0})+E(\scri^+_T).
$$
Using proposition \ref{aprioriestimate2}:
$$
E(\scri^+_T) \lesssim E(\Sigma_T),
$$
proposition \ref{aprioriestimate3}:
$$
 E(\Sigma_T)\lesssim E(S_{u_0})+E(\som),
$$
and proposition \ref{aprioriestimates1}:
$$
E(S_{u_0})+E(\scri^+_{u_0})\lesssim E(\sop),
$$
we get the other side of the inequality:
$$
E(\scri^+)\lesssim E(\sop)+E(\som)=E(\Sigma_0).
$$\end{proof}

\section{Goursat problem on $\scri^+$}

We show in this section that there exists a unique solution to the Goursat problem on $\scri^+$ for characteristic data in $H^1(\scri^+)$:
\begin{equation}\label{goursatproblem}\left\{
\begin{array}{l}
\hat \square \phi+ \frac16 \scal_{\hat g}\phi+b\phi^3=0\\
\phi\big|_{\scri^+}=\theta\in H^1(\scri^+).\\
\end{array}\right.
\end{equation}
It is known (see \cite{MR1073287}) that the linear Goursat problem on a smooth weakly hypersurface admits a global solution; nonetheless, due to technical problems coming from the singularity at $i^0$ (essentially the fact that some Sobolev embeddings are not valid), the existence of a global solution must be justified carefully.

The proof of the existence for problem \eqref{goursatproblem} is made in three steps:
\begin{enumerate}
\item considering two solutions of the wave equation, estimates on the difference of these solutions are established; the main technical problem, which consists in obtaining uniform Sobolev estimates of the nonlinearity, is encountered and solved in section \ref{partie20}. The proof is based on an estimate of the norm of a certain Sobolev embedding and on the decay assumption near $i^\pm$ of the function $b$.
\item Let $\mathcal{S}$ be a uniformly spacelike hypersurface for the metric $\hat g$ in the future of $\Sigma_0$ and close enough to $\scri^+$. The existence of solutions for problem \eqref{goursatproblem} with characteristic data whose compact support contains neither $i^0$ nor $i^+$ is obtained in the future of $\mathcal{S}$ for small data in section \ref{partie21} by slowing down the propagation of waves.
\item Then this solution is extended down to $\Sigma_0$, by solving a Cauchy problem on $\mathcal{S}$ and using a density result in section \ref{partie22}.
\item Finally, using estimates for the propagator between $\scri^+$ and $\Sigma_0$ obtained in section \ref{partie20}, the result is extended to $H^1(\scri^+)$ for small data in section \ref{partie23}.
\end{enumerate}

\subsection{Continuity result}\label{partie20}

This section is devoted to the proof of a continuity result in function of the characteristic data, although an existence theorem has not yet been stated. Consequences of these estimates will be required to obtain the well-posedness of problem \eqref{goursatproblem}.

\subsubsection{Sobolev embeddings over a foliation}

One of the main problems when dealing with a characteristic Cauchy problem for a nonlinear wave equation is the fact that, when using a spacelike foliation of the interior of a light cone, with leaves that shrink when approaching the vertex of the cone, the nonlinearity cannot be controlled uniformly on these leaves via Soblev embeddings. More precisely, it is indeed known (see \cite{MR1688256}) that, when dealing with a 3-dimensional Riemannian manifold with boundary $(M,g)$, the following inequality holds: there exist a constant $K$, depending only on the dimension of the manifold, a real number $\epsilon$, and a constant $B$, such that, for any function in $H^1(M)$:
\begin{equation}\label{sobolev111}
||u||^2_{L^6(M)}\leq (K+\epsilon)||\nabla u||^2_{L^2(M)}+B||u||^2_{L^2(M)}
\end{equation}
where the second best constant $B$ satisfies the inequality:
$$
B\geq \left(\text{Vol}(M,g)\right)^{-\frac{2}{3}}.
$$
As a consequence, when considering a foliation whose leaves are shrinking to a point, it is not possible to have uniform Sobolev embeddings over such a foliation.

This kind of embedding was thouroughly studied by M\"uller zum Hagen \emph{et al.} (see for instance \cite{MR633558} and \cite{MR0606056}) and extended by Dossa (see \cite{MR2015759}) with applications to the characteristic Cauchy problem for different kinds of equations (system of scalar nonlinear hyperbolic equations in \cite{MR2168739}, Einstein--Yang-Mills equations in \cite{MR2586740}, for instance). M. Dossa especially states that, following M\"uller zum Hagen, the Sobolev constant associated with the embedding of $H^1$ into $L^6$ satisfies the following result (theorem 2.1.4 in \cite{MR2015759}, adapted to dimension 3):
\begin{proposition}
Let us consider in the Minkowski spacetime $\R^{3,1}$ the cone defined by\\ $\left\{x^0=\sqrt{\left(x^1\right)^2+\left(x^2\right)^2+\left(x^3\right)^2}\right\}$. Then, there exists a constant $C$ such that, for any $t>0$ and for any function $u$ in $H^1(\Sigma_t)$, where $\Sigma_t$ is the intersection of the interior of the cone and the time slice at time $t$, the following Sobolev embedding is satisfied:
$$
||u||_{L^6(\Sigma_t)}\leq \frac{C}{t}||u||_{H^1(\Sigma_t)}
$$
\end{proposition}
The proof of this proposition is obtained by doing an homothety on the Sobolev inequality over the unit ball. It must be noted that this inequality is generalized to arbitrary metrics.

We intend here to prove the same type of inequality for a uniformly spacelike foliation transverse to the null conformal infinity $\scri^+$. The first step consists in blowing up the singularity in $i^+$ and endowing the obtained cylinder with a conformal metric which admits a regular and non degenerate extension up to the preimage of $i^+$:
\begin{proposition}\label{blowup}
Let $n^a$ be a past oriented null generator of $\scri^+$ and $r$ be an affine parameter associated with $n^a$ such that $r=0$ in $i^+$.\\
Then, for any $R>0$, there exists a global diffeomorphism $\Psi: (0,R)\times \mathbb{S}^2\rightarrow \scri^+\cap J^+\left(\{r=R\}\right)$. Furthermore, the conformal covering of $(\scri^+, g|_{\scri^+})$ by $\left((0,R)\times \mathbb{S}^2, \frac{1}{r^2}\Psi^\star g|_{\scri^+}\right)$ can be extended smoothly to $\left([0,R)\times \mathbb{S}^2, \frac{1}{r^2}\Psi^\star g|_{\scri^+}\right)$ such that the restrictions of the metric $-\frac{1}{r^2}\Psi^\star g|_{\scri^+}$ to $\{r\}\times \mathbb{S}^2$ are uniformly equivalent to the standard Euclidean metric on the two dimensional sphere $\mathbb{S}^2$.
\end{proposition}
\begin{proof} Since the metric on $\scri^+$ is the restriction of a smooth ($C^2$) metric defined in an extension of $\hat M$, it is possible to use a result of H\"afner-Nicolas in \cite{arXiv:0903.0515v1} to obtain the result. 
\end{proof}

Let us then consider a foliation of the interior of $\scri^+$ by uniformly spacelike hypersurfaces $(\Sigma_r)_{r\in [0,R]}$ whose intersections with $\scri^+$ are the level sets of the affine parameter $r$. For this choice of foliation, the following proposition holds:
\begin{proposition}\label{blowup2}
There exists a constant $C$ such that, for each $r\in (0,R]$ and for any function $u$ in $H^1(\Sigma_r)$, the following inequality holds:
$$
\left(\int_{\Sigma_r}u^6\ud \mu[g|_{\Sigma_r}]\right)^{\frac{1}{3}}\leq C \left(\int_{\Sigma_r}||\nabla u||^2\ud \mu[g|_{\Sigma_r}]+\frac{1}{r^2}\int_{\Sigma_r}u^2\ud \mu[g|_{\Sigma_r}]\right)
$$
\end{proposition}
\begin{proof} For each $r>0$, consider a diffeomorphism $\Phi_r$ of the three dimensional ball $B(0,1)$ into  $\Sigma_r$ whose restriction on the boundary is the restriction of the application $\Psi$ to the two dimensional sphere $\{r\}\times \mathbb{S}^2$. Such a map exists since the manifold is globally hyperbolic. Using proposition \ref{blowup}, the family of Riemannian metrics $-\frac{1}{r^2}\Phi^\star_r g|_{\Sigma_r}$ for $r\in (0,R]$ on $B(0,1)$ is uniformly equivalent to the Euclidean metric on the unit ball $B(0,1)$.

Let $r$ be in $(0,R]$ and $u$ in $H^1\left(\Sigma_r, -g|_{\Sigma_r}\right)$. $u\circ \Phi_r$ is then in $H^1\left(B(0,1), -\Phi^\star_r g|_{\Sigma_r}\right)$ and, a fortiori, for a given $r>0$, in $H^1\left(B(0,1), -\frac{1}{r^2} \Phi^\star_r g|_{\Sigma_r}\right)$. Let us consider the Sobolev inequality in $B(0,1)$ for the Euclidean metric: there exists a constant $c$, depending only on the geometry of $B(0,1)$ such that, for any function $f$ in $H^1(B(0,1))$, the following inequality is satisfied:
$$
\left(\int_{B(0,1)} f^6 \ud x\right)^{\frac{1}{3}}\leq c\left(\int_{B(0,1)} f^2 \ud x+\int_{B(0,1)} ||\nabla f||^2 \ud x\right) 
$$

Since the family of Riemannian metrics $-\frac{1}{r^2}\Phi^\star_r g|_{\Sigma_r}$ on $B(0,1)$ is uniformly equivalent to the Euclidean metric on $B(0,1)$, there exists a constant $C$, which depends only on $c$ and on the $L^\infty$-bounds of this family of metrics, such that, uniformly in $r \in (0,R]$, for any function $u$ in $H^1(\Sigma_r,-g|_{\Sigma_r})$, the following inequality holds:
\begin{eqnarray*}
\left(\int_{B(0,1)} (u\circ\Phi_r)^6 \ud \mu\left[-\frac{1}{r^2}\Phi^\star_r g|_{\Sigma_r}\right]\right)^{\frac{1}{3}}&\leq& C\Bigg(\int_{B(0,1)} (u\circ\Phi_r)^2 \ud \mu\left[-\frac{1}{r^2}\Phi^\star_r g|_{\Sigma_r}\right]\\
&&+
\int_{B(0,1)} ||D_r\left( u\circ\Phi_r\right)||^2 \ud \mu\left[-\frac{1}{r^2}\Phi^\star_r g|_{\Sigma_r}\right]\Bigg) 
\end{eqnarray*}
where $D_r$ is the connection associated with the metric $-\frac{1}{r^2}\Phi^\star_r g|_{\Sigma_r}$. Using the property of the conformal transformation and coming back to the manifold $\Sigma_r$, it follows:
\begin{equation}\label{sobolev11}
\left(\int_{\Sigma_r}u^6\ud \mu[g|_{\Sigma_r}]\right)^{\frac{1}{3}}\leq C \left(\int_{\Sigma_r}||\nabla u||^2\ud \mu[g|_{\Sigma_r}]+\frac{1}{r^2}\int_{\Sigma_r}u^2\ud \mu[g|_{\Sigma_r}]\right).
\end{equation}
\end{proof}
\begin{remark}\label{blowup3}
\begin{itemize}
\item It is clear that, using the same arguments as in the previous proof, when considering a uniformly spacelike compact foliation whose volume of the leaves has a uniform lower bound, it is possible to obtain in the same way uniform Sobolev estimates.
\item It must be noted that formula \eqref{sobolev11} agrees with equation \eqref{sobolev111} when considering the behavior of the Sobolev constants.
\end{itemize}
\end{remark}

\subsubsection{Continuity in term of the initial data for the Cauchy problem}\label{continuitypart}

The purpose of this section is to establish estimates on the difference of two solutions of the wave equation. The purpose of these estimates is to obtain continuity in term of initial data, characteristic or not. This step is an important one to obtain the continuity of the scattering operator.

Using the method developed below, it is then possible to obtain the apriori estimates of section \ref{aprioriestimates}.

Let $\phi$ and $\psi$ be two solutions of:
$$
\hat \square u+ \frac16 \scal_{\hat g}u+bu^3=0.
$$
We assume that they satisfy one of the problems:
\begin{itemize}
\item an initial value problem on $\Sigma_0$ with data in $H^1_0(\Sigma_0)\times L^2(\Sigma_0)$ with compact support in $\Sigma_0$;
\item an characteristic initial value problem with data in $H^1(\scri^+)$ with compact support which contains neither $i^0$ nor $i^+$.
\end{itemize}
This ensures that the supports of $u$ and $v$ do not contain the singularity $i^0$.

\begin{theorem}\label{continuity}
Let $\phi$ and $\psi$ two smooth solutions of the nonlinear problem:
$$
\square u+\frac16 \scal_{\hat g} u+bu^3=0.
$$
Then there exist two constants, depending on $\scal_{\hat g}$, $\hat \nabla^{(a} \hat T^{b)}$, $b$ and the energies of $\phi$ and $\psi$ on $\Sigma_0$ and $\scri^+$ such that  the following estimates hold:
$$
||\phi-\psi||^2_{H^1(\scri^+)}\leq C(E_{\phi}(\Sigma_0), E_{\psi}(\Sigma_0)) E_{\phi-\psi}(\Sigma_0)
$$
and
$$ 
E_{\phi-\psi}(\Sigma_0)\leq \tilde C(||\phi||_{H^1(\scri^+)}, ||\psi||_{H^1(\scri^+)})||\phi-\psi||^2_{H^1(\scri^+)},
$$
where the energy of a function on $\Sigma_0$ is chosen to be:
\begin{equation*}
E_{\phi}(\Sigma_0)\approx \int_{\Sigma_0} i^\star_{\Sigma_0}\left(\hat T^aT_{ab}\right)
\approx \int_{\Sigma_0}(\hat T^a\hat \nabla\phi)^2 + \sum_{i=1,2,3}(e^a_{\textbf{i}}\hat \nabla_a \phi)^2+\phi^2\ud\mu_{\Sigma_0},
\end{equation*}
where $T_{ab}$ is the energy tensor associated with the linear wave equation and $(e_{\textbf{i}}^a)_{i=1,2,3}$ is an orthonormal basis of $T\Sigma_0$. 
\end{theorem}
\begin{remark}
Because of the a priori estimates, the constants can be chosen indifferently to depend on the energy on $\scri^+$ or $\Sigma_0$. 
\end{remark}
\begin{proof} The proof relies on exactly the same strategy as in the first section when establishing the a priori estimates. Let $\delta$ be the difference between $\phi$ and $\psi$:
$$
\delta=\phi-\psi.
$$
$\delta$ satisfies the partial differential equation:
$$
\hat\square \delta +\frac16\scal_{\hat g} \delta+ b(\psi^2+\psi\phi+\phi^2)\delta=0.
$$
To establish the inequality, let us consider the energy tensor associated with the linear equation:
$$
T_{ab}=\hat \nabla_a \delta \hat \nabla_b \delta + \hat g_{ab}\left(-\frac12\hat \nabla_c\delta  \hat \nabla^c \delta+\frac{\delta^2}{2}   \right).
$$
The error term associated with this tensor is:
$$
\hat \nabla^a(\hat T^b T_{ab})= (\hat \nabla^{(a}\hat T^{b)} T_{ab})+\underbrace{\delta (\hat T^a\hat \nabla_a \delta)-(\hat T^a\hat \nabla_a \delta) \left(\frac16\scal_{\hat g} \delta+ b(\psi^2+\psi\phi+\phi^2)\delta \right)}_A.
$$
The term $A$ can be estimated by:
\begin{eqnarray*}
A&\leq& \frac12\left(\delta^2+(\hat T^a\hat \nabla_a \delta)^2 \right)+\frac12\sup_{\hat M}\left(|\scal_{\hat g}|\right)\left(\delta^2+(\hat T^a\hat \nabla_a \delta)^2\right) +2b\left((\hat T^a\hat \nabla_a \delta)(\phi^2+\psi^2)\delta\right)\\
A&\leq &2\max\left(\sup_{\hat M}\left(|\scal_{\hat g}|\right),1\right)\left((\hat T^a\hat \nabla_a \delta)^2+\delta^2+b^2\delta^2\psi^4+b^2\delta^2\phi^4\right).
\end{eqnarray*}

The estimates will then be obtained in exactly the same way as in section 1, provided that we are able to use the Sobolev embeddings on the spacelike slices. The main problem then arises when working in the Schwarzschildean section because of the choice of the foliation $\mathcal{H}_s$ which contains the singularity $i^0$.

On $U$, the analogue of the equation \eqref{inegaliteu2} is:
\begin{gather*}
E_\delta(\Sigma_t)- ||\delta ||_{H^1(\scri^+_T)}^2\\\leq\max\left\{1,\sup\left(|\scal_{\hat g}|\right), \sup\left(||\nabla^{(a}\hat T^{b)}||\right)\right\}\left(\int_{t}^{T_{max}} \left(E_\delta(\Sigma_t) +\int_{\Sigma_t} \left(b^2\delta^2\psi^4+b^2\delta^2\phi^4\right)\mu_{\Sigma_t}\right)\ud t\right),
\end{gather*}
where:
\begin{equation}\label{energie33u}
E_\delta(\Sigma_t)=\int_{\Sigma_t} \frac{1}{2} \left(\sum_{i=0}^4 (e^a_{\textbf{i}}\nabla_a\delta)^2+ \frac{\delta^2}{2}\right)\mu_{\Sigma_t}=\frac{1}{2} \left(||u||^2_{H^1(\Sigma_t)}+||\hat T^a \hat \nabla _a\delta||^2_{L^2(\Sigma_t)}\right).
\end{equation}

The first problem arises when dealing with the terms:
$$
\int_{\Sigma_t} \left(b^2\delta^2\psi^4+b^2\delta^2\phi^4\right)\mu_{\Sigma_t}.
$$
The calculations are made for the first term only. Using H\"older inequality, we have:
\begin{equation*}
\int_{\Sigma_t} b^2\delta^2\psi^4\mu_{\Sigma_t}\leq\left(\sup_{\Sigma_t}b\right)^2 \left(\int_{\Sigma_t}\delta^6 \mu_{\Sigma_t}\right)^\frac13\left(\int_{\Sigma_t}\psi^6 \mu_{\Sigma_t}\right)^\frac23.
\end{equation*}
Using proposition \ref{blowup2}, up to a modification of the affine parameter on $\scri^+$, it comes:
$$
\int_{\Sigma_t} b^2\delta^2\psi^4\mu_{\Sigma_t}\leq\frac{\left(\sup_{\Sigma_t}b\right)^2 }{t^6}||\delta||^2_{H^1(\Sigma_t)}||\psi||^4_{H^1(\Sigma_t)}, 
$$
and, consequently, using assumption \ref{A4} on function $b$, the following uniform Sobolev estimate holds: there exists a constant $c$ such that:
$$
\int_{\Sigma_t} b^2\delta^2\psi^4\mu_{\Sigma_t}\leq c ||\delta||^2_{H^1(\Sigma_t)}||\psi||^4_{H^1(\Sigma_t)}.
$$
The same inequality holds with $\phi$. The integral inequality then becomes:
\begin{gather*}
E_\delta(\Sigma_t)- ||\delta ||^2_{H^1(\scri^+_T)}\\\leq\left(C_1+||\phi||^4_{H^1(\Sigma)}+||\psi||^4_{H^1(\Sigma)}\right) \left(\int_{t}^{T_{max}} E_\delta(\Sigma_t) \ud t\right).
\end{gather*}
This gives, using Gronwall's lemma:
$$
E_\delta(\Sigma_t)\leq \exp\left(\left(C_1+||\phi||^4_{H^1(\Sigma)}+||\psi||^4_{H^1(\Sigma)}\right) (T_{max}-t)\right)||\delta ||_{H^1(\scri^+_T)}^2.
$$
The other estimate is obtained when noticing that the inequality also holds:
\begin{gather*}  
||\delta ||^2_{H^1(\scri^+_T)}-E_\delta(\Sigma_t)\\\leq\left(C_1+||\phi||^4_{H^1(\Sigma)}+||\psi||^4_{H^1(\Sigma)}\right) \left(\int_{t}^{T_{max}} E_\delta(\Sigma_t) \ud t\right).
\end{gather*}
Using proposition \ref{aprioriestimate2}, the $H^1$-norm of $\phi$ and $\psi$ on $\Sigma$ is controlled by the $H^1$-norm of $\psi$ and $\phi$ on $\scri^+_T$ and, as a consequence, by the $H^1$-norm of $\psi$ and $\phi$ on $\scri^+$.
We then finally obtain on $U$ the following inequality: there exist two increasing functions $c_U$ and $C_{U}$ such that:
\begin{eqnarray}\label{diff1}
E_{\delta}(\Sigma) &\leq& c_U\left((||\phi||^4_{H^1(\scri^+)}+||\psi||^4_{H^1(\scri^+)}\right)||\delta||^2_{H^1(\scri^+)}\nonumber\\
||\delta||^2_{H^1(\scri^+)}&\leq &C_U\left((||\phi||^4_{H^1(\scri^+)}+||\psi||^4_{H^1(\scri^+)}\right)E_{\delta}(\Sigma)
\end{eqnarray}

On $V$, the principle is exactly the same: the only modification comes from the fact that a new boundary term arises, corresponding to the boundary of the Schwarzschildean section. We work with the same geometric configuration. The equivalent of equation \eqref{ineg212} is then:
\begin{gather*}
E_{\delta}(S^1_t)+E_{\delta}(\Sigma_t)-E_{\delta}(\Sigma)\leq C_2 \int_{t}^1E_{\delta}(\Sigma_t)\ud t+\int_{\Sigma_t} b^2\delta^2\psi^4+b^2\delta^2\phi^4\ud t
\end{gather*}
where $E_{\delta}$ has the same expression as in equation \eqref{energie33u}. We finally obtain the energy equivalence:
\begin{eqnarray}\label{diff2}
E_{\delta}(\Sigma) &\leq &c_V\left((||\phi||^4_{H^1(\scri^+)}+||\psi||^4_{H^1(\scri^+)}\right)\left(E_{\delta}(\som)+E_\delta(S_{u_0})\right)\nonumber\\
E_{\delta}(\som)+E_\delta(S_{u_0})&\leq &C_V\left((||\phi||^4_{H^1(\scri^+)}+||\psi||^4_{H^1(\scri^+)}\right)E_{\delta}(\Sigma)
\end{eqnarray}
\begin{remark} In the subset $V$ of $\hat M$, the energy on a slice is controlled by the upper slice, which is denoted by $\Sigma$ as said in proposition \ref{refenergy4}. As this energy is controlled by proposition \ref{aprioriestimate2} by the $H^1$-norm on $\scri^+_T$ and, as a consequence, on $\scri^+$, this explains why  the energy on $\scri^+$ appears in the inequality.
\end{remark}

Finally, on $\Omega^+_{u_0}$, it is not possible to use the same method as above to control uniformly the Sobolev constant. The strategy consists in adopting the same foliation by the hypersurfaces $\mathcal{H}_s$. The energy on this foliation is weighted Sobolev norm with a precise decay. The Sobolev embeddings must then be adaptated to that decay. The identifying vector field is used to write the integral (see formulae \eqref{parametrization} and \eqref{idvecf}). The error term can then be expressed as:
\begin{gather*}
\int_ {0}^{\tau(s)}\Bigg(\int_{\mathcal{H}_{\tau}}\bigg\{4mR^2(3+uR)\left(\partial_R\phi\right)^2
 +\left(1-12mR\right)\phi\left(u^2\partial_u \phi-2(1+uR)\partial_R \phi\right)\\-(\hat T^a \hat \nabla_a \delta) \delta(\phi^2+\phi\psi+\psi^2) \bigg\}(r^\ast R)^\frac{3}{2}(1-2mR)\sqrt{\frac{R}{|u|}}\ud u\wedge \ud \omega_{\mathbb{S}^2} \Bigg)\ud \tau
\end{gather*}

In this subset of $\hat M$, the error is bounded above by, using H\"older inequality:
\begin{gather*}
\int_0^{\tau(s)}\left\{E_\delta(\mathcal{H}_\tau)+\left(\int_{\mathcal{H}_\tau}\delta^6\sqrt{\frac{R}{|u|}}\ud u\wedge \ud \omega_{\mathbb{S}^2}\right)^\frac13\left(\int_{\mathcal{H}_\tau}\left(\phi^6+\psi^6\right)\sqrt{\frac{R}{|u|}}\ud u\wedge \ud \omega_{\mathbb{S}^2}\right)^\frac23\right\} \ud \tau\\
\leq \int_0^{\tau(s)}\left\{E_\delta(\mathcal{H}_\tau)+\sqrt{\frac{\epsilon}{2m|u_0|}}\left(\int_{\mathcal{H}_\tau}\delta^6\ud u\wedge \ud \omega_{\mathbb{S}^2}\right)^\frac13\left(\int_{\mathcal{H}_\tau}\left(\phi^6+\psi^6\right)\ud u\wedge \ud \omega_{\mathbb{S}^2}\right)^\frac23\right\} \ud \tau.
\end{gather*}
The Sobolev embedding from $H^1$ into $L^6$ must then be realized uniformly in $\tau$ with respect to the volume form $\ud u\wedge \omega_{\mathbb{S}^2}$, which is the volume form associated with the cylinder $]u_0,+\infty [\times \mathbb{S}^2$ and the metric $(\ud u)^2+\ud \omega_{\mathbb{S}^2}^2$. Since the Sobolev embedding from $H^1(]u_0,+\infty [\times \mathbb{S}^2)$ into $L^6(]u_0,+\infty [\times \mathbb{S}^2)$ is valid in this geometry, we obtain, in the coordinates $(u,\omega_{\mathbb{S}^2})$ the following Sobolev inequality:
\begin{gather*}
\int_{\mathcal{H}_s}\phi^6\ud u \ud \omega_{\mathbb{S}^2}=\int_{]u_0,+\infty [\times \mathbb{S}^2}\phi^6\ud u \ud \omega_{\mathbb{S}^2}\\
\leq K
\int_{]u_0,+\infty [\times \mathbb{S}^2}(\partial_u(\phi\Big|_{\mathcal{H}_s=\{u=-sr^\ast\}}))^2+|\nabla_{\mathbb{S}^2}\phi|^2+\phi^2\ud u \ud \omega_{\mathbb{S}^2}\\
\leq\int_{]u_0,+\infty [\times \mathbb{S}^2}\left(\partial_u\phi+\frac{r^\ast R^2(1-2mR)}{|u|}\partial_R\phi\right)^2+|\nabla_{\mathbb{S}^2}\phi|^2+\phi^2\ud u \ud \omega_{\mathbb{S}^2}\\
\leq \int_{]u_0,+\infty [\times \mathbb{S}^2}2(\partial_u\phi)^2+2(r^\ast R)^2(1-2mR)^2\left(\frac{R}{|u|}\right)^2(\partial_R\phi^2)|+\nabla_{\mathbb{S}^2}\phi|^2+\phi^2\ud u \ud \omega_{\mathbb{S}^2}\\
\leq \int_{]u_0,+\infty [\times \mathbb{S}^2}2\frac{u^2}{u_0^2}(\partial_u\phi)^2+(1+\epsilon)^2\frac{\epsilon}{2m|u_0|}\frac{R}{|u|}(\partial_R\phi^2)+|\nabla_{\mathbb{S}^2}\phi|^2+\phi^2\ud u \ud \omega_{\mathbb{S}^2}.
\end{gather*}
Then there exists a constant $K$, depending on $u_0$ and $\epsilon$ such that, uniformly in $s$:
$$
\left(\int_{\mathcal{H}_\tau}\phi^6 \sqrt{\frac{R}{|u|}}\ud u \ud \omega_{\mathbb{S}^2}\right)^\frac23\leq K ||\phi||^4_{H^1(\mathcal{H}_\tau)} \text{ and }\left(\int_{\mathcal{H}_\tau}\delta^6 \sqrt{\frac{R}{|u|}}\ud u \ud \omega_{\mathbb{S}^2}\right)^\frac13 \leq K E_{\delta}(\mathcal{H}_\tau).
$$
Using the fact that (see equation \eqref{energyequivalenceeq}):
$$
||\phi||^4_{H^1(\mathcal{H}_{\tau(s)})}\lesssim \left(E_\phi(\sop)\right)^2,
$$
the following integral inequality holds:
\begin{gather*}
\left|E_{\delta}(\mathcal{H}_{\tau(s)}+ E_\delta(S_{u_0})-E_\delta(\sop)\right|\\
\leq (C+K(E_\phi(\sop)^2+E_\psi(\sop) ^2) )\int_0^{\tau(s)}E_{\delta}(\mathcal{H}_{\tau})\ud \tau
\\\leq (C+K(E_\phi(\Sigma_0)^2+E_\psi(\Sigma_0)^2) )\int_0^{\tau(s)}E_{\delta}(\mathcal{H}_{\tau})\ud \tau
\end{gather*}
and, using the a priori estimates given by theorem \ref{aprioriestimatesfinal}, 
 \begin{gather*}
\left|E_{\delta}(\mathcal{H}_{\tau(s)})+ E_\delta(S_{u_0})-E_\delta(\sop)\right|\leq \left(C+\tilde K\left((||\phi||^4_{H^1(\scri^+)}+||\psi||^4_{H^1(\scri^+)}\right)\right)\int_0^{\tau(s)}E_{\delta}(\mathcal{H}_{\tau})\ud \tau.
\end{gather*}
Finally, there exist two increasing functions $c_{\Omega^+_{u_0}}$ and $C_{\Omega^+_{u_0}}$ such that:
\begin{eqnarray}\label{diff3}
E_{\delta}(\scri^+_{u_0})+ E_\delta(S_{u_0})&\leq &c_{\Omega^+_{u_0}}\left(||\phi||^4_{H^1(\scri^+)}+||\psi||^4_{H^1(\scri^+)}\right) E_{\delta}(\sop)\nonumber\\
E_{\delta}(\sop)&\leq &C_{\Omega^+_{u_0}}\left(||\phi||^4_{H^1(\scri^+)}+||\psi||^4_{H^1(\scri^+)}\right)\left(E_{\delta}(\scri^+_{u_0})+ E_\delta(S_{u_0})\right).
\end{eqnarray}
 
 Eventually, combining inequalities \eqref{diff1}, \eqref{diff2} and \eqref{diff3} as in section \ref{aprioriestimates} for the proof of theorem \ref{aprioriestimatesfinal}, we get the existence of two increasing functions $c$ and $C$ such that:
\begin{eqnarray*}
E_{\delta}(\Sigma_0)&\leq & c\left((||\phi||^4_{H^1(\scri^+)}+||\psi||^4_{H^1(\scri^+)}\right)||\delta||_{H^1(\scri^+)}^2\\
||\delta||_{H^1(\scri^+)}^2&\leq &C\left((||\phi||^4_{H^1(\scri^+)}+||\psi||^4_{H^1(\scri^+)}\right)E_{\delta}(\Sigma_0).
\end{eqnarray*}
Because of the a priori estimates given by theorem \ref{aprioriestimatesfinal}, the $H^1$-norm of $\phi$ and $\psi$ on $\scri^+$ can be replaced by the energy of $\phi$ and $\psi$ on $\Sigma_0$.\end{proof}

This result is equivalent to the result obtained by H\"ormander at the beginning of his paper.

Finally, as already mentioned, a by-product of this result is the  continuity result for the Cauchy problem for the nonlinear wave equation on a uniformly spacelike hypersurface $\mathcal{S}$, transverse to $\scri^+$. The problems are exactly the same: obtaining uniform Sobolev estimates near $i^0$ and in the equivalent of the region $V$. The techniques to solve this problem are then exactly the same. We are working with functions $\phi$ and $\psi$ which satisfy the same assumptions as for theorem \ref{continuity}.
\begin{proposition}\label{contractioncauchy}
Let $\mathcal{S}$ be a uniformly spacelike hypersurface transverse to $\scri^+$. Let $\phi$ and $\psi$ be two smooth solutions of the nonlinear problem:
$$
\square u+\frac16 \scal_{\hat g} u+bu^3=0.
$$
Then there exists two constants, depending on $\scal_{\hat g}$, $\hat \nabla^a \hat T^b$, $b$ and the energy of $\phi$ and $\psi$ on $\Sigma_0$ and $\mathcal{S}$ such that the following estimates hold:
$$
E_{\phi-\psi}(\mathcal{S})\leq C(E_{\phi}(\Sigma_0), E_{\psi}(\Sigma_0)) E_{\phi-\psi}(\Sigma_0)
$$
and
$$ 
E_{\phi-\psi}(\Sigma_0)\leq \tilde C(||\phi||_{H^1(\scri^+)}, ||\psi||_{H^1(\scri^+)})E_{\phi-\psi}(\mathcal{S}),
$$
where the energy of a function $u$ on a uniformly spacelike hypersurface $\Sigma$ is chosen to be:
\begin{equation*}
E_{u}(\Sigma)\approx \int_{\Sigma_0} (\hat T^a
T_{ab})
\approx \int_{\Sigma_0}(\hat T^a\hat \nabla u)^2 + \sum_{i=1,2,3}(e^a_{\textbf{i}}\hat \nabla_a u)^2+u^2\ud\mu_{\Sigma_0},
\end{equation*}
where $T_{ab}$ is the energy tensor associated with the linear wave equation and $(e_{\textbf{i}}^a)_{i=1,2,3}$ is an orthonormal basis of $T\Sigma$. 
\end{proposition}

\subsection{Solution of the Goursat problem near $\scri^+$}\label{partie21}

This section presents a proof of the local existence of a solution of the characteristic Cauchy problem on $\scri^+$, which consists in slowing down the propagation speed of waves, as done in H\"ormander's work in \cite{MR1073287}. This gives a weak convergence result. The weak solution is constructed down to a certain uniformly spacelike hypersurface, constructed as follows.

Let us consider a smooth function $\theta$ defined on $\scri ^+$ whose compact support contains neither $i^0$ nor $i^+$. We consider a uniformly spacelike hypersurface $\mathcal{S}$ such that:
\begin{itemize}
\item $\mathcal{S}$ is in the future of $\Sigma_0$ and does not contain $i^0$;
\item $\mathcal{S}$ is transverse to $\scri^+$;
\item $\mathcal{S}$ is in the past of the support of $\theta$.
\end{itemize}
A local existence result for the characteristic data $\theta$ down to the hypersurface $\mathcal{S}$ can then be stated:
\begin{proposition}\label{localcharprob0}
Let us consider the nonlinear characteristic Cauchy problem on $\scri^+$:
\begin{equation*}\left\{
\begin{array}{l}
\hat \square u+ \frac16 \scal_{\hat g}u+bu^3=0\\
u\big|_{\scri^+}=\theta\in H^1(\scri^+).
\end{array}\right.
\end{equation*}
where $\theta$ is a smooth function whose compact support does not contain $i^+$ or $i^0$.
Then this problem admits a weak solution in the future of $\mathcal{S}$ in $H^1(J^+(\mathcal{S}))$.
\end{proposition}

\begin{proof}
Starting from the spacelike hypersurface $\mathcal{S}$ whose future contains the support of $\theta$, let $t$ be a time function on $\hat M$ with $t=0$ on $\mathcal{S}$. The metric is split with respect to this time function as follows:
$$
\hat g= N^2 (\ud t)^2 -h_t,  
$$
where $N$ is the lapse function and $h_t$ is a Riemannian metric on the spacelike slice $t=constant$.

A real parameter $\lambda$ in $(\frac12,1)$ is introduced to obtain a new family of metrics $g_\lambda$ defined by:
$$
\hat g_\lambda= \lambda^2 N^2 (\ud t)^2 - h_t.
$$
For a given value of $\lambda$, the hypersurface $\scri^+$ is now uniformly spacelike and the initial value problem on $\scri^+$ for the wave operator associated with $\hat g_\lambda$, denoted $\hat \square_\lambda$, is then a Cauchy problem on $\scri^+$. Furthermore, since the support of the data $\theta$ does not contain the singularity $i^+$ (and the singularity $i^0$), the problem can be considered to be set on a smooth spacelike hypersurface. The result by Cagnac and Choquet-Bruhat in \cite{MR789558} can then be used. Let us then consider the family of solutions $(u_\lambda)$ in $H^1(J^+(\mathcal{S}))$ of the Cauchy problem:
\begin{equation*}\left\{
\begin{array}{l}
\hat \square_\lambda u_\lambda+ \frac16 \scal_{\hat g_\lambda}u_\lambda+bu_\lambda^3=0\\
u_\lambda=\theta \text{ on }\scri^+\\
\partial_t u_\lambda=0\text{ on }\scri^+
\end{array}\right.
\end{equation*}

The first step of the proof consists in proving that the family $(u_\lambda)$ is bounded in $H^1(J^+(\mathcal{S}))$. This can be done by noting that the same inequality as in propositions \ref{energyequivalence} and \ref{aprioriestimate2} holds for the metric $g_\lambda$ with $\Sigma_T=\mathcal{S}$ in this case: as a consequence, there exists a constant $c_\lambda$ depending on the time function $t$ and depending continuously on the scalar curvature of $g_\lambda$ such that:
\begin{equation}\label{energylambda}
E_{u_\lambda}(\mathcal{S})\leq c_\lambda \int_{\scri^+}i_{\scri^+}^\star\left(\star T^a T_{ab}\right)
\end{equation}
where:
\begin{itemize}
 \item the integral on the right-hand side only depends on the function $\theta$;
\item the constant $c_\lambda$ is bounded above by a certain constant $C$ since it depends continuously on $\lambda$ in $[\frac12,1]$;
\item and the energy $E_{u_\lambda}(\mathcal{S})$ controls the $H^1$-norm of $u_\lambda$.
\end{itemize}
As a consequence the family $(u_\lambda)$ is bounded in $H^1(J^+(\mathcal{S}))$.

Then there exists a sequence $(\lambda_n)$ converging towards 1 and a function $u$ in $H^1(J^+(\mathcal{S}))$ such that the sequence $(u_n=u_{\lambda_n})$ converges towards $u$ weakly in $H^1(J^+(\mathcal{S}))$.

Consequently, there exists a subsequence of $(u_n)$, still denoted $(u_n)$, converging towards $u$ strongly in $L^2(J^+(\mathcal{S}))$. This sequence converges a fortiori strongly in $L^1(J^+(\mathcal{S}))$, using the Cauchy-Schwarz inequality, since we are working on a compact neighborhood of $i^+$.

The next step consists in considering the nonlinear term of the equation. We prove that $(u_n^3)$ is a Cauchy sequence in $L^1(J^+(\mathcal{S}))$. Noticing that, for any $(n,p)$ in $\N^2$
$$
||u_n^3-u_p^3||_{L^1(J^+(\mathcal{S}))}\leq \frac32 \left(||u_n||^2_{L^4(J^+(\mathcal{S}))}+||u_p||^2_{L^4(J^+(\mathcal{S}))}\right)||u_n-u_p||_{L^2(J^+(\mathcal{S}))}
$$
and, using the Sobolev embedding from $H^1$ into $L^4$ (in dimension 4), there exists a constant $C$ such that:
$$
||u_n^3-u_p^3||_{L^1(J^+(\mathcal{S}))}\leq C \sup_{n \in \N}\left(||u_n||^2_{H^1(J^+(\mathcal{S}))}\right)||u_n-u_p||_{L^2(J^+(\mathcal{S}))}.
$$
As a consequence, $(u_n^3)$ is converging strongly in $L^1(J^+(\mathcal{S}))$ towards $u$.

It finally remains to prove that the function $u$ satisfies the given characteristic Cauchy problem. Consider the trace of the functions $(u_n)$ on $\scri^+$. Since $(u_n)$ converges weakly in $H^1(J^+(\mathcal{S}))$ towards $u$, using theorem \ref{tracethm}, the restriction of $(u_n)$ to $\scri^+$ converges weakly in $H^{\frac12}(\scri^+)$ towards the restrictions of $u$ to $\scri^+$. Using Rellich-Kondrachov theorem, passing to a subsequence if necessary, the constant sequence $(u_n|_{\scri^+}=\theta)$ converges towards $u|_{\scri^+}$ strongly in $L^2(\scri^+)$.  The function $u$ then satisfies the initial condition $u=\theta$ on $\scri^+$.
\end{proof}

\subsection{Global characteristic Cauchy problem}

A global Cauchy problem is finally derived in two steps:
\begin{enumerate}
\item a preliminary result about the Cauchy problem for a hypersurface in the future of $\Sigma_0$ whose past contains $i^0$;
\item the characteristic Cauchy problem is then solved for small data with compact support which contains neither $i^0$ nor $i^+$ and then extended to functions in $H^1(\scri^+)$.
\end{enumerate}

\subsubsection{Global Cauchy problem for compactly supported data}\label{partie22}

Starting from the same data $\theta$ in $H^1(\scri^+)$ whose support contains neither $i^+$ nor $i^0$, the solution obtained in theorem \ref{goursatproblem} is extended to the future of $\Sigma_0$ by the means of density results and continuity of the propagator. The purpose of this section is to show that is it possible, starting from the hypersurface $\mathcal{S}$ in the future of $\Sigma_0$, to obtain a solution down to $\Sigma_0$ despite the singularity in $i^0$.

\begin{proposition}\label{globalcharprob}
Let $V^a$ be an orthogonal and normalized vector field to the uniformly spacelike hypersurface $\mathcal{S}$.\\ 
The non-linear problem on $\mathcal{S}$:
\begin{equation*}\left\{
\begin{array}{l}
\hat \square v+ \frac16 \scal_{\hat g}v+bv^3=0\\
v\big|_{\Sigma_0}=\xi\in H^1_0(\mathcal{S})\\
V^a\hat \nabla_a v\big|_{\Sigma_0}= \zeta \in L^2(\mathcal{S})
\end{array}\right.
\end{equation*}
admits a global unique solution down to $\Sigma_0$ in $C^0(\R, H^1_0(\mathcal{S}))$.
\end{proposition}
\begin{proof} The method consists in approximating the solution by solutions of the same problem with truncated data since the existence result for the Cauchy problem given by theorem \ref{choquetbruhat} cannot be applied here directly because of the singularity in $i^0$. The uniqueness directly comes from theorem \ref{continuity} and its corollary.

Let $(\chi_n)_{n\in \N}$ be a sequence of smooth functions with compact support in the interior of $\mathcal{S}$ such that:
$$
\forall n \in \N,\text{supp} \left(\chi_{n}\right) \subset\text{supp} \left(\chi_{n+1}\right) \text{ and } \bigcup_{n \in \N}\text{supp}\left(\chi_n\right)=\mathcal{S}\backslash\partial \mathcal{S}.
$$

Let $(v_n)_{n\in\N}$ be the sequence defined by:
\begin{equation*}\left\{
\begin{array}{l}
\hat \square v_n+ \frac16 \scal_{\hat g}v_n+bv_n^3=0\\
v_n\big|_{\mathcal{S}}=\chi_n\xi \in H^1(\mathcal{S}).\\
V^a\hat \nabla_a v_n\big|_{\mathcal{S}}=\chi_n\zeta \in L^2(\mathcal{S}).
\end{array}\right.
\end{equation*}
Since the data are with compact support in the interior of $\mathcal{S}$, their pasts do not intersect $\scri^+$ and, as a consequence, $i^0$.

Using proposition \ref{contractioncauchy}, this sequence converges towards a function $v$ in the past of $\mathcal{S}$ down to $\Sigma_0$ for the $L^\infty H^1$ norm. This function  satisfies the initial conditions:
$$
v\big|_{\mathcal{S}}=\xi \text{ on } \mathcal{S}\text{ and } 
 V^a\hat \nabla_a v\big|_{\mathcal{S}}=\zeta.
$$
Furthermore, proposition \ref{contractioncauchy} also gives convergence in $H^1(J^-(\mathcal{S}))$ and, as a consequence, using Sobolev embeddings, in $L^6(J^-(\mathcal{S}))$. $v$ then satisfies the nonlinear wave equation in the distribution sense. \end{proof}

\begin{remark}
A direct consequence of this construction is that the trace of the solution of the Cauchy problem on $\Sigma_0$ is in $H^1_0(\Sigma_0)$.
\end{remark}

\subsubsection{Global characteristic Cauchy problem for initial data in $H^1(\scri^+)$}\label{partie23}

A global solution to the Goursat problem with compact support which contains neither $i^+$ nor $i^0$ is then obtained by gluing solution of the local characteristic Cauchy problem obtained in proposition \ref{localcharprob0} and the solution of a well-chosen Cauchy problem on $\mathcal{S}$:
\begin{proposition}\label{extension}
Let $u$ be a solution to the Goursat problem for data $\theta$ with compact support which contains neither $i^+$ nor $i^0$.\\
Then $u$ can be extended from $\mathcal{S}$ down to $\Sigma_0$ in $C^0(\R, H^1(\mathcal{S}))$.
\end{proposition}
\begin{proof} Consider the Cauchy problem on $\mathcal{S}$:
\begin{equation*}\left\{
\begin{array}{l}
\hat \square v+ \frac16 \scal_{\hat g}v+bv^3=0\\
v\big|_{\mathcal{S}}= u \in H^1(\mathcal{S}).\\
V^a\hat \nabla_a v\big|_{\mathcal{S}}=V^a\hat \nabla_a u \in L^2(\mathcal{S}).
\end{array}\right.
\end{equation*}
According to proposition \ref{globalcharprob}, this problem admits a global solution $v$ down to $\Sigma_0$ in $C^0(\R, H^1(\mathcal{S}))$.\\
Finally, the function $w$ defined piecewise by:
$$
w=u \text{ on } J^+{\mathcal(S)} \text{ and }  w=v \text{ on } J^+(\Sigma_0)\cap J^-{\mathcal(S)}.
$$
satisfies the Goursat problem:
\begin{equation*}\left\{
\begin{array}{l}
\hat \square u+ \frac16 \scal_{\hat g}u+bu^3=0\\
\phi\big|_{\scri^+}=\theta\in H^1(\scri^+)\\
\end{array}\right.
\end{equation*}\end{proof}

Using proposition \ref{extension} and the continuity result, we can state the theorem of existence of the Goursat problem:

\begin{theorem}\label{globalcharprob1}
Let us consider the nonlinear characteristic Cauchy problem on $\scri^+$:
\begin{equation*}\left\{
\begin{array}{lcl}
\hat \square u+ \frac16 \scal_{\hat g}u+bu^3&=&0\\
u\big|_{\scri^+}=\theta\in H^1(\scri^+).&&\\
\end{array}\right.
\end{equation*}
Then, this problem admits a global unique solution down to the future of $\Sigma_0$ in $C^0(\R, H^1(\Sigma_0))$.
\end{theorem}
\begin{proof} The proof relies on the density of data with compact support in $H^1(\scri^+)$ which contains neither $i^0$ nor $i^+$ in $H^1(\scri^+)$ (proposition \ref{localsobolev}) and proposition \ref{extension}.\end{proof}

\begin{remark}
As noticed above, the trace of the solution of the Goursat problem on $\Sigma_0$ is in $H^1_0(\Sigma_0)$.
\end{remark}
\section{Construction of the scattering operator}

The construction of the scattering operator can now be done by solving Cauchy problems on $\scri^+$, $\scri^-$ and $\Sigma_0$ via the composition of trace operators.

\subsection{Existence and continuity of trace operators}\label{traceopsection}

The purpose of this section is to define trace operators for the solutions of the wave equation on the hypersurfaces $\Sigma_0$ and $\scri^+$. A similar construction can of course be realized on the past null infinity $\scri^-$.

These trace operators are obtained using the following theorem (\cite{MR1395148}, p. 287):
\begin{theorem}\label{tracethm}
Let $M$ be a smooth compact manifold with piecewise $C^1$ boundary and consider the application $T$ defined by:
$$
T:\left\{
\begin{array}{ccc}
 C^0(M)&\longrightarrow & C^0(\partial M)\\
 f&\longmapsto & f\big |_{\partial M}.
\end{array}
\right.
$$
Then, for all $s>\frac12$, the operator $T$ extends uniquely to a continuous map from $H^s(M)$ into $H^{s-\frac12}(\partial M)$.
\end{theorem}
Existence theorems \ref{choquetbruhat} and \ref{globalcharprob1} give solutions to the initial (characteristic) problem in $H^1(J^+(\Sigma_0))$. As a consequence, their traces on $\Sigma_0$ and $\scri^+$ are respectively in $H^{\frac12}(\Sigma_0)$ and $H^{\frac12}(\scri^+)$. Nonetheless, using the a priori estimates, they are in fact $H^1(\scri^+)$ and $H^1(\Sigma_0)$.
\begin{remark}
The singularity in $i^+$ is not a threat to the existence of a trace since the manifold and the metric can be extended with arbitrary regularity in a neighborhood of $i^+$. The problem with the singularity $i^0$ is avoided since the function spaces $H^1(\scri^+)$ and $H^1_0(\Sigma_0)$ are the completions of smooth functions whose compact support does not contain $i^0$.
\end{remark}

Let us consider the trace operators:
\begin{equation}
T_0^+:=\left\{
\begin{array}{ccc}
C^{\infty}_{0}(\Sigma_0)\times C^{\infty}_{0}(\Sigma_0)&\longrightarrow& H^1(\scri^+) \\
(\theta,\tilde \theta) &\longmapsto & \phi\big |_{\scri^+}
\end{array}\right.
\end{equation}
where $\phi$ is the unique solution of the problem:
\begin{equation*}\left\{
\begin{array}{l}
\hat \square \phi+ \frac16 \scal_{\hat g}\phi+b\phi^3=0\\
\phi\big|_{\Sigma_0}=\theta\in C^\infty_0(\Sigma_0)\\
\hat T^a\hat \nabla_a \phi\big|_{\Sigma_0}= \tilde \theta\in C^\infty_0(\Sigma_0)
\end{array}\right.
\end{equation*}
obtained by theorem \ref{choquetbruhat} and 
\begin{equation}
T_+^0:=\left\{
\begin{array}{ccc}
\mathscr{E} &\longrightarrow& H^1_0(\Sigma_0)\times L^2(\Sigma_0)\\
\theta &\longmapsto & (\phi\big |_{\Sigma_0},(\hat T^a \hat \nabla_a \phi)\big |_{\Sigma_0})
\end{array}\right.
\end{equation}
where $\mathscr{E}$ is the set of smooth functions with compact support which contains neither $i^+$ nor $i^0$ and  $\phi$ is the unique solution of the problem:
\begin{equation*}\left\{
\begin{array}{lcl}
\hat \square \phi+ \frac16 \scal_{\hat g}\phi+b\phi^3&=&0\\
\phi\big|_{\scri ^+}=\theta\in C^\infty_0(\scri ^+)&&\\
\end{array}\right.
\end{equation*}
obtained by theorem \ref{globalcharprob1}.

These operators can be extended to $H^1_0(\Sigma_0)$ and $H^1(\scri^+)$:
\begin{proposition}\label{traceop}
The operator $T_0^+$ can be extended to a locally Lipschitz operator from $H_0^1(\Sigma_0)\times L^2(\Sigma_0)$ to $H^1(\scri^+)$.\\
The operator $T_+^0$  can be extended to a locally Lipschitz operator from $H^1(\scri^+) $ to $H_0^1(\Sigma_0)\times L^2(\Sigma_0)$.
\end{proposition}

\begin{proof} The proof is done for the operator $T^+_0$ (it is exactly the same on the other side).

Using theorem \ref{continuity}, this operator satisfies:
\begin{gather*}
\forall (\theta, \tilde \theta, \xi, \tilde \xi)\in \left(C_0^\infty (\Sigma_0)\right)^4,\\
||T^+_0(\theta, \tilde \theta)-T^+_0(\xi, \tilde \xi)||_{H^1(\scri^+)}\leq C(R)\left(||\theta-\xi||^2_{H^1(\Sigma_0)}+||\tilde \theta -\tilde \xi ||^2_{L^2(\Sigma_0)}\right)^\frac12.
\end{gather*}
As a consequence, since the smooth functions with compact support in $\Sigma_0$ are dense in $H^1_0(\Sigma_0)$, it admits a unique locally Lipschitz extension from $H^1_0(\Sigma_0)\times L^2(\Sigma_0)$ into $H^1(\scri^+)$.

The same proof holds for the operator $T^0_+$. \end{proof}

As already noted at the beginning of this section, a similar construction can be achieved on the past null infinity: there exist two Lipschitz trace operators $T^0_-$ and $T_0^-$ defined by:
\begin{equation}
T_0^-:=\left\{
\begin{array}{ccc}
H^1_0(\Sigma_0)\times L^2(\Sigma_0)&\longrightarrow& H^1(\scri^-) \\
(\theta,\tilde \theta) &\longmapsto & \phi\big |_{\scri^+}
\end{array}\right.
\end{equation}
where $\phi$ is the unique solution of the problem:
\begin{equation*}\left\{
\begin{array}{l}
\hat \square \phi+ \frac16 \scal_{\hat g}\phi+b\phi^3=0\\
\phi\big|_{\Sigma_0}=\theta\in H^1_0(\Sigma_0)\\
\hat T^a\hat \nabla_a \phi\big|_{\Sigma_0}= \tilde \theta\in L^2(\Sigma_0)
\end{array}\right.
\end{equation*}
obtained by theorem \ref{choquetbruhat} and 
\begin{equation}
T_-^0:=\left\{
\begin{array}{ccc}
H^1(\scri^-) &\longrightarrow& H^1_0(\Sigma_0)\times L^2(\Sigma_0)\\
\theta &\longmapsto & (\phi\big |_{\Sigma_0},(\hat T^a \hat \nabla_a \phi)\big |_{\Sigma_0})
\end{array}\right.
\end{equation}
where $\phi$ is the unique solution of the problem:
\begin{equation*}\left\{
\begin{array}{lcl}
\hat \square \phi+ \frac16 \scal_{\hat g}\phi+b\phi^3&=&0\\
\phi\big|_{\scri ^-}=\theta\in H^1(\scri ^-)&&\\
\end{array}\right.
\end{equation*}
obtained by theorem \ref{globalcharprob1}.

\subsection{Conformal scattering operator}

Finally, the conformal scattering operator is obtained as the composition of two trace operators. Following the idea of of Friedlander in $\cite{MR583989}$ and applied by Mason-Nicolas in \cite{mn04} for the Dirac and wave equations, the conformal scattering operator $S$ is defined by composing the operators $T^0_-$ and $T^+_0$:
\begin{equation}
S=T^+_0\circ T_-^0: H^1(\scri^-)\longrightarrow H^1(\scri^-)
\end{equation}
and its inverse is given by
\begin{equation}
S^{-1}=T^-_0\circ T_+^0: H^1(\scri^+) \longrightarrow H^1(\scri^-)
\end{equation}

Finally, the following existence result for the conformal scattering operator can be stated:
\begin{theorem}[Scattering operator]
The operator $S$ is an inversible, locally Lipschitz operator from $H^1(\scri^-)$ into in $H^1(\scri^+)$. This operator is called conformal scattering operator.
\end{theorem}
\begin{proof} The proof is an immediate consequence of proposition \ref{traceop}.\end{proof}

\begin{remark}
The conformal scattering operator was introduced to avoid the use of the spectral theory which requires the metric to be static. It is nonetheless possible to talk about geometric scattering at least in the Schwarzschildean part of the manifold and wonder wether it is possible to establish an equivalence in  this region. Some answers to this question can be found in \cite{mn04} (section 4.2) for the Dirac and Maxwell equations.
\end{remark}

\section*{Concluding remarks}

There exist several possible extensions to this work:
\begin{itemize}
\item the case where the metric in the neighborhood of $i^0$ is the Kerr-Newman metric;
\item the nonlinearity could be modified and the equation could, for instance, be quasilinear, or satisfy the null condition;
\item following \cite{mn07}, these results could be extended to peeling results for the same cubic defocusing wave equation.
\end{itemize}

One of the main problems of general relativity is the construction of solutions to Einstein equations. One intermediate step could be to establish the same kind of result for the Yang-Mills equations.

\appendix

\section{Appendix}

As previouly mentioned, this appendix is divided in three sections. The first one contains the proof of lemma \ref{techlemma1}. The second section gives the necessary details leading to the a priori estimates. The last one presents an alternate proof, for small data, of the well posedness of the Cauchy problem on conformal infinity.

\subsection{Proof of lemma \ref{techlemma1}}

\begin{lemma}\label{techlemma11}
Let us consider the set of smooth functions defined in $\overline{B(0,1)}\subset \R^3$ with support which does not contain $0$. Then this set is dense in $H^1(B(0,1))$.
\end{lemma}
\begin{proof} It is sufficient to prove that the constant function $1$ can be approximated by a sequence of smooth functions whose compact support does not contain $0$. 

Let $f$ be the function defined on $\R^+$ by:
\begin{itemize}
\item $f$ is a smooth function on $\R^+$ with value in $[0,1]$;
\item $f=1$ in $[\frac12, +\infty)$;
\item $f$ vanishes in $[0,\frac13]$. 
\end{itemize}
Let us consider the sequence of smooth spherically symmetric functions defined by:
$$
\forall n \in \N, \forall x \in B(0,1), \psi_n(x)= f(n||x||).
$$
They satisfy, for all $n$ in $\N$:
\begin{itemize}
\item $\psi_n=1$ in $B(0,1)\backslash B(0,\frac{1}{2n})$;
\item $\psi_n$ vanishes in $B(0,\frac{1}{3n})$;
\item $\psi_n$ is a smooth function on $B(0,1)$ with value in $[0,1]$ since it vanishes in a neighborhood of zero.
\end{itemize}
Finally, the difference $(1-\psi_n)_n$ converges towards $0$ in $H^1$-norm:
\begin{eqnarray*}
||1-\psi_n||_{H^1}^2&=&\int_{B(0,1)}\left( (1-f(nr))^2+n^2(f'(nr))^2  \right) r^2\ud r \ud \omega_{\mathcal{S}^2}\\
&=&\frac43 \pi\left(\int_{0}^1(1-f(nr))^2\ud r+n^
2\int_{\frac{1}{3n}}^{\frac{1}{2n}}(f'(nr))^2r^2 \ud r\right)\\
&\leq& \frac43 \pi\left(\int_{0}^1(1-f(nr))^2\ud r+\frac{\sup_\R((f')^2)}{n}\right).
\end{eqnarray*}
The remaining integral converges towards $0$ by Lebesgue theorem. As a consequence, the sequence $(\psi_n)_n$ converges towards the constant $1$ in $H^1(B(0,1))$.

Finally, let $f$ be a function $H^1(B(0,1))$ (or in $H^1_0(B(0,1))$). Cauchy-Schwarz inequality gives, for all $n$ in $\N$:
\begin{eqnarray*}
||f(1-\psi_n)||_{H^1}^2 &= &||f(1-\psi_n)||_{L^2}^2+\big|\big|f|\nabla (1-\psi_n)|  \big|\big|^2_{L^2}+\big|\big||\nabla f| (1-\psi_n)\big|\big|^2_{L^2}\\
&\leq & 2||f||^2_{H^1}||1-\psi_n||^2_{H^1}.
\end{eqnarray*} 
$(f\psi_n)_n$ is then a sequence of functions in $H^1(B(0,1))$ whose support does not contain 0 which converges in $H^1$ towards $f$.
\end{proof}  

\subsection{A priori estimates -- Some details}

Most of the calculations which are made in this paper was already made in the papers of Mason-Nicolas \cite{mn04, mn07} in the linear case. Nonetheless, for the sake of self-consistency, it has been decided to give more details about the proofs of these inequalities. Two kinds of details are given here: the first one is about the geometry of the Schwarschild spacetime and the second one about the estimates themselves.

\begin{proposition}\label{expressionschwenergie2}The energy 3-form, written in the coordinates $(R,u,\theta, \psi)$, is given by:
\begin{gather*}
\star \hat T ^aT_{ab}=\left[u^2(\partial_u\phi)^2+R^2(1-2mR)\left(u^2\partial_R\phi\partial_u\phi-(1+uR)(\partial_R\phi\right)^2)\right.\\
+\left.\left(\frac12|\nabla_{\mathbb{S}^2}\phi|^2+\frac{\phi^2}{4}+b\frac{\phi^4}{4}\right)\right]\sin(\theta)\ud u \wedge \ud \theta \wedge \ud \psi \\
+\left[\frac12\left((2+uR)^2-2mR^3u^2\right)\left(\partial_R \phi^2\right) +u^2\left(\frac12|\nabla_{\mathbb{S}^2}\phi|^2+\frac{\phi^2}{2}+b\frac{\phi^4}{4}\right)\right]\sin(\theta)\ud R \wedge \ud \theta \wedge \ud \psi\\
+\sin(\theta)\left[u^2\partial_u\phi-2(1+uR)\partial_R\phi\right]\left(-\partial_\theta\phi \ud u \wedge \ud R \wedge \ud \psi+\partial_\psi\phi \ud u \wedge \ud R \wedge \ud \theta \right)
\end{gather*}
The restriction of the energy 3-form can be written:
\begin{itemize}
\item to $\mathcal{H}_s$:
\begin{gather*}
i^\star_{\mathcal{H}_s}(\star \hat T^aT_{ab})=\left(u^2(\partial_u\phi)^2+R^2(1-2mR)u^2\partial_R\phi\partial_u\phi\right.\\\left.+R^2(1-2mR)\left( \frac{(2+uR)^2}{2s}-\frac{mu^2R^3}{s}-(1+uR)\right)(\partial_R \phi^2)\right.\\
+\left.\left(\frac{u^2R^2(1-2mR)}{s}+2(1+uR)\right)\left(\frac12|\nabla_{\mathbb{S}^2}\phi|+\frac{\phi^2}{2}+b\frac{\phi^4}{4}\right)\right) \sin(\theta)\ud u \wedge\ud \theta \wedge \ud \psi;
\end{gather*}
\item to $S_u$:
\begin{gather*}
i^\star_{S_u}(\star \hat T^aT_{ab})=\Bigg(\frac12\left((2+uR)^2-2mR^3u^2\right)\left(\partial_R \phi^2\right)\\ +u^2\left(\frac12|\nabla_{\mathbb{S}^2}\phi|^2+\frac{\phi^2}{2}+b\frac{\phi^4}{4}\right)\Bigg)\sin(\theta)\ud R \wedge \ud \theta \wedge \ud \psi;
\end{gather*}
\item to $\scri^+_{u_0}$
$$
i^\star_{\scri^+}(\star \hat T^aT_{ab})=\left(u^2(\partial_u\phi)^2+\|\nabla_{\mathbb{S}^2}\phi|+\phi^2+b\frac{\phi^4}{2}\right)\sin(\theta)\ud u  \wedge \ud \theta \wedge \ud \psi.
$$
\end{itemize}
\end{proposition}

\begin{proof} For the calculation to come, let us denote by $A$, the quantity:
$$
A=-\frac12\hat \nabla_c\phi \hat\nabla^c\phi +\frac{\phi^2}{2}+b\frac{\phi^4}{4}
$$
We calculate the general form of $\star T^aT_{ab}$:
\begin{gather*}
\star \hat T^aT_{ab}=\star (u^2\partial^a_u-2(1+uR)\partial^a_R)T_{ab}\\
=(u^2\partial_u\phi-2(1+uR)\partial_R\phi)\star \hat \nabla_b\phi+u^2A \partial^a_u\lrcorner \mu[\hat g]-2A(1+uR) \partial_R\lrcorner \mu[\hat g].
\end{gather*}
Using the expression of the volume form in the coordinates $(u,R,\omega_{\mathbb{S}^2})$ (equation \eqref{volumeformcoor}), we obtain:
\begin{equation*}\begin{array}{lcl}
\partial_u\lrcorner \mu[\hat g]&=&\ud R\wedge \ud^2\omega_{\mathbb{S}^2},\\
\partial_R\lrcorner \mu[\hat g]&=&-\ud u\wedge \ud^2\omega_{\mathbb{S}^2},
\end{array}
\begin{array}{lcl}
\partial_\theta\lrcorner \mu[\hat g]&=&\sin(\theta)\ud R\wedge \ud u \wedge\ud \theta,\\
\partial_\phi \lrcorner \mu[\hat g]&=&-\sin(\theta)\ud R\wedge\ud u \wedge\ud \phi.
\end{array}
\end{equation*}
$\hat \nabla \phi$ is written in the coordinates $(R,u, \omega_{S^2})$ as follows:
$$
\hat \nabla \phi = -\partial_R\phi \partial_u-(\partial_u\phi+R^2(1-2mR)\partial_R\phi)\partial_R-\hat\nabla_{\mathbb{S}^2}\phi;
$$
its norm is then:
\begin{gather*}
\hat \nabla_c \phi\hat \nabla^c \phi= (-\partial_R \phi)^2g(\partial_u, \partial_u)+2\partial_R\phi(\partial_u\phi+R^2(1-2m R)\partial_R\phi))g(\partial_u, \partial_R)-|\hat\nabla_{\mathbb{S}^2}\phi|^2\\
=-R^2(1-2mR)(\partial_R \phi)^2-2\partial_R\phi\partial_u\phi-|\hat\nabla_{\mathbb{S}^2}\phi|^2.
\end{gather*}
The Hodge dual of $\nabla_a\phi$ is calculated by splitting $\hat \nabla_a u$ over $(\ud u, \ud R, \ud \omega_{S^2})$:
\begin{eqnarray*}
\star\hat  \nabla_b \phi&=&\star \left(\partial_u\phi \ud u+\partial_R\phi \ud R+\partial_\theta\phi \ud \theta +\partial_\psi \phi \ud \psi\right)\\
&=&\partial_u\phi\ud R\wedge \ud \omega_{\mathbb{S}^2}-\partial_R\phi\ud u\wedge \ud \omega_{\mathbb{S}^2}
\end{eqnarray*}
The gradient of each coordinates is calculated:
\begin{equation*}
\begin{array}{cc}
\begin{array}{lcl}
\hat\nabla^b u&=&\hat g^{ab}\hat \nabla_a u=-\partial_R\\
\hat\nabla^b R&=&-\partial_u-R^2(1-2mR)\partial_R
\end{array}
&
\begin{array}{lcl}
\hat\nabla^b \theta&=&-\sin(\theta)\partial_{\theta}\\
\hat\nabla^b \psi&=&-\partial_{\psi},
\end{array}
\end{array}
\end{equation*}
and, as a consequence,
\begin{equation*}
\begin{array}{cc}
\begin{array}{lcl}
\star \ud u &=&\ud u\wedge \ud \omega_{\mathbb{S}^2}\\
\star \ud R &=& - \ud R \wedge \ud \omega_{\mathbb{S}^2}+R^2(1-2mR)\ud u\wedge \ud \omega_{\mathbb{S}^2}
\end{array}&
\begin{array}{lcl}
\star \ud \theta &=&-\sin{\theta}\ud u\wedge \ud R\wedge \ud \psi\\
\star \ud \psi &=&\ud u\wedge \ud R\wedge \ud \theta
\end{array},
\end{array}
\end{equation*}
so that $\star \nabla_a\phi$ is:
\begin{eqnarray*}
\star \nabla_a\phi&=&\partial_u\phi\ud u\wedge \ud \omega_{\mathbb{S}^2}+\partial_R\phi \left(- \ud R \wedge \ud \omega_{\mathbb{S}^2}+R^2(1-2mR)\ud u\wedge \ud \omega_{\mathbb{S}^2}\right)\\
&&-\partial_\theta\phi\sin{\theta}\ud u\wedge \ud R\wedge \ud \psi+\partial_\psi\ud u\wedge \ud R\wedge \ud \theta\\
&=&\left(\partial_u\phi +R^2(1-2mR)\partial_R\phi\right)\ud u\wedge \ud \omega_{\mathbb{S}^2}-\partial_R\phi\ud R\wedge \ud \omega_{\mathbb{S}^2}\\
&& -\partial_\theta\phi\sin{\theta}\ud u\wedge \ud R\wedge \ud \psi+\partial_\psi\ud u\wedge \ud R\wedge \ud \theta
\end{eqnarray*}

The energy 3-form is then:
\begin{gather*}
\star \hat T ^aT_{ab}=\left[u^2(\partial_u\phi)^2+R^2(1-2mR)\left(u^2\partial_R\phi\partial_u\phi-(1+uR)(\partial_R\phi\right)^2\right.\\
+2(1+uR)\left.\left(\frac12|\nabla_{\mathbb{S}^2}\phi|^2+\frac{\phi^2}{2}+b\frac{\phi^4}{4}\right)\right]\ud u \wedge \ud \omega_{\mathbb{S}^2}\label{energymainpart}\\
+\left[\frac12\left((2+uR)^2-2mR^3u^2\right)\left(\partial_R \phi^2\right) +u^2\left(\frac12|\nabla_{\mathbb{S}^2}\phi|^2+\frac{\phi^2}{2}+b\frac{\phi^4}{4}\right)\right]\ud R \wedge \ud \omega_{\mathbb{S}^2}\\
+\sin(\theta)\left[u^2\partial_u\phi-2(1+uR)\partial_R\phi\right]\left(-\partial_\theta\phi \ud u \wedge \ud R \wedge \ud \psi+\partial_\psi\phi \ud u \wedge \ud R \wedge \ud \theta \right)\label{energyrotpart}
\end{gather*}

\begin{remark}
To calculate the restrictions of $\star \hat T ^aT_{ab}$ to the hypersurfaces $\mathcal{H}_s$, $S_u$ and $\scri^+$, it is necessary to give the restrictions of each of the differentials of the coordinates. Nonetheless, since $\partial_\theta$ and $\partial_\phi$ are tangent to $\mathcal{H}_s$, $S_u$ and $\scri^+$, the only remaining 3-forms to consider when restricting to these hypersurfaces are $\ud u \wedge \ud \omega_{\mathbb{S}^2}$ and $\ud R \wedge \ud \omega_{\mathbb{S}^2}$. This means that only $\ud u$ and $\ud R$ should be taken care of.
\end{remark}

Noticing that
$$
\frac{\ud r^\ast}{\ud R}=\frac{-1}{R^2(1-2mR)},
$$
we get, on $\mathcal{H}_s$, defined in $\Omega_{u_0}^+$ by $u=-sr^\ast$:
$$
\left.\ud R\right|_{\mathcal{H}_s}= \frac{R^2(1-2mR)}{s}\left.\ud u\right|_{\mathcal{H}_s}=\frac{r^\ast R^2(1-2mR)}{|u|}\left.\ud u \right|_{\mathcal{H}_s}.
$$
The restriction to $\mathcal{H}_s$ of the energy 3-form is then:
\begin{gather*} 
i^\star_{\mathcal{H}_s}(\star \hat T^aT_{ab})=\left(\left[u^2(\partial_u\phi)^2+R^2(1-2mR)\left(u^2\partial_R\phi\partial_u\phi-(1+uR)(\partial_R\phi\right)^2\right.\right.\\
+2(1+uR)\left.\left(\frac12|\nabla_{\mathbb{S}^2}\phi|^2+\frac{\phi^2}{2}+b\frac{\phi^4}{4}\right)\right]\\
+\frac{R^2(1-2mR)}{s}\left[\frac12\left((2+uR)^2-2mR^3u^2\right)\left(\partial_R \phi^2\right) +u^2\left(\frac12|\nabla_{\mathbb{S}^2}\phi|^2+
\frac{\phi^2}{2}+b\frac{\phi^4}{4}\right)\right]\\
\ud u \wedge \ud \omega_{\mathbb{S}^2}\\
i^\star_{\mathcal{H}_s}(\star \hat T^aT_{ab})=\left(u^2(\partial_u\phi)^2+R^2(1-2mR)u^2\partial_R\phi\partial_u\phi\right.\\\left.+R^2(1-2mR)\left( \frac{(2+uR)^2}{2s}-\frac{mu^2R^3}{s}-(1+uR)\right)(\partial_R \phi^2)\right.\\
+\left.\left(\frac{u^2R^2(1-2mR)}{s}+2(1+uR)\right)\left(\frac12|\nabla_{\mathbb{S}^2}\phi|+\frac{\phi^2}{2}+b\frac{\phi^4}{4}\right)\right)\ud u \wedge \ud \omega_{\mathbb{S}^2}.
\end{gather*}

$\scri^+$ is defined by $R=0$; so the restriction to $\scri^+$ of the energy 3-form is:
$$
i^\star_{\scri^+}(\star \hat T^aT_{ab})=\left(u^2(\partial_u\phi)^2+\left(|\nabla_{\mathbb{S}^2}\phi|+\phi^2+b\frac{\phi^4}{2}\right)\right)\ud u  \wedge \ud \omega_{\mathbb{S}^2}.
$$

Finally, for $S_u$ which is defined by $\{u=constant\}$, we obtain:
$$
i^\star_{S_u}(\star \hat T^aT_{ab})=\left(\frac12\left((2+uR)^2-2mR^3u^2\right)\left(\partial_R \phi^2\right) +u^2\left(\frac12|\nabla_{\mathbb{S}^2}\phi|^2+\frac{\phi^2}{2}+b\frac{\phi^4}{4}\right)\right)\ud R \wedge \ud \omega_{\mathbb{S}^2}.
$$\end{proof}

\begin{proposition}\label{energyequivalenceschwarzschild1}
There exists $u_0$, such that the following energy estimates holds on $\mathcal{H}_s$ in $\Omega^+_{u_0}$:
$$
\int_{\mathcal{H}_s}i^\star_{\mathcal{H}_s}\left(\star \hat T^a T_{ab}\right)\approx \int_{\mathcal{H}_s}\left(u^2(\partial_u\phi)^2+\frac{R}{|u|}(\partial_R\phi)^2+|\nabla_{\mathbb{S}^2}\phi|^2+\frac{\phi^2}{2}+b\frac{\phi^4}{4}\right)\ud u \wedge \ud \omega_{\mathbb{S}^2}
$$
\end{proposition}
\begin{proof} Let us consider the expression of $i^\star_{\mathcal{H}_s}(\star \hat T^aT_{ab})$:
\begin{eqnarray} 
i^\star_{\mathcal{H}_s}(\star \hat T^aT_{ab})&=&\Big(u^2(\partial_u\phi)^2\label{se11}\\
&&\hspace{-1cm}+R^2u^2(1-2mR)\partial_R\phi\partial_u\phi\label{se21}\\
&&\hspace{-1cm}+R^2(1-2mR)\left( \frac{(2+uR)^2}{2s}-\frac{mu^2R^3}{s}-(1+uR)\right)(\partial_R \phi^2)\label{se31}\\
&&\hspace{-1cm}+\left(\frac{u^2R^2(1-2mR)}{s}+2(1+uR)\right)\left(\frac12|\nabla_{\mathbb{S}^2}\phi|+\frac{\phi^2}{2}+b\frac{\phi^4}{4}\right)\Big)\ud u \wedge \ud \omega_{\mathbb{S}^2}.\label{se41}
\end{eqnarray}
Each of these terms is estimated separately using lemma \ref{schwarzestimates} and the obtained estimates are summed. Let $\epsilon$ be a given positive number and let $u_0$ be the non positive constant associated to $\epsilon$ via lemma \ref{schwarzestimates}; $\epsilon$ will be chosen during the proof.

Nothing needs to be done for \eqref{se11}.

For \eqref{se41}, since $u$ is non-positive and $s=-\frac{u}{r^\ast}=\frac{|u|}{r^\ast}$, we have, on one side:
\begin{eqnarray*}
\left(\frac{u^2R^2(1-2mR)}{s}+2(1-|u|R)\right)&=& (R r^\ast)(R|u|)(1-2mR)+2(1-|u|R)\\
&\leq & (R r^\ast)(R|u|)(1-2mR)+2\\
&\leq& (1-\epsilon)(1+\epsilon)+2 
\end{eqnarray*}
and, on the other side:
\begin{eqnarray}
\left(\frac{u^2R^2(1-2mR)}{s}+2(1-|u|R)\right)&=&(R r^\ast)(1-2mR)(R|u|)+2(1-|u|R)\nonumber\\
&\geq&1\cdot(1-\epsilon)\cdot(R|u|)+2(1-|u|R)\nonumber\\
&\geq& 2-(1+\epsilon)(R|u|)\nonumber\\
&\geq& 2-(1+\epsilon)(1+\epsilon)\nonumber\\
&\geq& 1-2\epsilon-\epsilon^2.\label{choiceeps1}
\end{eqnarray}
$\epsilon$ is chosen such as \eqref{choiceeps1} is positive.

For \eqref{se31}, the proof is slightly more complicated. We have, since $u$ is non-positive and $s=-\frac{u}{r^\ast}=\frac{|u|}{r^\ast}$, on one hand:
\begin{gather}
R^2(1-2mR)\left( \frac{(2+uR)^2}{2s}-\frac{mu^2R^3}{s}-(1+uR)\right)(\partial_R \phi)^2\nonumber\\
=R^2(1-2mR)\left( \frac{r^\ast(2-|u|R)^2}{2|u|}-(mR)(R|u|)(Rr^\ast)-(1+uR)\right)(\partial_R \phi)^2\nonumber\\
=\left(\frac{R}{|u|}(\partial_R \phi)^2\right)(1-2mR)(Rr^\ast)  \left(\frac{(2-|u|R)^2}{2}-(mR)(R|u|)^2-\frac{|u|}{r^\ast}+\frac{|u|}{r^\ast}R|u| \right)\label{ineg2248}\\
\leq\left(\frac{R}{|u|}(\partial_R \phi)^2\right)\cdot 1 \cdot (1+\epsilon)\left(\frac{(3+\epsilon^2)^2}{2}+1\cdot (1+\epsilon)\right). \nonumber
\end{gather}
Starting from equation \eqref{ineg2248}, the lower bound for \eqref{se31} is obtained as follows, setting $X=|u|R$:
\begin{gather*}
\left(\frac{R}{|u|}(\partial_R \phi)^2\right)(1-2mR)(Rr^\ast)\left(\frac{(2-|u|R)^2}{2}-(mR)(R|u|)^2-\frac{|u|}{r^\ast}+\frac{|u|}{r^\ast}R|u|\right)\\
\geq \left(\frac{R}{|u|}(\partial_R \phi)^2\right)(1-\epsilon)\cdot \left(\frac{(2-X)^2}{2}(Rr^\ast)-(mR)(Rr^\ast)X^2-X+X^2\right)\\
\geq \left(\frac{R}{|u|}(\partial_R \phi)^2\right)\frac{(1-\epsilon)}{2}\left(4-6X+3X^2-\epsilon(1+\epsilon)X^2\right)
\end{gather*}
The polynomial $4-6X+3X^2$ reaches its minimum for $X=1$ and equals $1$ at $X=1$ so that:
\begin{gather*}
\left(\frac{R}{|u|}(\partial_R \phi)^2\right)(1-2mR)(Rr^\ast)\left(\frac{(2-|u|R)^2}{2}-(mR)(R|u|)^2-\frac{|u|}{r^\ast}+\frac{|u|}{r^\ast}R|u|\right)\\
\geq \left(\frac{R}{|u|}(\partial_R \phi)^2\right)\frac{(1-\epsilon)}{2}\left(1-\epsilon(1+\epsilon)^3\right)
\end{gather*}

To deal with \eqref{se21}, we write:
$$
|R^2u^2(1-2mR)\partial_R\phi\partial_u\phi|=(1-2mR)\left(R^2 |u|\sqrt{\frac{2}{3}}\partial_R\phi\right)\left(\sqrt{\frac{3}{2}}u\partial_u\phi\right),
$$
so that:
\begin{eqnarray*}
|R^2u^2(1-2mR)\partial_R\phi\partial_u\phi|&\leq&(1-2mR) \frac{1}{2}\left(\frac{3}{2}\left(u\partial_u\phi\right)^2+\frac{2}{3}(R^3|u|^3)\frac{R}{|u|}(\partial_R\phi)^2\right)\\
&\leq&\frac{3}{4}\left(u\partial_u\phi\right)^2+\frac{1}{3}(1+\epsilon)^3\frac{R}{|u|}(\partial_R\phi)^2
\end{eqnarray*}

Finally, the following equivalence estimates hold:
$$
c_\epsilon \int_{\mathcal{H}_s}\left(u^2(\partial_u\phi)^2+\frac{R}{|u|}(\partial_R\phi)^2+|\nabla_{\mathbb{S}^2}\phi|^2+\frac{\phi^2}{2}+b\frac{\phi^4}{4}\right)\ud u \wedge \ud \omega_{\mathbb{S}^2}\leq \int_{\mathcal{H}_s}i^\star_{\mathcal{H}_s}\left(\star \hat T^a T_{ab}\right)
$$
and 
$$
 \int_{\mathcal{H}_s}i^\star_{\mathcal{H}_s}\left(\star \hat T^a T_{ab}\right)\leq C_{\epsilon}\int_{\mathcal{H}_s}\left(u^2(\partial_u\phi)^2+\frac{R}{|u|}(\partial_R\phi)^2+|\nabla_{\mathbb{S}^2}\phi|^2+\frac{\phi^2}{2}+b\frac{\phi^4}{4}\right)\ud u \wedge \ud \omega_{\mathbb{S}^2}
$$
where 
$$
C_{\epsilon}=\max\left(1, \frac{1}{3}(1+\epsilon)^3, (1+\epsilon)\left(\frac{(3+\epsilon^2)^2}{2}+(1+\epsilon)\right), (1-\epsilon)(1+\epsilon)+2 \right)
$$
and
$$
c_{\epsilon}=\min\left(\frac14, \frac{1}{6}-\epsilon P(\epsilon), 1-2\epsilon-\epsilon^2\right),
$$
where $P(\epsilon)$ is a polynomial in $\epsilon$. $\epsilon$ is chosen such that the constant $c_{\epsilon}$ is positive. Using lemma \ref{schwarzestimates}, there exists $u_0$, negative, $|u_0|$ large enough, such that the estimates of the coordinates hold in $\Omega^+_{u_0}$ and, consequently, the equivalence is true on this domain.\end{proof}

\begin{lemma}\label{errorschwarzschild2}
The error is given by:
$$
\hat\nabla^{a}\left(\hat T^bT_{ab}\right)=4mR^2(3+uR)\left(\partial_R\phi\right)^2+\left(1-12mR\right)\phi\left(u^2\partial_u \phi-2(1+uR)\partial_R \phi\right)+\hat T^a \hat \nabla_a b \frac{\phi^4}{4}.
$$
\end{lemma}
\begin{proof} The Killing  form of the vector $\hat T^a$ is calculated via the Lie derivative of the metric:
\begin{gather*}
\hat \nabla_{(a}\hat T_{b)}=\textrm{L}_{\hat T}\hat g\\
=\textrm{L}_{\hat T}(R^2(1-2mR)) (\ud u)^2+2R^2(1-2mR) \textrm{L}_{\hat T}(\ud u) \ud u-2\textrm{L}_{\hat T}(\ud u)\ud R-2\textrm{L}_{\hat T}(\ud R)\ud u-2\textrm{L}_{\hat T}(\ud \omega_{\mathbb{S}^2})\\
=-4R(1+uR)(1-3mR)(\ud u)^2+2R^2(1-2mR)\ud \left(\textrm{L}_{\hat T}(u)\right)\ud u-2\ud \left(\textrm{L}_{\hat T}(u) \right)\ud R-2\ud \left(\textrm{L}_{\hat T}(R) \right)\ud u-0\\
=-4R(1+uR)(1-3mR)(\ud u)^2+2R^2(1-2mR)\ud (u^2)\ud u-2\ud \left(u^2 \right)\ud R-2\ud \left(-2(1+uR) \right)\ud u\\
=-4R(1+uR)(1-3mR)(\ud u)^2+4uR^2(1-2mR)(\ud u)^2-4u \ud u\ud R+4(R\ud u+u\ud R)\ud u\\
=(-2R(1+uR)(1-3mR)+4  uR^2(1-2mR)+4R)(\ud u)^2\\
=(12 mR^2+4muR^3)(\ud u^2),
\end{gather*}
or, in the vector form:
$$
\hat \nabla^{(a}\hat T^{b)}=4mR^2(3+uR)\partial_R\partial_R.
$$
A direct consequence of the above formula is that the Killling form is trace-free:
\begin{eqnarray*}
\nabla^{a}\hat T_{a}&=&4mR^2(3+uR)  \hat g(\partial_R,\partial_R)\\
&=&0.
\end{eqnarray*}
The scalar curvature of the rescaled metric is given by equation \eqref{conformalchangecurvature}; choosing $\Omega=R$ gives:
\begin{eqnarray*}
\frac16 \scal_{\hat g}&=&R^3\nabla_b \nabla^b R\\
&=&2 m R.
\end{eqnarray*}
Finally, the error term is given by:
\begin{eqnarray*}
\nabla^{(a}\hat T^{b)}T_{ab}&=&4mR^2(3+uR)\left( \partial_R\phi\right)^2+\left(1-12mR\right)(\hat T^a\nabla_a \phi)\phi +\hat T^a \hat \nabla_a b \frac{\phi^4}{4}\\
&=&4mR^2(3+uR) \left(\partial_R\phi\right)^2+\left(1-12mR\right)\phi\left(u^2\partial_u \phi-2(1+uR)\partial_R \phi\right)\\
&&+\hat T^a \hat \nabla_a b \frac{\phi^4}{4}. 
\end{eqnarray*}\end{proof}

\begin{proposition}\label{eqenergie333}
The restriction of the energy 3-form to $\Sigma_t$ is given by:
\begin{gather*}
i^\star_{\Sigma_t}\left(\star\hat T^aT_{ab}\right)=
\Bigg\{\frac{(f^c\hat \nabla_c\phi)^2}{2(1+\sum_{i=1,2,3}(\hat g_{cd}f^ce^dc_i)^2)}\\+\frac{1}{2}\left(\sum_{i=1,2,3}\left(1-\frac{(\hat g_{cd}f^ce^dc_i)^2}{1+\sum_{i=1,2,3}(\hat g_{cd}f^ce^dc_i)^2}\right)(e^a_{\textbf{i}}\hat \nabla_a\phi)^2\right)+\frac{\phi^2}{2}+b\frac{\phi^4}{4}\Bigg\}\frac{\hat g_{cd}\hat T^c\hat T^d}{\hat g_{cd}\hat T^ce^d_0}e^a_1\wedge e^a_2\wedge e^a_3
\end{gather*}
and, as a consequence, the following equivalence holds, for all $t$ in $[0,1]$:
$$
\int_{\Sigma_t}i^\star_{\Sigma_t}\left(\star\hat T^aT_{ab}\right)\approx\int_{\Sigma_t}\left((f^c\hat \nabla_c\phi)^2+\sum_{i=1,2,3}(e^a_{\textbf{i}}\hat \nabla_a\phi)^2+\frac{\phi^2}{2}+b\frac{\phi^4}{4}\right)e^a_1\wedge e^a_2\wedge e^a_3.
$$
We denote by $E(\Sigma_t)$ this energy:
$$
E(\Sigma_t)=\int_{\Sigma_t}\left((f^c\hat \nabla_c\phi)^2+\sum_{i=1,2,3}(e^a_{\textbf{i}}\hat \nabla_a\phi)^2+\frac{\phi^2}{2}+b\frac{\phi^4}{4}\right)e^a_1\wedge e^a_2\wedge e^a_3.
$$
\end{proposition} 
\begin{remark}
This proposition together with proposition \ref{energyequivalence} states that the energy on a spacelike slice for two uniformly timelike (for the metric $\hat g$) vector fields are equivalent. This justifies that we write in the same way the energy in proposition \ref{energyequivalence} and in this proposition.
\end{remark}
\begin{proof}  The proof is essentially the same as the case where the vector field $\hat T^a$ is orthogonal to the foliation

The strategy of the proof is the same as usual: the geometric objects are split over the basis $(f^a,e^a_{\textbf{i}})_{i=1,2,3}$ where $f^a$ is transverse to $\Sigma_t$ and $(e_i^a)_{i=1,2,3}$ tangent to $\Sigma_t$.

Considering the 4-form
$$
f_a\wedge e^a_1\wedge e^a_2\wedge e^a_3
$$
 $f^a$ is decomposed as follows:
$$
f^a=\beta e^a_0-\sum_{i=1,2,3}\delta^\textbf{i}  e^a_{\textbf{i}} 
$$  
with 
$$ 
\delta^\textbf{i}=\hat g_{cd}f^c e^d_{\textbf{i}} \text{ and } \beta^2=1+\sum_{i=1,2,3}(\delta^\textbf{i})^2.
$$
The volume form $\mu[\hat g]$ satisfies:
\begin{equation*}
\mu[\hat g]=\frac{1}{\beta}f_a\wedge e^a_1\wedge e^a_2\wedge e^a_3
\end{equation*}
and its contraction with the vector $f^a$ is:
$$
 f^a\lrcorner \mu[\hat g] = \left(\beta e^a_0-\sum_{i=1,2,3}\delta^\textbf{i}  e^a_{\textbf{i}} \right)\lrcorner \mu[\hat g] = \beta e^a_1\wedge e^a_2\wedge e^a_3
$$
as a consequence, we have: 
\begin{equation*}
i^\star_{\Sigma_t}(\star \hat T_a)=||\hat T^a|| i^\star_{\Sigma_t}(\star f_a)=||\hat T^a||f^a\lrcorner\mu[\hat g]= ||\hat T^a||\beta  e^a_1\wedge e^a_2\wedge e^a_3.
\end{equation*}

We then deal with $\hat \nabla^c \phi$  which can be written:
\begin{equation*}
\hat \nabla^c \phi= b f^c-\sum_{i=1,2,3}a^{\textbf{i}} e^c_{\textbf{i}} 
\end{equation*}
where
\begin{equation*}
b=\frac{\hat g_{ab}\hat \nabla^a\phi e^b_0}{\beta}=\frac{f^a\hat \nabla_a\phi +\sum_{i=1,2,3}\delta^{\textbf{i}}e^a_{\textbf{i}} \hat \nabla_a\phi}{\beta^2} \text{ and } a_{\textbf{i}} =e^a_{\textbf{i}}\hat \nabla_a\phi-b\delta^{\textbf{i}}\text{ for }i\in\{1,2,3\}.
\end{equation*}
Consequently, its norm is:
\begin{eqnarray*}
\hat\nabla^c\phi\hat\nabla_c\phi&=&\frac{(f^a\hat \nabla_a\phi +\sum_{i=1,2,3}\delta^{\textbf{i}}e^a_{\textbf{i}}\hat \nabla_a\phi)^2}{\beta^2}-\sum_{i=1,2,3}(e^a_{\textbf{i}}\hat \nabla_a\phi)^2
\end{eqnarray*}
and the restriction to $\Sigma_t$ of $\star \nabla_b \phi$ is:
\begin{eqnarray}
i^\star_{\Sigma_t}(\star \nabla_b \phi)&=&\frac{(f^a\hat \nabla_a\phi +\sum_{i=1,2,3}\delta^{\textbf{i}}e^a_{\textbf{i}}\hat \nabla_a\phi)}{\beta^2}i^\star_{\Sigma_t}(\star f_a)\nonumber\\
&=&\left(f^a\hat \nabla_a\phi +\sum_{i=1,2,3}\delta^{\textbf{i}}e^a_{\textbf{i}}\hat \nabla_a\phi\right)\frac{ e^a_1\wedge e^a_2\wedge e^a_3}{\beta} .
\end{eqnarray}

Using these results, the energy 3-form is given by:
\begin{gather*}
i^\star_{\Sigma_t}(\star \hat T^aT_{ab})
=\hat T^a\hat \nabla_a \phi\left(\left(f^a\hat \nabla_a\phi +\sum_{i=1,2,3}\delta^{\textbf{i}}e^a_{\textbf{i}}\hat \nabla_a\phi\right)\frac{ e^a_1\wedge e^a_2\wedge e^a_3}{\beta}  \right)\\
+\left(-\frac{1}{2}\left(\frac{(f^a\hat \nabla_a\phi +\sum_{i=1,2,3}\delta^{\textbf{i}}e^a_{\textbf{i}}\hat \nabla_a\phi)^2}{\beta^2}-\sum_{i=1,2,3}(e^a_{\textbf{i}}\hat \nabla_a\phi)^2
\right)+ \frac{\phi^2}{2}+b\frac{\phi^4}{4}\right)||\hat T^a||\beta e^a_1\wedge e^a_2\wedge e^a_3\\
=\Bigg\{\frac12\left(  (f^c\hat \nabla_c\phi)^2+\sum_{i=1,2,3}(e^a_i\hat \nabla_a \phi)^2(\beta^2-(\delta^\textbf{i})^2)\right)     +\beta ^2\left(\frac{\phi^2}{2}+b\frac{\phi^4}{4}\right)\Bigg\}\frac{||\hat T^a||^2}{\beta}e^a_1\wedge e^a_2\wedge e^a_3
\end{gather*}
To get the equivalence, it is sufficient to notice that:
\begin{itemize}
\item $\hat g_{cd}\hat T^c\hat T^d$ is a positive function over the compact $\overline{V}$ and, as such, is bounded below and behind by two positive constants;
\item as already noticed in remark \ref{constructiontime}, $\beta$ is a positive function over  $\overline{V}$;
\item and, finally, the scalar products $\beta$, $\delta_\textbf{i}$ and the difference $\beta^2-\delta_\textbf{i}$ are  bounded below by $1$ and above by a certain constant since we are working on a compact set.
\end{itemize}

The energy equivalence then holds:
$$
\int_{\Sigma_t}i_{\Sigma_t}^\star\left(\star\hat T^aT_{ab}\right)\approx\int_{\Sigma_t}\left(\frac{(f^c\hat \nabla_c\phi)}{2}^2+\frac12\sum_{i=1,2,3}(e^a_{\textbf{i}}\hat \nabla_a\phi)^2+\frac{\phi^2}{2}+b\frac{\phi^4}{4}\right)e^a_1\wedge e^a_2\wedge e^a_3
$$\end{proof}

\subsection{Local characteristic Cauchy problem for small data}
It has been considered as interesting to give here a proof of the existence of a solution for small data because of its technical simplicity.

We choose to work here with a function in $H^1(\scri^+)$ with compact support which contains neither $i^0$ nor $i^+$. Let $\tau$ be a "reverse" time function on $\hat M$, in the sense that its gradient is past directed with respect to $\hat T^a$ (defined in section \ref{estimatesonu}). We assume that $\tau(i^+)=0$. $\hat M$ is endowed with an orthonormal basis $(e^a_\textbf{i})_{i=0,1,2,3} $ such that $e^a_0$ is colinear to $\hat \nabla \tau$ and $(e^a_\textbf{i})_{i=1,2,3}$ is tangent to the time slices $\{\tau = constant\}$. The integral flow associated with $\partial_\tau$ is denoted by $\Phi_\tau$.

\begin{figure}
\begin{center}
\hspace{0cm}\scalebox{2.1}{\input{SAVE3.pstex_t}}
\end{center}
\end{figure}

Let $\epsilon$ be a positive constant smaller than 1 and consider $\Phi_\epsilon(i^+)$. Let $\mathcal{S}$ be a uniformly spacelike hypersurface for the metric $\hat g$ between $\scri^+$ and $\{\Phi_{\epsilon}(p)|p\in \scri^+\}$. We assume that $\mathcal{S}$ is uniformly spacelike, transverse to $\scri^+$ in the past of the support of the characteristic data, $\theta$, and contains $\Phi_\epsilon(i^+)$. 
\begin{remark} The geometric framework is then exactly the same as in section \ref{estimatesonv}: the timelike vector field is not normal to the foliation. Nonetheless, since the hypersurface $\mathcal{S}$ is uniformly spacelike, the same estimates as in section \ref{estimatesonv} hold without the nonlinearity.
\end{remark}

 Finally, the future of $S$ in $\hat M$ is foliated by the surfaces $\mathcal{S}_\tau=\{\Phi_{\epsilon -\tau}(p)| p\in \mathcal{S}\}$ for $\tau$ in $[0,\epsilon]$ so that $\mathcal{S}=\mathcal{S}_{\epsilon}$ and $\mathcal{S}_0=\{i^+\}$. The future of $\mathcal{S}$ is denoted by $\mathcal{R}$ and the subset of $\hat M$ between $\mathcal{S}_0$ and $\mathcal{S}_\tau$, $\mathcal{R}_\tau$.

The solution of the nonlinear problem is approximated via solutions of the linear problem on $\scri^+$. Hormander solved this problem in \cite{MR1073287}:
\begin{proposition}[Hormander]\label{Hormander2}
Let us consider the linear inhomogeneous characteristic Cauchy problem on $\scri^+$:
\begin{equation*}\left\{
\begin{array}{l}
\hat \square \phi+ \frac16 \scal_{\hat g}\phi=f\\
\phi\big|_{\scri^+}=\theta\in H^1(\scri^+).
\end{array}\right.
\end{equation*}
where $f$ is in $L^1_{loc}H^1$ and $\theta$ is a function whose compact support does not contain $i^+$ or $i^0$.
Then, this problem admits a unique global solution in the future of $\Sigma_0$ in $C^0([0,\epsilon], H^1(\mathcal{S}_\tau))$.
\end{proposition}

 Using proposition \ref{Hormander2} and estimates for the linear problem, the following theorem holds:
\begin{theorem}\label{localcharprob}
Let us consider the nonlinear characteristic Cauchy problem on $\scri^+$:
\begin{equation*}\left\{
\begin{array}{l}
\hat \square u+ \frac16 \scal_{\hat g}u+bu^3=0\\
u\big|_{\scri^+}=\theta\in H^1(\scri^+).
\end{array}\right.
\end{equation*}
where $\theta$ is a function whose compact support does not contain $i^+$ or $i^0$.
Then, for $||\theta||_{H^1(\scri^+)}$ small enough, there exists a uniformly spacelike hypersurface $\mathcal{S}$ close enough to $\scri^+$ such that this problem admits a smooth global solution on $\mathcal{R}$ in $C^0([0,\epsilon], H^1(\mathcal{S}_\tau))$.
\end{theorem}
\begin{remark} The proof of the well-posedness in $C^0([0,\epsilon], H^1(\mathcal{S}_\tau))$ is given in section \ref{continuitypart} were the geometric estimates required to obtain it are established (see theorem \ref{continuity} which remains true in that context).
\end{remark}
\begin{proof} Let $u_0$ be a solution on $\mathcal{R}$ of the problem:
\begin{equation*}\left\{
\begin{array}{l}
\hat \square \phi+ \frac16 \scal_{\hat g}\phi=0\\
\phi\big|_{\scri^+}=\theta\in H^1(\scri^+).
\end{array}\right.
\end{equation*}
and let $(u_n)_{n\in \N}$ be the sequence of smooth functions on $\mathcal{R}$ defined by the recursion:
\begin{equation*}\left\{
\begin{array}{l}
\hat \square u_{n+1}+ \frac16 \scal_{\hat g}u_{n+1}+bu_{n}^3=0\\
u_{n+1}\big|_{\scri^+}=\theta\in H^1(\scri^+).
\end{array}\right.
\end{equation*}
The sequence defined by the difference of two consecutive terms of this sequence is denoted by $(\delta_n=u_{n+1}-u_n)_{n\in\N}$. For ${n\in\N}$, the smooth function $\delta_n$ satisfies the Cauchy problem:
\begin{equation*}\left\{
\begin{array}{l}
\hat \square \delta_{n}+ \frac16 \scal_{\hat g}\delta_n=-b(u_n^2+u_nu_{n-1}+u^2_{n-1})\delta_{n-1}\\
\delta_{n}=0 \text{ on }\scri^+.
\end{array}\right.
\end{equation*}

The proof of the convergence is made in two steps: the first one consists in proving that, for initial data which are small enough, the sequence $(u_n)$ is bounded; the second part proves the convergence of $(u_n)$ by showing that the sequence $(\delta_n)$ is summable.

\begin{proposition}
For $||\theta||_{H^1(\scri^+)}$ small enough, the sequence $ \sup_{\tau\in[0,\epsilon]}||u_n||_{H^1(\mathcal{S}_\tau)}$ is bounded. 
\end{proposition}
\begin{proof} Let $n$ be a integer greater than 1. Let us finally consider the energy tensor associated with the linear wave equation:
$$
T_{ab}=\hat \nabla_a u_{n+1} \hat \nabla_b u_{n+1} +\hat g_{ab}\left(-\frac12\hat\nabla_cu_{n+1}\hat\nabla^cu_{n+1}+\frac{u_{n+1}^2}{2}\right).
$$
The energy associated with a time slice $\mathcal{S}_\tau$ is written as:
$$
E_{u_n}(\mathcal{S}_\tau)=\int_{\mathcal{S}_\tau}\left(\sum_{i=0,1,2,3}(e^a_\textbf{i} \hat \nabla_a u_{n})^2+\frac{u_n^2}{2}\right)\ud \mu_{\mathcal{S}_\tau}
$$
and it is equivalent to $\int_{\mathcal{S}_\tau} i^\star(\star e^a_0T_{ab})$ (with constants which only depend on the geometric data, the $L^\infty$-norm of $b$ and the Killing form of $e^a_0$).

The error term associated to this energy tensor is:
$$
\hat \nabla^a \left(e_0^b T_{ab}\right)=\hat \nabla^{(a} e_0^{b)} T_{ab}+u_{n+1} e^a_0\hat\nabla_a u_{n+1}-\frac16\scal_{\hat g} u_{n+1} e^a_0\hat\nabla_a u_{n+1} - (e^a_0\hat\nabla_a u_{n+1})bu_n^3,
$$
which is smaller than, in absolute value:
$$
|\hat \nabla^a e_0^b T_{ab}|\leq C\left(\sum_{i=0}^3 (e^a_\textbf{i}\hat \nabla_a u_{n+1})^2 +u_{n+1}^2\right)+ Cbu_{n}^6,
$$
where $C$ is a positive constant depending on $\sup\left(|\scal_{\hat g}|\right)$, $\sup\left(||\hat \nabla^{(a} e_0^{b)}||\right)$ and the foliation $\mathcal{S}_\tau$.

The next step consists in using the Sobolev embedding of $H^1(\mathcal{S}_\tau)$ in $L^6(\mathcal{S}_\tau)$. There exist two obstacles to the use of this embedding:
\begin{enumerate}
\item the first is the fact that the estimates of the nonlinearity must be uniform over the foliation in the sense that it must not depend on the parameter $\tau$ of the foliation (or the Sobolev constant must be the same all along the foliation) and does not explode despite the behavior of the Sobolev constant ;
\item the second comes from the fact that we must deal with the singularity in $i^+$.
\end{enumerate}

To deal with the second problem, the manifold $\hat M$ is extended beyond $\scri^+$ by pulling backwards the hypersurface $\mathcal{S}$ through the flow associated with the vector field $\partial_\tau$ of the time function $\tau$. Since the regularity of the metric is arbitrarily smooth at $i^+$ (say, at least $C^2$, in order to insure the existence of the different curvatures), this gives an extension as a smooth Lorentzian manifold of the manifold $(\hat M, \hat  g)$ in the neighborhood of $i^+$. $\scri^+$ is then the past light cone from $i^+$ obtained from a $C^2$-extension of the metric $\hat g$ behind $i^+$.

To obtain a Sobolev embedding from $H^1(\mathcal{S}_\tau)$ into $L^6(\mathcal{S}_\tau)$ uniformly in $\tau$, it is necessary to control the blow-up of the Sobolev norm. This is achieved in the same way as in section \ref{continuitypart} using proposition \ref{blowup2}. A direct consequence of the decay of the function $b$ as stated in assumption \ref{A4} is then the following inequality:
\begin{eqnarray*}
\int_{\mathcal{S}_\tau}|\hat \nabla^a e_0^b T_{ab}|\mu_{\mathcal{S}_\tau}&\leq& CE_{u_{n+1}}(\mathcal{S}_\tau)+C||u_n||^6_{H^1(\mathcal{S}_\tau)}\\
&\leq& CE_{u_{n+1}}(\mathcal{S}_\tau) +C\left(\sup_{\tau\in [0,\epsilon]}{E_{u_{n}}(\mathcal{S}_\tau)}\right)^3.
\end{eqnarray*}
As a consequence, there exists a constant $\tilde C$ which depends only on the scalar curvature, the Killing form of $ e^a_0$, the decay of $b$ and Sobolev constants such that:
$$
E_{u_{n+1}}(\mathcal{S}_\tau)\leq \tilde C\left( \int_{0}^\tau E_{u_{n+1}}(\mathcal{S}_r)\ud r+||\theta||^2_{H^1(\scri^+)}+\left(\sup_{r\in [0,\epsilon]}\left( E_{u_{n}}(\mathcal{S}_r) \right)\right)^3\right).
$$
\begin{remark}
The constant $\tilde C$ can be chosen arbitrarily high. As a consequence, it is rescaled later without consequence for the proof (see remark \ref{remarqueconstant} in the proof of proposition \ref{contraction}). 
\end{remark}
Using Stokes theorem and Gronwall lemma, as in section \ref{estimatesonu}, the energy of $u_{n+1}$ satisfies:
$$
E_{u_{n+1}}(\mathcal{S}_\tau)\leq \tilde C\exp(\tilde C \epsilon)\left(||\theta||^2_{H^1(\scri^+)}+\left(\sup_{\tau\in [0,\epsilon]}{E_{u_{n}}(\mathcal{S}_\tau) }\right)^3\right).
$$
For $n=0$, we have:
$$
E_{u_{0}}(\mathcal{S}_\tau)\leq \tilde C\exp(\tilde C \epsilon)||\theta||_{H^1(\scri^+)}^2.
$$
We denote by $(C_n)$ the sequence defined by:
$$
C_n=\sup_{\tau\in [0,\epsilon]}\left\{E_{u_n}(\mathcal{S}_\tau)\right\}.
$$
This sequence satisfies the inequality:
$$
\forall n \in \N, C_{n+1}\leq \underbrace{\tilde C\exp(\tilde C \epsilon)}_{\alpha}\left(C_n^3+\underbrace{||\theta||^2_{H^1(\scri^+)}}_{\beta}\right)
\text{ with }  
C_0\leq \tilde C\exp(\tilde C \epsilon)||\theta||_{H^1(\scri^+)}^2.
$$
Let us then consider the sequence $(c_n)_n$ defined by:
\begin{equation*}\left\{
\begin{array}{lcl}
c_0&=&\alpha \beta\\
c_{n+1}&=&\alpha\left(c_n^3+\beta\right).
\end{array}
\right.
\end{equation*} 
The purpose is to choose correctly $||\theta||_{H^1(\scri^+)}$ such that this sequence is bounded. The function $x\mapsto  \alpha\left(x^3+\beta\right)$ has three fixed points provided that the discriminant of the polynomial $X^3-\frac{1}{\alpha}X+\beta$ satisfies:
\begin{equation}\label{pointfixe1}
\beta^2-\frac{4}{27\alpha^3}<0, \text{ ie } \frac{4}{27}> (\tilde C\exp(\tilde C \epsilon))^3||\theta||_{H^1(\scri^+)}^4.
\end{equation}
Using Cardano's formulae, its three zeros are:
\begin{eqnarray}
\lambda_0&=&\sqrt{\frac{4}{3\alpha}}\cos\left(\frac13\arccos\left(-\sqrt{\frac{27\beta^2\alpha^3}{4}}\right)\right)\label{root0}\\
\lambda_1&=&\sqrt{\frac{4}{3\alpha}}\cos\left(\frac13\arccos\left(-\sqrt{\frac{27\beta^2\alpha^3}{4}}\right)+\frac{2\pi}{3}\right)\label{root1}\\
\lambda_2&=&\sqrt{\frac{4}{3\alpha}}\cos\left(\frac13\arccos\left(-\sqrt{\frac{27\beta^2\alpha^3}{4}}\right)+\frac{4\pi}{3}\right).\label{root2}
\end{eqnarray}
Since 
\begin{eqnarray*}
\frac13\arccos\left(-\sqrt{\frac{27\beta^2\alpha^3}{4}}\right)&\in&\left[\frac{\pi}{6},\frac{\pi}{3}\right]\\
\frac13\arccos\left(-\sqrt{\frac{27\beta^2\alpha^3}{4}}\right)+\frac{2\pi}{3}&\in&\left[\frac{5\pi}{6},\pi\right]\\
\frac13\arccos\left(-\sqrt{\frac{27\beta^2\alpha^3}{4}}\right)+\frac{4\pi}{3}&\in&\left[\frac{3\pi}{2},\frac{5\pi}{3}\right],
\end{eqnarray*}
these roots can be compared as follows:
$$
\lambda_1<0<\lambda_2<\lambda_0.
$$
The fixed points $\lambda_0$ and $\lambda_1$ are repulsive whereas the fixed point $\lambda _2$ is attractive. As a consequence, if the (positive) initial condition $c_0$ is below the positive repulsive fixed point  (the greater fixed point $\lambda _0$) of the function $x\mapsto  \alpha\left(x^3+\beta\right)$, that is to say if
\begin{gather*}
\sqrt{\frac{4}{3\alpha}}\cos\left(\frac13\arccos\left(-\sqrt{\frac{27\beta^2\alpha^3}{4}}\right)\right)\geq \alpha\beta\\
3\cos\left(\frac13\arccos\left(-\sqrt{\frac{27\beta^2\alpha^3}{4}}\right)\right)\geq \sqrt{\frac{27\beta^2\alpha^3}{4}},
\end{gather*}
the sequence $(c_n)$ converges towards $\lambda_2$. This inequality is always satisfied as soon as:
$$
\beta^2-\frac{4}{27\alpha^3}<0.
$$
As a consequence, assuming that 
$$
 \frac{4}{27 (\tilde C\exp(\tilde C \epsilon))^3}> ||\theta||_{H^1(\scri^+)}^4,
$$
 the sequence $(c_n)$ converges to the remaining attractive fixed point $\lambda_2$; $(c_n)$ is bounded and so is $(C_n)$, which is the expected result.

 Another useful consequence of the convergence of the sequence $(c_n)$ is the following. The limit of this sequence satisfies:
\begin{eqnarray*}
\lambda_2\leq\sqrt{\frac{4}{3\alpha}}=\sqrt{\frac{4}{3\tilde C\exp(\tilde C \epsilon)}}.
 \end{eqnarray*}
 As a consequence, there exists a integer $n_0$ such that:
 \begin{equation}\label{estimatesequence}
 \forall n\geq n_0, \sup_{\tau\in [0,\epsilon]}\left\{E_{u_n}(\mathcal{S}_\tau)\right\} \leq c_n  \leq  2 \sqrt{\frac{4}{3\tilde C\exp(\tilde C \epsilon)}}.
 \end{equation}\end{proof}

\begin{proposition}\label{contraction} The sequence $(u_n)$ converges on $\mathcal{R}$ in $C^0([0,\epsilon],H^1(\mathcal{S}_\tau))$, that is to say in the norm $\left(\sup_{\tau\in[0,\epsilon]}||u||_{H^1(\mathcal{S}_\tau)}^2\right)$.
\end{proposition}
\begin{proof} The method is exactly the same as in the previous proposition. Let $n$ be a positive integer and consider the energy tensor associated with the linear wave equation for the function $\delta_n$
$$
T_{ab}=\hat \nabla_a \delta_{n} \hat \nabla_b \delta_{n} +\hat g_{ab}\left(-\frac12\nabla_c\delta_{n}\nabla^c\delta_{n}+\frac{\delta_{n}^2}{2}\right).
$$
The energy associated with a time slice $\mathcal{S}_\tau$ is written as in the previous proposition:
$$
E_{\delta_n}(\mathcal{S}_\tau)=\int_{\mathcal{S}_\tau}\left(\frac12\sum_{i=0,1,2,3}(e^a_\textbf{i} \hat \nabla \delta_{n})^2+\frac{\delta_n^2}{2}\right)\ud \mu_{\mathcal{S}_\tau}
$$
and it is equivalent to $\int_{\mathcal{S}_\tau} i_{\mathcal{S}_\tau}^\star(\star e^a_0T_{ab})$ (with constant which only depends on the geometric data, $b$ and the Killing form of $e^a_0$).

Finally, the error term is:
$$
\hat \nabla^a \left(e_0^b T_{ab}\right)=\hat \nabla^{(a} e_0^{b)} T_{ab}+\delta_{n} e^a_0\nabla_a \delta_{n}-\frac16\scal_{\hat g} \delta_{n} e^a_0\nabla_a \delta_{n} - b (e^a_0\nabla_a \delta_{n}) \delta_{n-1}(u_n^2+u_nu_{n-1}+u_{n-1}^2).
$$
and can be estimated in absolute value by:
$$
\int_{\mathcal{S}_\tau}|\hat \nabla^a e_0^b T_{ab}|\mu_{\mathcal{S}_\tau}\leq C E_{\delta_n}(\mathcal{S}_\tau)+2\int_{\mathcal{S}_\tau}\delta_{n-1}^2(u_n^4+u_{n-1}^4)\mu_{\mathcal{S}_\tau}.
$$
where $C$ is a positive constant depending on $\sup\left(|\scal_{\hat g}|\right)$, $\sup\left(||\hat \nabla^{(a} e_0^{b)}||\right)$, $\sup\left(|b|\right)$ and the foliation $\mathcal{S}_\tau$.

Using H\"older inequality, proposition \ref{blowup2} for the foliation $\mathcal{S}_\tau$, the non-linearity in the error term is estimated by:
\begin{eqnarray*}
\int_{\mathcal{S}_\tau}b^2\delta_{n-1}^2u_n^4\mu_{\mathcal{S}_\tau}&\leq &\left(\sup_{\Sigma_\tau} b\right)^2\left(\int_{\mathcal{S}_\tau}\delta_{n-1}^6\mu_{\mathcal{S}_\tau}\right)^{\frac13}\left(\int_{\mathcal{S}_\tau}u_n^6\mu_{\mathcal{S}_\tau}\right)^{\frac23}\\
&\leq&c\frac{\left(\sup_{\Sigma_\tau} b\right)^2}{\tau^6}||\delta_n||^2_{H^1(\mathcal{S}_\tau)}||u_n||^4_{H^1(\mathcal{S}_\tau)}\\
&\leq&C||\delta_n||^2_{H^1(\mathcal{S}_\tau)}||u_n||^4_{H^1(\mathcal{S}_\tau)}
\end{eqnarray*}
The same inequality holds for $\delta_{n-1}^2u_{n-1}^4$.

Finally, there exists a constant $K$, such that:
\begin{gather*}
\int_{\mathcal{S}_\tau}|\hat \nabla^a e_0^b T_{ab}|\mu_{\mathcal{S}_\tau}\\
\leq K\left(
\int_0^t E_{\delta_n}(\mathcal{S}_r)\ud r+\epsilon
\left(\sup_{r\in [0,\epsilon]}
\left( E_{\delta_{n-1}}(\mathcal{S}_r)
\right)
\right)
\sup_{k\geq n-1}
\left(\sup_{r\in [0,\epsilon]}
\left( E_{u_n}(\mathcal{S}_r) \right)
\right)^4
\right).
\end{gather*}
Stokes theorem is then applied beween $\mathcal{S}_\tau$ and $\scri^+$: since the characteristic data for $\delta_n$ are zero, the only remaining term is the energy on the surface $\mathcal{S}_\tau$. Modulo a constant which only depends on the same data as the constant $\tilde C$, the integral inequality holds:
$$
E_{\delta_n}(\mathcal{S}_\tau)\leq \tilde K
\left(
\int_0^\tau E_{\delta_n}(\mathcal{S}_r)\ud r+\epsilon
\left(\sup_{r\in [0,\epsilon]}
\left( E_{\delta_{n-1}}(\mathcal{S}_r)
\right)
\right)
\sup_{k\geq n-1}
\left(\sup_{r\in [0,\epsilon]}
\left( E_{u_n}(\mathcal{S}_r) \right)
\right)^4
\right),
$$ 
for some contant $\tilde K$ and, using Gronwall's lemma, we get:
$$
E_{\delta_n}(\mathcal{S}_\tau)\leq \tilde K\exp(\tilde K \epsilon)\epsilon
\left(\sup_{r\in [0,\epsilon]}
\left( E_{\delta_{n-1}}(\mathcal{S}_r)
\right)
\right)
\sup_{k\geq n-1}
\left(\sup_{r\in [0,\epsilon]}
\left( E_{u_n}(\mathcal{S}_r) \right)
\right)^4
.
$$
\begin{remark}\label{remarqueconstant}
The constant $\tilde K$, as the constant $\tilde C$ depends only the foliation $\mathcal{S}_\tau$, its scalar curvature, the Killing form of $e^a_0$ and the supremum of $b$. As a consequence, up to a rescaling of $\tilde C$ or $\tilde K$, these constants can be chosen to be equal.
\end{remark}

Finally, the sequence $(\delta_n)$ satisfies:
$$
\sup_{r\in [0,\epsilon]}\left(
E_{\delta_n}(\mathcal{S}_\tau)\right)
\leq\tilde C\exp(\tilde C \epsilon) \epsilon
\left(\sup_{r\in [0,\epsilon]}
\left( E_{\delta_{n-1}}(\mathcal{S}_r)
\right)
\right)
\sup_{k\geq n-1}
\left(\sup_{r\in [0,\epsilon]}
\left( E_{u_n}(\mathcal{S}_r) \right)
\right)^4
.
$$
Using inequality \eqref{estimatesequence}, we have:
\begin{eqnarray}
\tilde C\exp(\tilde C \epsilon) \epsilon
\sup_{k\geq n-1}
\left(\sup_{r\in [0,\epsilon]}
\left( E_{u_n}(\mathcal{S}_r) \right)
\right)^4
&\leq& \tilde C \epsilon\exp(\tilde C \epsilon)  \left(2 \sqrt{\frac{4}{3\tilde C\exp(\tilde C \epsilon)}}\right)^2\nonumber\\
&\leq& \frac{16}{3} \epsilon\label{choiceepsilon}.
\end{eqnarray}
Since $\epsilon$ is smaller than $\frac{1}{16}$, the sequence $(\delta_n)$ is then eventually contracting. The series of $(\sup_{r\in [0,\epsilon]}\left(||\delta_{n}||_{H^1(\mathcal{S}_r)}\right)^2)_n$ converges in the norm $\left(\sup_{\tau\in[0,\epsilon]}||u||_{H^1(\mathcal{S}_\tau)}^2\right)$, that is to say in\\ $C^0([0,\epsilon], H^1(\mathcal{S}_\tau))$, and so does the sequence $(u_n)$.\end{proof}

\emph{End of the proof of theorem \ref{localcharprob}.} The proof of the local existence is a direct consequence of the fact the sequence $(u_n)_{n\in \N}$ converges strongly on $\mathcal{R}$ for the norm
$$
\sup_{r\in [0,\epsilon]}\left(||u||_{H^1(\mathcal{S}_r)}\right).
$$
Let $u$ be the limit of the sequence $(u_n)_n$. The only remaining thing to show is that the limit solves the problem of theorem \ref{localcharprob}. It is clear that $u$ satisfies the initial conditions since all the functions $u_n$ are identically equal to $\theta$ on $\scri^+$. Finally, when noticing that $u$ is in $H^1(\mathcal{R})$ which is continuously embedded in $L^3(\mathcal{R})$ (since $\mathcal{R}$ is four dimensional), the sequence $(u_n^3)_n$ converges in $L^1(\mathcal{R})$ and, as a consequence, in the distribution sense. $u$ then satisfied the equation 
$$
\hat\square u +\frac16\scal_{\hat g} u + b u^3=0
$$
in the distribution sense.\end{proof}

\bibliographystyle{siam}
\bibliography{scatt.bib}

\end{document}